\documentclass[a4paper, 10pt]{amsart}

\usepackage[utf8]{inputenc}
\usepackage[T1]{fontenc}
\usepackage[english]{babel}
\usepackage{color}
\usepackage{amsmath} 
\usepackage{amssymb} 
\usepackage{amsthm}
\usepackage{graphicx}
\usepackage{hyperref}
\usepackage{cancel}

\usepackage{lmodern}
\usepackage{microtype}

\theoremstyle{plain}
\newtheorem{thm}{Theorem}[section]
\newtheorem*{thm*}{Theorem}
 
\newtheorem{lem}{Lemma}[section] 
\newtheorem{cor}{Corollary}[section]
\newtheorem*{cor*}{Corollary}

\newtheorem{defi}{Definition}[section]

\newtheorem{rem}{Remark}

\newcommand {\R} {\mathbb{R}} \newcommand {\Z} {\mathbb{Z}}
\newcommand {\T} {\mathbb{T}} \newcommand {\N} {\mathbb{N}}
\newcommand {\C} {\mathbb{C}}
\newcommand {\p} {\partial}
\newcommand {\dt} {\partial_t}

\usepackage{tikz}
\usetikzlibrary{shapes,arrows,calc}

\begin{document}

\date{\today}
\author{Christian Zillinger}
 \thanks{This article is based on the author's PhD thesis written under the supervision of Prof. Dr. Herbert Koch, to whom I owe great gratitude for his constant advice, helpful direction and insightful comments. I would also like to thank Jacob Bedrossian for his comments on the first version of this article.}
\thanks{I also thank the CRC 1060 ``On the Mathematics of Emergent Effects'' of the DFG, the Bonn International Graduate School (BIGS) and the Hausdorff Center of Mathematics for their support.}
\address{Mathematisches Institut, Universität Bonn, 53115 Bonn, Germany}
\email{zill@math.uni-bonn.de}

\title{Linear Inviscid Damping for Monotone Shear Flows}

\begin{abstract}
In this article, we prove linear stability, scattering and inviscid damping with optimal decay rates for the linearized 2D Euler equations around a large class of strictly monotone shear flows, $(U(y),0)$, in a periodic channel under Sobolev perturbations.
Here, we consider the settings of both an infinite periodic channel of period $L$, $\T_{L}\times \R$, as well 
 as a finite periodic channel, $\T_{L} \times [0,1]$, with impermeable walls.
The latter setting is shown to not only be technically more challenging, but to exhibit qualitatively different behavior due to boundary effects.
\end{abstract}

\maketitle
\tableofcontents

\section{Introduction}
\label{sec:introduction}
In this article, we study the phenomenon of linear inviscid damping for the 2D incompressible Euler equations in a periodic channel
\begin{align}
  \tag{Euler}
  \begin{split}
    \dt \omega + v \cdot \nabla \omega =0, \\
    v = \nabla^{\bot} \Delta^{-1}\omega,
  \end{split}
\end{align}
written here in vorticity formulation, linearized around strictly monotone shear flow solutions 
\begin{align*}
  v&=(U(y),0), \\
  \omega&= -U''(y).
\end{align*}

\begin{figure*}[h]
  \centering
  \includegraphics[page=5]{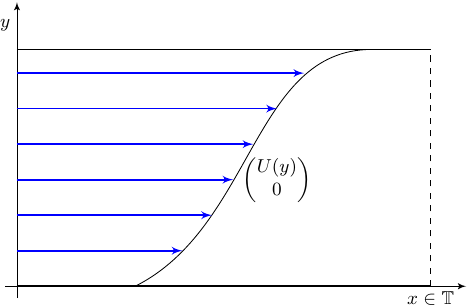}
  \caption{An example of a shear flow, $(U(y),0)$, in a periodic channel.}
\end{figure*}

Here, we consider both the common setting of an infinite periodic channel of period $L$, $(x,y) \in \T_{L} \times \R$, as well as the physically relevant setting of a finite periodic channel, $\T_{L} \times [0,1]$, with impermeable walls, i.e. in that setting we require
\begin{align*}
  v_{2} =0, \quad \text{ for } y \in \{0,1\}.
\end{align*}
The latter setting is shown to be not only technically more challenging but to exhibit qualitatively different behavior due to boundary effects.
\\

\subsection{Motivation and literature}
\label{sec:motiv-liter}

The motivating example for the flows we consider is given by Couette flow, i.e. the linear shear $U(y)=y$, in an infinite periodic channel, $\T_{L} \times \R$. For this specific flow the linearized Euler equations simplify to the free transport equations:
\begin{align*}
  \dt \omega + y \p_{x}\omega &=0,\\
  (t,x,y) &\in \R \times \T_{L} \times \R,
\end{align*}
and, in particular, can be solved explicitly in both spatial and Fourier variables.
In this case, one can hence directly compute that perturbations $(\omega,v)$ are damped to a shear flow with algebraic rates:
\begin{align}
  \label{eq:optdampCouette}
  \begin{split}
  \|v_{1}(t)- \langle v_{1} \rangle_{x}\|_{L^{2}} &\leq \mathcal{O}(t^{-1})\|\omega_{0}- \langle \omega_{0} \rangle_{x}\|_{H^{-1}_{x}H^{1}_{y}}, \\
  \|v_{2}(t)\|_{L^{2}}&\leq \mathcal{O}(t^{-2})\|\omega_{0}- \langle \omega_{0} \rangle_{x}\|_{H^{-1}_{x}H^{2}_{y}},
  \end{split}
\end{align}
and that the decay rates and regularity requirements are sharp.
This classical and, in view of the Hamiltonian structure of the Euler equations (c.f. \cite{Arn1}), at first surprising result was experimentally observed and proven for the linearized equations by Kelvin, \cite{kelvin1887stability}, and Orr, \cite{orr1907stability}, and is called (linear) inviscid damping and shares similarities with Landau damping in plasma physics.
\\

Going beyond the explicitly solvable (and in this sense trivial) setting of linearized Couette flow, has, however, remained open until recently: 
\begin{itemize}
\item In \cite{Euler_stability}, Bouchet and Morita give heuristic results suggesting that linear damping and stability results should also hold for general monotone shear flows. However, their methods are non-rigorous and lack necessary regularity, stability and error estimates, as discussed in \cite{Zill}.
In particular, even supposing their asymptotic computations were valid, they do not yield the above decay rates.
\item Lin and Zeng, \cite{Lin-Zeng}, use the explicit solution of linearized Couette flow to establish linear damping also in a finite periodic channel.
Furthermore, they show the existence of non-trivial stationary solutions to the full 2D Euler equations in arbitrarily small $H^{s}$ neighborhoods of Couette flow for any $s<\frac{3}{2}$. As a consequence, nonlinear inviscid damping can not hold in such low regularity. 
\item Recently, following the work of Villani and Mouhot, \cite{Villani_long}, on nonlinear Landau damping, Masmoudi and Bedrossian, \cite{bedrossian2013inviscid}, have proven nonlinear inviscid damping for small Gevrey perturbations to Couette flow in an infinite periodic channel.
\end{itemize}

\subsection{Strategy and main results}
\label{sec:strat-main-results}

As the main results of this article, we, for the first time, rigorously prove linear inviscid damping for a general class of monotone shear flows.
Here, in addition to the common setting of an infinite periodic channel of period $L$, $\T_{L}\times \R$, we also prove linear inviscid damping in the physically relevant setting of a finite periodic channel, $\T_{L}\times [0,1]$, with impermeable walls.
As we show in Section \ref{sec:asympt-stab-with}, in the latter case, boundary effects asymptotically lead to the formation of (logarithmic) singularities of derivatives of solutions. Stability results are thus limited to fractional Sobolev spaces $H^{s},s\leq \frac{3}{2}$, unless one restricts to perturbations, $\omega_{0}$, with vanishing Dirichlet data, $\omega_{0}|_{y=0,1}$.
As damping with the optimal algebraic rates, \eqref{eq:optdampCouette}, only requires stability in $H^{2}$, in this article we limit ourselves to establishing stability in $H^{1}$ for general perturbations and stability in $H^{2}$ for perturbations with vanishing Dirichlet data, $\omega_{0}|_{y=0,1}=0$.

In a follow-up article, we show that the fractional Sobolev spaces are indeed critical and establish stability in all sub-critical spaces and blow-up in all super-critical spaces.
There, we also further discuss boundary effects, the associated asymptotic singularity formations for derivatives of $W$ and implications for the instability of the nonlinear dynamics and the problem of nonlinear inviscid damping in a finite periodic channel, where very high regularity would be essential to control nonlinear effects (see \cite{bedrossian2013inviscid}).   
\\

Our strategy to prove linear inviscid damping is described in Figure \ref{fig:strategy}. 
\begin{figure*}[h]
\label{fig:strategy}
  \centering
  \includegraphics[page=1]{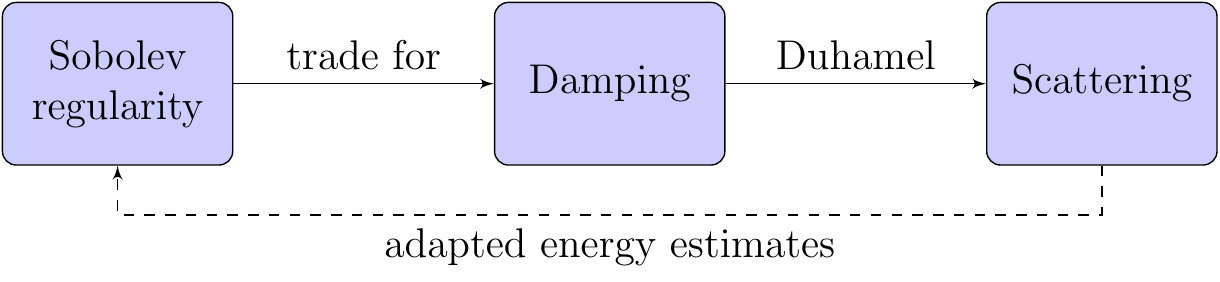}
  \caption{Sobolev regularity can be used to obtain damping and scattering in less regular spaces. However, in order to close the argument any loss of regularity has to be avoided.}
\end{figure*}

As a first step, in Section \ref{sec:damping}, we show that linear inviscid damping, like Landau damping, is fundamentally a problem of regularity.
For this purpose, we consider the linearized 2D Euler equations
\begin{align}
\label{eq:linEueq}
  \begin{split}
  \dt \omega + U(y)\p_{x}\omega &= U''(y) v_{2}= U''(y) \p_{x} \phi, \\
  \Delta \phi &= \omega,
  \end{split}
\end{align}
as a perturbation to the underlying transport equation
\begin{align*}
  \dt f + U(y) \p_{x} f =0,
\end{align*}
and consider coordinates moving with the flow:
\begin{align*}
  W(t,x,y) :=\omega(t,x-tU(y),y). 
\end{align*}
In analogy to conventions in dispersive equations we call $W$ the \emph{scattered vorticity} (with respect to the underlying flow).

\emph{Assuming} $W$ to be regular uniformly in time, i.e.
\begin{align*}
  \|W(t)\|_{L^{2}_{x}H^{2}_{y}}< C< \infty,
\end{align*}
it has been shown in the author's Master's thesis, \cite{Zill}, that damping estimates of the form \eqref{eq:optdampCouette} can be extended to general, strictly monotone shear flows (Theorem \ref{thm:lin-zeng}):
\begin{thm*}[Generalization of {\cite[Theorem 3]{Lin-Zeng}}; \cite{Zill}]
Let $\Omega$ be either the infinite periodic channel, $\T_{L}\times \R$, or the finite periodic channel, $\T_{L}\times [0,1]$.
  Let $\omega$ be a solution to the linearized Euler equations, \eqref{eq:linEueq}, around a strictly monotone shear flow $U(y)$, on the domain $\Omega$.
 Suppose further that $\frac{1}{U'}\in W^{2,\infty}(\Omega)$. Then the following statements hold:
  \begin{enumerate}
  \item If $W(t)\in H^{-1}_{x}H^{1}_{y}(\Omega)$ for all times, then
    \begin{align*}
      \|v(t)- \langle v \rangle_{x}\|_{L^{2}(\Omega)}=\mathcal{O}(t^{-1}) \|W(t)- \langle W \rangle_{x}\|_{H^{-1}_{x}H^{1}_{y}(\Omega)}, \text{ as } t\rightarrow \pm \infty.
  \end{align*}
  \item If $W(t)\in H^{-1}_{x}H^{2}_{y}(\Omega)$ for all times, then
    \begin{align*}
      \|v_{2}(t)\|_{L^{2}(\Omega)}=\mathcal{O}(t^{-2})\|W(t)- \langle W \rangle_{x}\|_{H^{-1}_{x}H^{2}_{y}(\Omega)}, \text{ as } t\rightarrow \pm \infty.
    \end{align*}
  \end{enumerate}
\end{thm*}
As a consequence of sufficiently rapid decay of $v_{2}$, we observe that the right-hand-side $U''v_{2}$ in \eqref{eq:linEueq} is an integrable perturbation:
\begin{cor*}[Scattering]
  Let $U'' \in L^{\infty}$ and let $W(t) \in L^{2}$ be a solution of \eqref{eq:linEueq} such that, for some $\epsilon>0$, 
  \begin{align*}
    \|v_{2}(t)\|_{L^{2}}= \mathcal{O}(t^{-1-\epsilon}),
  \end{align*}
  as $t \rightarrow \infty$.
  Then, there exists $W^{\infty}\in L^{2}$ such that 
  \begin{align*}
    W(t) \xrightarrow{L^{2}} W^{\infty},
  \end{align*}
  as $t \rightarrow \infty$ and 
  \begin{align*}
    \|W(t)-W^{\infty}\|_{L^{2}} =\mathcal{O}(t^{-\epsilon}).
  \end{align*}
\end{cor*}
In Section \ref{sec:consistency-outlook}, this scattering result is further extended to arbitrary $L^{2}$ initial data.
We stress that the \emph{higher regularity} of $W$ is necessary in Theorem \ref{thm:lin-zeng}, as the underlying shear, 
\begin{align*}
  e^{tS}:(x,y) \mapsto (x-tU(y),y),
\end{align*}
is an $L^{2}$-isometry and hence
\begin{align*}
  W= e^{-tS}\omega \mapsto V:= e^{-tS}v= e^{-tS}\nabla^{\bot}\Delta e^{tS} W,
\end{align*}
when considered as an operator from $L^{2}$ to $L^{2}$, has a time-independent operator norm.
In order to prove the desired stability result for $W$, 
\begin{align*}
  \|W(t)- \langle W \rangle_{x}\|_{H^{-1}_{x}H^{2}_{y}} \lesssim \|\omega_{0}- \langle \omega_{0} \rangle_{x}\|_{H^{-1}_{x}H^{2}_{y}},
\end{align*}
we thus have to invest considerable technical effort to use finer properties of the dynamics.

As the first main result of this article, in Section \ref{sec:asympt-stab-y}, we establish stability of the linearized Euler equations around regular, strictly monotone shear flows in an infinite periodic channel, $\T \times \R$, for arbitrarily high Sobolev norms (Theorem \ref{thm:Sobolevstab}):
\begin{thm}[Sobolev stability for the infinite periodic channel]
  Let $s \in \N_{0}$ and let $U'(U^{-1}(\cdot)), U''(U^{-1}(\cdot)) \in W^{s+1,\infty}(\R)$ and suppose that there exists $c>0$, such that
  \begin{align*}
    0<c<U'<c^{-1}<\infty .
  \end{align*}
  Suppose further that 
  \begin{align*}
     L \|U''(U^{-1}(\cdot))\|_{W^{s+1,\infty}} 
  \end{align*}
  is sufficiently small.
  Then for all $m \in \Z$ and $\omega_{0} \in H^{m}_{x}H^{s}_{y}(\T_{L}\times \R)$, the solution W of the linearized Euler equations in scattering formulation, \eqref{eq:LEscat}, with initial datum $\omega_{0}$ satisfies 
  \begin{align*}
    \|W(t)\|_{H^{m}_{x}H^{s}_{y}(\T_{L}\times\R)} \lesssim \|\omega_{0} \|_{ H^{m}_{x}H^{s}_{y}(\T_{L}\times\R)}.
  \end{align*}
\end{thm}
The proof of this theorem is broken down into several steps, which form the subsections of Section \ref{sec:asympt-stab-y}.
\\

When considering a finite channel instead, we show that such a stability result can not hold in arbitrary Sobolev spaces, but rather that in general boundary derivatives of $W$ asymptotically develop (logarithmic) singularities at the boundary. Thus, stability in Sobolev spaces $H^{s}, s>\frac{3}{2}$, is not possible, unless one restricts to perturbations, $\omega_{0}$, with vanishing Dirichlet data, $\omega_{0}|_{y=0,1}$.
The stability result in $H^{2}$ under such perturbations is then given by Theorem \ref{thm:finsum}:
\begin{thm}[$H^{2}$ stability for the finite periodic channel]
  Let $W$ be a solution of the linearized Euler equations in scattering formulation \eqref{eq:2}.
  Let further $U'(U^{-1}(\cdot)), U'' (U^{-1}()\cdot) \in W^{3,\infty}([0,1])$ and suppose that there exists $c>0$ such that
 \begin{align*}
   0<c<U'<c^{-1}<\infty.
 \end{align*}
 Suppose further that 
 \begin{align*}
   \|U'' (U^{-1}(\cdot))\|_{W^{1,\infty}} L 
 \end{align*}
 is sufficiently small.

 Then, for any $m \in \Z$  and any $\omega_{0} \in H^{m}_{x}H^{2}_{y}(\T_{L}\times [0,1])$ with $\omega_{0}|_{y=0,1}=0$ and for any time $t$,
\begin{align*}
\|W(t)\|_{H^{m}_{x}H^{2}_{y}(\T_{L}\times [0,1])} \lesssim \|\omega_{0}\|_{H^{m}_{x}H^{2}_{y}(\T_{L}\times [0,1])}.
\end{align*}
\end{thm}

In Section \ref{sec:consistency-outlook}, we combine the stability and damping results to prove linear inviscid damping with the optimal rates in both an infinite periodic channel and a finite periodic channel in Theorem  \ref{thm:summary}:
\begin{thm*}[Linear inviscid damping for monotone shear flows]
  Let $\omega$ be a solution to the linearized Euler equation
  \begin{align*}
    \dt \omega + U(y)\p_{x} \omega = U'' v_{2}, 
  \end{align*}
  with initial data $\omega_{0} \in H^{2}_{y}L^{2}_{x}$ on either the infinite channel, $\T_{L}\times \R$, or on the
  finite channel, $\T_{L} \times [0,1]$, where $\T_{L}=[0,L] /\sim$.

  Suppose there exists $c >0$  such that
  \begin{align*}
    c < |U'| <c^{-1},
  \end{align*}
  and that $U''$ and $L$ are such that  
  \begin{align*}
    L \left\| U''(U^{-1}(\cdot))\right\|_{W^{3,\infty}}
  \end{align*}
  is sufficiently small.
  In the case of a finite channel, additionally assume that $\omega_{0}$ vanishes on the boundary
  \begin{align*}
\omega_{0}(x,0)\equiv 0 \equiv \omega_{0}(x,1).
  \end{align*}
Then there exist asymptotic profiles $W^{\infty}(x,y)$ and $v^{\infty}(y)$ such that
  \begin{align*}
    \tag{Stability}
    \|\omega(t,x-tU(y),y)\|_{L^{2}_{x}H^{2}_{y}} &\lesssim \|\omega_{0}\|_{L^{2}_{x}H^{2}_{y}}, \\
      \tag{Damping}
    \|v(t)-(v^{\infty},0)\|_{L^{2}_{xy}} &= \mathcal{O}(t^{-1})\|\omega_{0}- \langle \omega_{0} \rangle_{x}\|_{\dot H^{-1}_{x}H^{1}_{y}} , \\
    \|v_{2}(t)\|_{L^{2}}&= \mathcal{O}(t^{-2})\|\omega_{0}- \langle \omega_{0} \rangle_{x}\|_{\dot H^{-1}_{x}H^{2}_{y}}, \\
    \tag{Scattering}
   \omega(t,x-tU(y),y) &\xrightarrow{L^{2}} W^{\infty}, \\
    \|\omega(t,x-tU(y),y)-W^{\infty}\|_{L^{2}_{xy}}&= \mathcal{O}(t^{-1})\|\omega_{0}\|_{\dot H^{-1}_{x}H^{2}_{y}},
  \end{align*}
as $t \rightarrow \pm \infty$.
\end{thm*}

As a consequence of Theorem \ref{thm:summary} and the stability results of Section \ref{sec:reduct-gener-monot} and Section \ref{sec:l2-stability-via-1}, we also obtain scattering for general $L^{2}$ initial data in Corollary \ref{cor:L2scatterin}:

\begin{cor*}[$L^{2}$ scattering]
  Let $U,L$ be as in Theorem \ref{thm:summary} and let $\omega_{0} \in L^{2}$, then there exists $W_{\infty} \in L^{2}$ such that 
  \begin{align*}
    W(t,x,y) \xrightarrow {L^{2}} W_{\infty}, \text { as } t \rightarrow \infty.
  \end{align*}
\end{cor*}

\subsection{Outline of the article}
\label{sec:outline-article}

We conclude this introduction with a short overview of the organization of the article:
\begin{itemize}
\item In Section \ref{sec:motiv-exampl-couette}, we consider linearized Couette flow on the infinite periodic channel, $\T \times \R$, as a motivating example, which allows explicit solutions in physical as well as Fourier space.
In particular, the damping mechanism and the regularity requirements are most transparent in this setting.
\item In Section \ref{sec:damping}, the damping results are generalized to smooth strictly monotone shear flows, under the \emph{assumption} of controlling the Sobolev regularity of the perturbation $W(t,x,y):=\omega(t,x-tU(y),y)$.
Linear inviscid damping is hence shown to fundamentally be a problem of regularity and stability, as is also the case for Landau damping.
This section is in part based on the author's  Master's thesis, \cite{Zill}, and generalizes previous results by \cite{Lin-Zeng} and \cite{Euler_stability}.

\item In Section \ref{sec:asympt-stab-y}, we establish stability for the case of an infinite periodic channel, $\T_{L}\times \R$, in any Sobolev norm $H^{s}_{y}H^{m}_{x}(\R \times \T_{L}), s,m \in \N_{0}$, provided $L \|U''(U^{-1}(\cdot))\|_{W^{s,\infty}(\R)}$ is sufficiently small. In particular, instead of imposing assumptions on the period $L$, we could restrict to shear flows that are close to affine.

\item In Section \ref{sec:asympt-stab-with}, we treat the case of a finite periodic channel, $\T_{L} \times [0,1]$, with impermeable walls.
Here, we show that boundary effects can not be neglected and that for perturbations, $\omega_{0}$, with non-trivial Dirichlet data, $\omega_{0}$, $\p_{y}W$ asymptotically develops logarithmic singularities at the boundary.
While $H^{1}$ stability results can be established for general perturbations, the $H^{2}$ stability results hence necessarily have to restrict to perturbations with vanishing Dirichlet data, $\omega_{0}|_{y=0,1}=0$. 

In a follow-up article, the singularity formation is studied in more detail. In particular, we prove that the fractional Sobolev space $H^{\frac{3}{2}}$ is critical, in the sense that stability holds in all sub-critical spaces and blow-up occurs in all super-critical spaces.
Furthermore, even restricting to perturbations with vanishing Dirichlet data, $\omega_{0}|_{y=0,1}$, the critical space is only improved to $H^{\frac{5}{2}}$.
As a consequence, we show that the nonlinear Euler equations in a finite channel can not stay regular in high Sobolev norms and that thus results in Gevrey regularity such the ones by Bedrossian and Masmoudi, \cite{bedrossian2013asymptotic}, can not hold in this setting.

\item In the final Section \ref{sec:consistency-outlook}, we conclude our proof of linear inviscid damping with optimal decay rates for monotone shear flows in an infinite periodic channel and a finite periodic channel.
In the case of an infinite periodic channel, we also discuss consistency with the nonlinear equation, following an argument of \cite{Euler_stability}.
The case of a finite periodic channel and the implications of the boundary effects and the associated singularity formations are considered in a follow-up article.
\end{itemize}

\section{Couette flow}
\label{sec:motiv-exampl-couette}

The linearized Euler equations around Couette flow, $U(y)=y$, in an infinite periodic channel, $\T \times \R$, are given by
\begin{align*}
  \label{eq:1}
  \dt \omega + y \p_{x} \omega &=0,\\
  v&= \nabla^{\bot}\Delta^{-1}\omega, \\
  (t,x,v) & \in \R \times \T \times \R. 
\end{align*}
We note that the first equation is (up to a change of notation) identical to free transport.
The equations can hence be explicitly solved using the method of
characteristics:
\begin{align*}
  \omega(t,x,y) = \omega_{0}(x-ty,y).
\end{align*}
As an example of the behavior of solutions, consider the case $\omega_{0}$ being the indicator function of a box, depicted in figure \ref{fig:mixin}. 
\begin{figure}[h]
  \centering

  \includegraphics[width=0.45\linewidth, page=2]{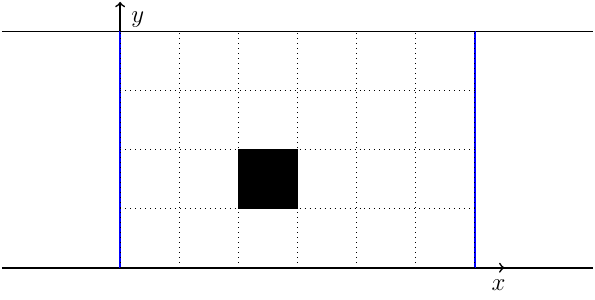}
  \includegraphics[width=0.45\linewidth, page=3]{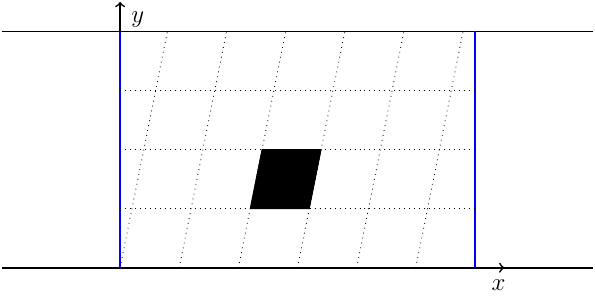}
  \includegraphics[page=4]{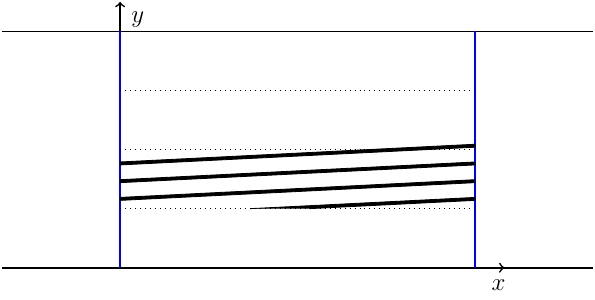}
  \caption{The vorticity is sheared by the flow.}
  \label{fig:mixin}
\end{figure}

From this one observes two opposing behaviors:
\begin{itemize}
\item Rapid oscillations in $y$ damp anti-derivatives such as the velocity $v= \nabla^{\bot}\Delta \omega$ towards averaged quantities with a rate depending on the regularity of the initial data $\omega_{0}$.      
\item The evolution loses regularity in $y$ as time increases. For this reason the mechanism is sometimes called ``violent relaxation'' \cite{Villani_long}.
\end{itemize}

Due to the distinct role of the average, we briefly pause to discuss its behavior.  
The average in $x$ is a function of $y$ and $t$ only and satisfies:
\begin{align*}
  \dt \langle \omega \rangle_{x} + y \langle \p_{x} \omega \rangle_{x} =0.
\end{align*}
By periodicity $\langle \p_{x} w \rangle_{x}=0$, and thus
\begin{align*}
   \langle \omega \rangle_{x} (y)= \langle \omega_{0} \rangle_{x} (y)
\end{align*}
is conserved by the evolution.

Incorporating the average of the initial perturbation into the underlying shear flow $U \mapsto U(y)+\int^{y}\langle \omega_{0} \rangle_{x} dy'$ or using the linearity of the equation, we may thus without loss of generality assume that our perturbation satisfies
\begin{align*}
  \langle \omega \rangle_{x} =\langle \omega_{0} \rangle_{x} \equiv 0.
\end{align*}

\begin{rem}
  The same reduction can be used for the linearized equation for general shear flows $U(y)$, as 
  \begin{align*}
    \langle U''(y)v_{2} \rangle_{x}= U''(y) \langle \p_{x} \phi \rangle_{x} =0.
  \end{align*}
  In the nonlinear setting one would also like to remove this average, however it is not conserved anymore. 
  Therefore, one has to scatter around a shear profile changing in time, which introduces considerable technical difficulties (see \cite{bedrossian2013inviscid}).
\end{rem}

With the average set to zero, the above heuristic example suggests that positive Sobolev norms in $y$ blow up as $t \rightarrow \pm \infty$ while negative Sobolev norms tend to zero.

In order to obtain a more quantitative description, it is useful to restrict to the whole space setting $\T \times \R$, where a Fourier transform is available.
After a Fourier transform in $x$ and $y$, which in the sequel is denoted by $\tilde \cdot$ , our equation is given by
\begin{align*}
  \dt \tilde \omega + k\p_{\eta} \tilde \omega &=0, \\
  \tilde v &=
  \begin{pmatrix}
 -i\eta \\ ik    
  \end{pmatrix}
\frac{1}{k^{2}+\eta^{2}} \tilde \omega.
\end{align*}
So we again obtain a transport equation, which we may solve using the method of characteristics:
\begin{align*}
  \tilde \omega(t,k,\eta)= \tilde \omega_{0} (k,\eta+kt).
\end{align*}

\begin{figure}[h]
  \centering
  \includegraphics[width=0.8\linewidth, page=6]{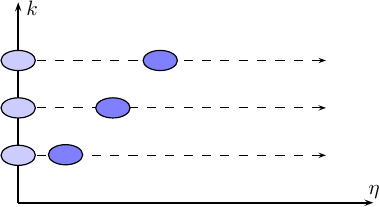}
  \caption{Transport in Fourier space.}
  \label{fig:fourier}
\end{figure}

\begin{rem}
  When considering linear Landau damping, to compute the force field one is only interested in an average of the density, which in our notation would be the case $\eta=0$. In that case, high regularity of $\omega_{0}$ directly translates into a high decay speed of $\tilde \omega_{0}(k,kt)$. In particular, analytic regularity allows one to deduce exponential decay, \cite{Villani_script} \cite{Villani_long}, \cite{Villani_short}.
In the case of the Euler equations however, the velocity field depends on all $\eta$ and a main difficulty arises in the control of  $\eta \approx -kt$.  
\end{rem}

\begin{rem}
  \begin{itemize}
  \item While neither Rayleigh's nor Fj\o rtoft's theorems are
    applicable, since $(y)''=0$, Couette flow is linearly stable in $L^{p}$ for all $p \in [0,\infty]$ as the above change of variables is a volume preserving diffeomorphism.
\item $L^{p}$ stability remains true in the nonlinear setting, as the conservation of
    \begin{align*}
      \int f (\Omega) dx
    \end{align*}
    for $\Omega = 1 + \omega$ also implies the conservation of
    $\|\omega\|_{L^{p}}$ by choosing $f(t)=|t-1|^{p}$. For further
    results in this direction confer \cite{romanov1973stability},
    \cite{bourgain2013}, \cite{Lin-Zeng}, \cite{li2011resolution}.

  \item Despite its simplicity Couette flow has been of significant
    interest in physical research due to being (up to symmetries) the only shear flow
    that also is a stationary solution of the Navier-Stokes equation as well as
    being a sufficient model for naturally occurring flows in pipes,
    channels or similar simple geometries. 

    It is also frequently studied in the context of turbulence (for an introduction see
    \cite{drazin2}).  Associated with it, is the so-called
    ``Sommerfeld paradox'': The flow is a linearly stable solution of
    the Navier-Stokes equation for all Reynold's numbers $Re>0$, but as seen by numerical
    and physical experiments it becomes turbulent when $Re$ is
    large.
  \end{itemize}
\end{rem}

While all $L^{p}$ norms are conserved, Sobolev norms involving $y$ change in time.  Using the characterization in Fourier variables and the explicit solution we compute
\begin{align*}
  \|\omega\|_{H^{s_{1}}_{x}H^{s_{2}}_{y}}^{2}&= \sum \int <k>^{2s_{1}}<\eta>^{2s_{2}} |\omega_{0}(k,\eta+kt)|^{2} dk d\eta \\
 &= \sum \int <k>^{2s_{1}}<\eta-kt>^{2s_{2}} |\omega_{0}(k,\eta)|^{2} dk d\eta
\end{align*}

We thus heuristically observe that $\|\omega\|_{H^{s}_{y}} \sim <t>^{s}$, i.e. positive
Sobolev norms grow as $t$ increases, while negative Sobolev norms tend
to zero.

  However, these estimates are only asymptotic and not
uniform. Consider for example an initial datum $ \omega_{0}$ highly
concentrated at $(k_{0}, -k_{0} c)$ for some $k_{0}$ and $c \gg 0$.
The vorticity $\omega$ will then in turn be concentrated at $(k_{0}, k_{0}(t-c))$,
which in particular implies that for $0<t<c$ any negative Sobolev norm
of $\omega$ is in fact \emph{increasing} and, even-though it
is decreasing for $t>c$, it will only be small for $t \gg 2c$.

 Therefore, to obtain uniform estimates, $L^{2}$ control of $\omega_{0}$ is not sufficient as it is
invariant under translation in Fourier space and it is necessary to invest additional regularity to penalize Fourier modes with 
\begin{align*}
\eta \approx kt .
\end{align*}

A more precise theorem concerning the decay properties of the
velocity field depending on the regularity of the initial datum is
given by Lin and Zeng.
\begin{thm}[Damping for Couette flow {\cite[Theorem 3]{Lin-Zeng}}]
\label{thm:linzengCouette}
  Let $\omega_{0}$ be initial data such that $\int \omega_{0}(x,y)dx =0$ and let $\omega,v$ be the
  corresponding solution. Then,
  \begin{enumerate}
  \item if $\omega_{0}\in H^{-1}_{x}L^{2}_{y}$, then
    $\|v\|_{L^{2}}\rightarrow 0$,
  \item if $\omega_{0}\in H^{-1}_{x}H^{1}_{y}$, then
    $\|v\|_{L^{2}}=\mathcal{O}(t^{-1})$,
  \item if $\omega_{0}\in H^{-1}_{x}H^{2}_{y}$, then
    $\|v_{2}\|_{L^{2}}=\mathcal{O}(t^{-2})$.
  \end{enumerate}
\end{thm}
\begin{rem}
  The original proof of Lin and Zeng also handles the case of a
  bounded domain and is generalized in section \ref{sec:damping} to general monotone shear flows. For
  this section, we however prefer a more direct and easier proof
  using explicit calculations in Fourier variables.
\end{rem}

\begin{proof}
  Define $V_{2}(t,x,y)=v_{2}(t,x-ty,y)$. Then $\|V_{2}\|_{L^{2}}=\|v_{2}\|_{L^{2}}$ and it is thus sufficient to consider $V_{2}$. By the previous calculations, 
  \begin{align*}
\tilde V_{2} =
  \frac{-k^{2}}{k^{2}+(\eta-kt)^{2}} \frac{\tilde\omega_{0}}{ik}.
  \end{align*}
  We note that the first factor is uniformly bounded by $1$ and converges
  point-wise to $0$ and that $\frac{\omega_{0}}{ik} \in L^{2}$, since by assumption $\omega_{0}\in H^{-1}L^{2}$.
  Splitting
  \begin{align*}
\tilde V_{2} = \frac{-k^{2}}{k^{2}+(\eta-kt)^{2}}  (1_{|\eta|\leq R}\frac{\tilde\omega_{0}}{ik}) 
+  \frac{-k^{2}}{k^{2}+(\eta-kt)^{2}}  (1_{|\eta|> R}\frac{\tilde\omega_{0}}{ik}),
  \end{align*}
for  $R$ sufficiently large the $L^{2}$ norm of the second term is
  smaller than $\epsilon$, while for fixed $R$ the multiplier is supported in the compact set $B_{R}$ and decays to zero uniformly as $t \rightarrow \pm \infty$.
Taking an appropriate diagonal sequence in (R,t) then yields the first result for $v_{2}$.

  For $V_{1}(t,x,y):=v_{1}(t,x-ty,y)$ we proceed analogously with
  \begin{equation*}
    \tilde V_{1} = \frac{-(\eta-kt)k}{k^{2}+(\eta-kt)^{2}}\frac{\tilde\omega_{0}}{ik}.
  \end{equation*}

  In order to show the other two claims, we multiply by a factor
  $1=\frac{\max(1,|\eta|^{j})}{\max(1,|\eta|^{j})}, j=1,2$, and split
  it as follows:
  \begin{equation*}
    \tilde V_{2} =  \frac{-k^{2}}{(k^{2}+(\eta-kt)^{2})\max(1,|\eta|^{j})} \frac{\max(1,|\eta|^{j})\tilde\omega_{0}}{ik}.
  \end{equation*}
  Note that the first factor is still uniformly bounded, but in
  addition decays uniformly like $t^{-j}$ in the case of $V_{2}$.
  When considering $V_{1}$ only a decay of $t^{-1}$ may be obtained in
  this way, since 
  \begin{align*}
    \frac{-(\eta-kt)}{k^{2}+(\eta-kt)^{2}}
  \end{align*}
  only decays with rate $t^{-1}$.
\end{proof}

\begin{rem}
  \begin{itemize}
  \item The decay speed depends on the regularity of $\omega_{0}$ and
    can be seen to be sharp in the sense that for each fixed $t$ one
    can find a worst case $\omega_{0}$ such that the multiplier is of size $1$.
  \item In contrast to the Vlasov-Poisson equation assuming regularity
    higher than $H^{2}$ on $\omega_{0}$ does not give any additional
    decay speed for $\|v_{2}\|_{L^{2}}$.
  \item Lin and Zeng prove this theorem by integrating by parts and
    testing the equation.  In the simple setting of Couette flow in the whole space
    this is not necessary as we may calculate explicitly in Fourier
    variables. However, their method may be generalized to other shear
    flows where Fourier methods are not available, as we will explore
    in the following section.
  \item Lin and Zeng in addition give interpolated inequalities for
    $\omega_{0}\in H^{s_{1}}H^{s_{2}}$. Details can be found in
    \cite{Lin-Zeng}.
  \end{itemize}
\end{rem}

\section{Damping under regularity assumptions}
\label{sec:damping}
In the following, we extend the damping result for Couette flow of Section \ref{sec:motiv-exampl-couette} to more general shear flows $(U(y),0)$.
Here, we consider the settings of an infinite channel of period $L$, $\T_{L} \times \R$, as well as of a finite periodic channel, $\T_{L}\times [0,1]$.
In both settings the linearized Euler equations around a shear flow $(U(y),0)$ are given by:
\begin{align}
  \label{eq:LEaroundU}
  \begin{split}
  \dt \omega + U(y)\p_{x} \omega &= U'' v_{2}, \\
  v_{2}&= \p_{x} \phi, \\
\Delta \phi &=\omega,
  \end{split}
\end{align}
where for the infinite channel, the velocity field $v=\nabla^{\bot}\phi$ is required to be integrable, i.e. 
\begin{align*}
  \phi \in \dot H^{1}(\T_{L}\times \R),
\end{align*}
and in the case of a finite channel, $\T_{L}\times [0,1]$, we consider impermeable walls, i.e. we additionally require
\begin{align*}
  v_{2}=0 \text{ for } y \in \{0,1\}.
\end{align*}

In view of the damping results of Section \ref{sec:motiv-exampl-couette}, we consider the right-hand-side, $U''v_{2}$, to be a perturbation and introduce the \emph{scattered vorticity}
\begin{align}
  \label{eq:scatteredvorticity}
  W(t,x,y):= \omega(t,x-tU(y),y).
\end{align}
As for Couette flow, taking the $x$ average of the equation, we see that
\begin{align}
\label{eq:xaverage}
  \langle W \rangle_{x} (t,y)= \langle \omega  \rangle_{x}(t,y)= \langle \omega_{0} \rangle_{x}(y)
\end{align}
is independent of time.
By linearity and writing
\begin{align*}
\omega_{0}=(\omega_{0}- \langle \omega_{0} \rangle_{x})+ \langle \omega_{0} \rangle_{x},
\end{align*}
in the following without loss of generality we only consider the case $\langle W \rangle_{x} \equiv 0$.

The results of Section \ref{sec:motiv-exampl-couette} for Couette flow show that regularity of $W$ is needed to establish damping results for the velocity field.
In this section, we \emph{assume} $W$ to be of regularity comparable to $\omega_{0}$ also in high Sobolev norms, uniformly in time.

The proof of stability of $W$ and hence control of 
\begin{align*}
  \|W(t)\|_{L^{2}_{x}H^{2}_{y}},
\end{align*}
which is the main result of this article, is obtained in Sections \ref{sec:asympt-stab-y} and \ref{sec:asympt-stab-with} for the setting of an infinite periodic channel and a finite periodic channel, respectively.

Using the regularity, we establish damping results with the same optimal algebraic rates as for Couette flow also for general, strictly monotone shear flows, where the bounds are now in terms of $W$ instead of $\omega_{0}$.
In Section \ref{sec:damp-scatt-under}, these results are further generalized and reformulated in terms of the respective flow maps.

As we consider general shear flows and also the setting of a finite periodic channel, Fourier methods are not available anymore. We therefore obtain results by duality in analogy to classical stationary phase arguments and as an extension of \cite{Lin-Zeng} and \cite[Appendix A.1]{Euler_stability}.


\begin{thm}[Generalization of {\cite[Theorem 3]{Lin-Zeng}}; \cite{Zill}]
\label{thm:lin-zeng}
Let $\Omega$ be either the infinite periodic channel, $\T_{L}\times \R$, or the finite periodic channel, $\T_{L}\times [0,1]$.
  Let $\omega$ be a solution to the linearized Euler equations, \eqref{eq:LEaroundU}, around a strictly monotone shear flow $U(y)$, on the domain $\Omega$.
 Suppose further that the initial datum, $\omega_{0}$, satisfies $\langle \omega_{0} \rangle_{x}=0$ and that $\frac{1}{U'}\in W^{2,\infty}(\Omega)$. Then the following statements hold:
  \begin{enumerate}
  \item If $W(t)\in H^{-1}_{x}H^{1}_{y}(\Omega)$ for all times, then
    \begin{align*}
      \|v(t)- \langle v \rangle_{x}\|_{L^{2}(\Omega)}=\mathcal{O}(t^{-1}) \|W(t)\|_{H^{-1}_{x}H^{1}_{y}(\Omega)}, \text{ as } t\rightarrow \pm \infty.
  \end{align*}
  \item If $W(t)\in H^{-1}_{x}H^{2}_{y}(\Omega)$ for all times, then
    \begin{align*}
      \|v_{2}(t)\|_{L^{2}(\Omega)}=\mathcal{O}(t^{-2})\|W(t)\|_{H^{-1}_{x}H^{2}_{y}(\Omega)}, \text{ as } t\rightarrow \pm \infty.
    \end{align*}
  \end{enumerate}
\end{thm}

\begin{proof}
The results are established by testing. More precisely, in the infinite channel case, denoting the stream function by $\phi$, $v$ satisfies
\begin{align}
\label{eq:integratingbypartslin}
  \begin{split}
  \|v- \langle v \rangle_{x}\|_{L^{2}(\T_{L} \times \R)}^{2}&\leq \|v\|_{L^{2}}^{2}= \iint_{\T_{L} \times \R} |\nabla^{\bot}\phi|^{2} = \iint_{\T_{L}\times \R} |\nabla \phi|^{2} \\
&= -\iint_{\T_{L}\times \R}  \phi \Delta \phi= -\iint_{\T_{L}\times \R}  \phi \omega,
  \end{split}
\end{align}
where we used that $\phi$ decays sufficiently rapidly for $|y|\rightarrow \infty$ and that $\Delta \phi = \omega$.
Hence, 
\begin{align}
\label{eq:infinitchancase}
  \|v- \langle v \rangle_{x}\|_{L^{2}(\T_{L}\times \R)} \lesssim \sup_{\psi \in H^{1}(\T_{L}\times\R), \|\psi\|_{H^{1}}\leq 1} \iint_{\T_{L}\times \R} \psi \omega .
\end{align}
It can be shown (see \cite[Lemma 3]{lin2004nonlinear}), that an estimate of this form also holds in the setting of a finite channel, where the supremum is instead taken over elements of $\hat H^{1} :=\{ \psi \in H^{1}(\T_{L} \times [0,1]) : \psi=0 \text{ for } y \in \{ 0,1\} \}$, i.e.
\begin{align}
\label{eq:finitchancase}
  \|v\|_{L^{2}(\T_{L}\times [0,1])} \lesssim \sup_{\psi \in \hat{ H^{1}}, \|\psi\|_{H^{1}}\leq 1} \iint_{\T_{L}\times [0,1]} \psi \omega.
\end{align}
Indeed, let $\phi$ be the stream function corresponding to $v$, then 
\begin{align*}
  \nabla^{\bot}(\phi- \langle \phi \rangle_{x}) &= v - \langle v \rangle_{x}, \\
  \phi- \langle \phi \rangle_{x}|_{y=0,1}&=0,
\end{align*}
where we used that on the boundary, $y \in \{0,1\}$ 
\begin{align*}
  0= v_{2} = \p_{x} \phi,  
\end{align*}
and hence $\phi - \langle \phi \rangle_{x}|_{y=0,1}=0$. 
An integration by parts as in \eqref{eq:integratingbypartslin} thus yields no boundary contributions and hence the same estimate.
\\

For simplicity of notation, in the following we use $\hat{H}^{1}$ to also denote $H^{1}(\T_{L}\times\R)$, so that both \eqref{eq:infinitchancase} and \eqref{eq:finitchancase} read the same.

We further introduce $f_{k}(t,y):= \mathcal{F}_{x}W (t,k,y)$. Then,
\begin{align*}
  \begin{split}
\|v- \langle v \rangle_{x}\|_{L^{2}(\Omega)} &\lesssim\sup_{\psi \in \hat H^{1}, \|\psi\|_{H^{1}}\leq 1} \left| \iint_{\Omega} \psi \omega  \right| \\
 &= \sup_{\psi \in \hat H^{1}, \|\psi\|_{H^{1}}\leq 1} \left| \sum_{k\neq 0 }  \int \psi_{-k}f_{k} e^{iktU(y)}\right|.
  \end{split}
\end{align*}
We integrate by parts to obtain
\begin{align}
  \label{eq:75}
 \int \psi_{-k}f_{k} e^{iktU(y)} dy=  -\int \frac{e^{iktU(y)}}{ikt} \p_{y}\left(\frac{\psi_{-k}f_{k}}{U'}\right) dy,
\end{align}
where, in the case of a finite channel, the boundary terms
\begin{align*}
  \left.\frac{e^{iktU(y)}}{iktU'(y)} \psi_{-k}f_{k}\right|_{y=0}^{1}
\end{align*}
vanish as $\psi$ vanishes on the boundary.
Using the strict monotonicity of $U$ and Hölder's inequality, we thus bound 
\begin{align}
\label{eq:100}
  \|v(t)- \langle v \rangle_{x}\|_{L^{2}(\Omega)}\lesssim \sup_{\psi \in \hat H^{1}, \|\psi\|_{H^{1}}\leq 1} \mathcal{O}(t^{-1}) \|W(t)\|_{H^{-1}_{x}H^{1}_{y}} \|\psi\|_{H^{1}}, 
\end{align}
which establishes the first statement.
\\

In order to bound $v_{2}$, we proceed slightly differently.
Note that $v_{2}$ satisfies 
\begin{align}
\label{eq:101}
  \Delta v_{2} = \p_{x}\omega .
\end{align}
We thus introduce a potential $\psi$ such that 
\begin{align*}
 \Delta \psi=v_{2}.
\end{align*}
In the case of an infinite channel, we require that $v=\nabla^{\bot} \psi \in L^{2}(\T_{L}\times \R)$.
For the finite channel we additionally require zero Dirichlet conditions, i.e.
\begin{align}
 \psi=0, \text{ for } y \in \{0,1\}. 
\end{align}
Therefore,
\begin{align*}
   \iint_{\T_{L}\times[0,1]} \p_{x}\omega \psi &= \iint_{\T_{L}\times[0,1]} \Delta v_{2} \psi \\ &=
 \int_{0}^{1} \left. \psi \p_{x}v_{2}\right|_{x=0}^{L} dy + \int_{\T_{L}} \left. \psi \p_{y}v_{2}\right|_{y=0}^{1} dx
  - \iint_{\T_{L}\times[0,1]} \nabla v_{2} \cdot\nabla \psi \\
 &= -\int_{0}^{1} \left. v_{2}\p_{x}\psi\right|_{x=0}^{L} dy - \int_{\T_{L}} \left. v_{2}\p_{y}\psi \right|_{y=0}^{1} dx
 + \iint v_{2}\Delta \psi = \|v_{2}\|_{L^{2}(\T_{L}\times [0,1])}^{2}, 
\end{align*}
where we used periodicity in $x$ and that $v_{2}$ and $\psi$ vanish whenever $y \in \{0,1\}$.
Hence, for both the infinite and finite channel,
\begin{align}
\label{eq:kinbyenstrophy}
\|v\|_{L^{2}(\Omega)}^{2}= \iint_{\Omega} \p_{x}\omega \psi .  
\end{align}

Using \eqref{eq:kinbyenstrophy}, we compute
\begin{align*}
  \begin{split}
    \|v_{2}\|_{L^{2}(\Omega)}^{2}&= \iint \p_{x}\omega \psi = \sum_{k} \int ik e^{iktU(y)}f_{k} \psi_{-k} \\
&= \sum_{k} \int \frac{e^{iktU(y)}}{t}\p_{y}\left(\frac{f_{k}\psi_{-k}}{U'}\right).
\end{split}
\end{align*}
Integrating by parts once more, we obtain 
\begin{align*}
- \frac{1}{t^{2}}\sum_{k} \int \frac{e^{iktU}}{ik}\p_{y}\left(\frac{1}{U'}\p_{y}\left(\frac{f_{k}\psi_{-k}}{U'}\right)\right),
  \end{align*}
and an additional boundary term in the setting of a finite channel:    
  \begin{align*}
 \frac{1}{t^{2}}\sum \left. \frac{e^{iktU(y)}}{ikU'}\p_{y}\left(\frac{f_{k}\psi_{-k}}{U'}\right)\right|_{y=0}^{1}. 
\end{align*}

Using Hölder's inequality, trace estimates and that $\frac{1}{U'} \in W^{2,\infty}(\Omega)$, we hence obtain:
\begin{align}
  \label{eq:78}
  \|v_{2}(t)\|_{L^{2}(\Omega)}^{2}\lesssim \mathcal{O}(t^{-2})\|W(t)\|_{H^{-1}_{x}H^{2}_{y}(\Omega)}\| \psi\|_{H^{2}(\Omega)}.
\end{align}
By classic elliptic regularity theory for the Laplacian, $\|\psi\|_{H^{2}(\Omega)}\lesssim \| v_{2}\|_{L^{2}(\Omega)}$. Thus, dividing by $\|v\|_{L^{2}(\Omega)}$ yields the result.
\end{proof}

\begin{rem}
  \begin{itemize}
  \item Assuming that $\|W(t)\|_{H^{-1}_{x}H^{2}_{y}}$ is bounded uniformly in $t$, we hence obtain damping with the optimal algebraic rates.
Furthermore, slightly slower decay still holds, if the growth of the norms of $W(t)$ can be adequately controlled.
Consider for example the last inequality \eqref{eq:78}:
\begin{align*}
  \|v_{2}(t)\|_{L^{2}}\lesssim \mathcal{O}(t^{-2})\|W(t)\|_{H^{-1}_{x}H^{2}_{y}}. 
\end{align*}
If $\|W(t)\|_{H^{-1}_{x}H^{2}_{y}}$ grows with a rate of $\mathcal{O}(t^{\alpha}), \alpha <2$, then $\|v(t)\|_{L^{2}}=\mathcal{O}(t^{\alpha-2})$ still decays.
\item Analogously to Theorem \ref{thm:linzengCouette}, it is possible to interpolate between the two estimates of Theorem \ref{thm:lin-zeng} and hence obtain
  \begin{align*}
    \|v(t)- \langle v \rangle_{x}\|_{L^{2}(\Omega)} = \mathcal{O}(t^{-s})\|W(t)\|_{H^{-1}_{x}H^{s}_{y}(\Omega)},
  \end{align*}
  for $1<s<2$, provided $W(t) \in H^{-1}_{x}H^{s}_{y}(\Omega)$ for all times. 
  \end{itemize}
\end{rem}

Consider the linearized Euler equations, \eqref{eq:LEaroundU}, in either the finite or infinite channel and introduce 
\begin{align*}
  V_{2}(t,x,y):= v_{2}(t,x-tU(y),y).
\end{align*}
Then $W$ satisfies 
\begin{align}
\label{eq:evoW}
  \dt W = U''(y) V_{2}.
\end{align}
Furthermore, since
\begin{align*}
  (x,y) \mapsto (x-tU(y),y)
\end{align*}
is an $L^{2}$ isometry,
\begin{align*}
  \|V_{2}\|_{L^{2}(\Omega)}= \|v_{2}\|_{L^{2}(\Omega)}.
\end{align*}
Integrating \eqref{eq:evoW}, sufficient decay of $\|v_{2}\|_{L^{2}}$ hence implies a scattering result.
\begin{thm}[Scattering]
  Let $\Omega$ be either the infinite periodic channel or finite periodic channel and let $\omega$ be a solution of the linearized Euler equations, \eqref{eq:LEaroundU}, on $\Omega$
  with initial datum $\omega_{0} \in L^{2}_{x}H^{2}_{y}(\Omega)$.
  Let further $U$ satisfy the assumptions of Theorem \ref{thm:lin-zeng}, $U'' \in L^{\infty}(\Omega)$ and suppose that, for all times $t$, $W$ satisfies
  \begin{align*}
 \|W(t)- \langle W \rangle_{x} \|_{L^{2}_{x}H^{2}_{y}(\Omega)}< C< \infty .    
  \end{align*}
Then there exist asymptotic profiles $W^{\pm \infty}\in L^{2}_{x}H^{2}_{y}(\Omega)$, such that
  \begin{align*}
    W(t) \xrightarrow {L^{2}} W^{\pm \infty},
  \end{align*}
as $t \rightarrow \pm \infty$.
\end{thm}

\begin{proof}
  By Duhamel's formula, which in our scattering formulation is just integrating \eqref{eq:evoW}, $W$ satisfies 
  \begin{align}
    \label{eq:scatter2}
    W(t)=\omega_{0} + \int^{t}_{0}U''V_{2}(\tau) d\tau.
  \end{align}
  By Theorem \ref{thm:lin-zeng}, we control
  \begin{align*}
    \left\| \int^{t}_{0}U''V_{2}(\tau) d\tau \right\|_{L^{2}(\Omega)} \leq \|U''\|_{L^{\infty}(\Omega)} \int^{t}_{0} \mathcal{O}(\tau^{-2}) d\tau.
  \end{align*}
  Therefore, the limits $W^{\pm \infty}$ of \eqref{eq:scatter2} as $t \rightarrow \pm \infty$ exist in $L^{2}(\Omega)$ and by weak compactness of the unit ball of $L^{2}_{x}H^{2}_{y}(\Omega)$ and lower semi-continuity, also $W^{\pm \infty} \in L^{2}_{x}H^{2}_{y}(\Omega)$. 
\end{proof}
In the following subsection, we further generalize the conditional damping results from shear flows, $(x,y) \mapsto (x-tU(y),y)$, to diffeomorphisms $Y$, which are structurally similar to shear flows.

\subsection{Diffeomorphisms with shearing structure}
\label{sec:damp-scatt-under}
Consider the full 2D Euler equations in  either the infinite periodic channel, $\T \times \R$,  or the finite periodic channel, $\T \times [0,1]$,
\begin{align}
  \label{eq:Eulerflow}
  \begin{split}
  \dt \omega + v \cdot \nabla \omega &=0, \\
  \nabla \times v &= \omega, \\
  \nabla \cdot v &=0, \\
 \omega |_{t=0}&=\omega_{0},
  \end{split}
\end{align}
where, in the case of a finite periodic channel, we consider impermeable walls, i.e.
\begin{align}
  v_{2}=0, \text{ for } y \in \{0,1\}.
\end{align}

Restricting to sufficiently regular solutions, we may equivalently consider the evolution of the flow maps $X_{t}$ (c.f. \cite[Chapter 2.5]{Maj}):
\begin{align}
\label{eq:4}
\begin{split}
  \dt X_{t} &= v(t,X_{t}), \\ 
  X_{0}&=Id, \\
  \omega(t,X_{t}) &=\omega_{0}.
\end{split}
\end{align}
We further recall that, as $v$ is divergence-free, $DX$ satisfies  
\begin{align*}
\det (DX) \equiv& 1,
\end{align*}
and is thus measure-preserving and invertible. 
Hence, if $\omega_{0} \in L^{p}(\Omega)$, then for any time $t$ also $\omega(t) \in L^{p}(\Omega)$ and
\begin{align*}
  \|\omega(t)\|_{L^{p}(\Omega)}= \|\omega_{0}\|_{L^{p}(\Omega)}.
\end{align*}
However, we note that, in an infinite periodic channel, for solutions close to a monotone shear flow, $U(y)$, in general $\omega_{0} \not \in L^{p}(\T_{L}\times \R)$, since $\nabla \times (U(y),0)= -U'(y) \not \in L^{p}(\T_{L}\times\R)$.
Furthermore, if $X_{t}$ is not a shear, then 
\begin{align*}
  \langle \omega \rangle_{x}=\langle \omega_{0} \circ X \rangle_{x} \neq \langle \omega_{0} \rangle_{x} \circ X \neq \langle \omega_{0} \rangle_{x}.
\end{align*}
Thus, unlike in the linear setting, the ``underlying shear'':
\begin{align}
\label{eq:notshearprop1}
  \langle v \rangle_{x} =
  \begin{pmatrix}
    \langle v_{1} \rangle_{x}(t,y) \\0
  \end{pmatrix}
\end{align}
corresponding to 
\begin{align*}
  \nabla \times \langle v \rangle_{x} &= \langle \omega \rangle_{x}, \\
  \nabla \cdot \langle v \rangle_{x} &=0,
\end{align*}
is not anymore time-independent.
\\

In the following, we thus instead consider $\langle \omega \rangle_{x}(t,y)$ and $\langle v \rangle_{x}(t,y)$ as given functions and let $Y_{t}$ denote the flow by $\langle v \rangle_{x}$, i.e. the solution map of 
\begin{align*}
  \dt f + \langle v \rangle_{x} \cdot \nabla f = 0.
\end{align*}
The flow, $Y_{t}$, is then of the form 
\begin{align}
\label{eq:niceformforY}
  Y_{t}: (x,y)\mapsto (x-u(t,y), y), 
\end{align}
where 
\begin{align}
  u(t,y)= \int^{t}_{0}\langle v_{1} \rangle_{x}(\tau,y) d\tau .
\end{align}
In particular, denoting
\begin{align*}
  W(t):= (\omega - \langle \omega \rangle_{x})\circ Y_{t}^{-1},
\end{align*}
we observe that, unlike \eqref{eq:notshearprop1}:
\begin{align}
  \langle W(t) \rangle_{x} = \langle W(t) \rangle_{x} \circ Y_{t} = 0. 
\end{align}

Similar to Theorem \ref{thm:lin-zeng}, in the following theorem we \emph{assume} that $Y_{t}$ is a good approximation to $X$ in the sense that $W(t) \in H^{2}_{x,y}(\Omega)$, uniformly in time.

We then study under which assumptions on $Y_{t}$, the perturbation to the velocity field $v- \langle v \rangle_{x}$:
\begin{align}
  \begin{split}
  \nabla \times (v- \langle v \rangle_{x}) &= \omega- \langle \omega \rangle_{x}= W \circ Y_{t}, \\
  \nabla \cdot (v- \langle v \rangle_{x}) &=0,
  \end{split}
\end{align}
decays with algebraic rates.

\begin{thm}[Damping in terms of the flow map $Y$ and $W$]
\label{thm:dampingsheardiffeo}
Let $W(t) \in L^{2}_{x}H^{1}_{y}(\Omega)$ be such that for all times  
\begin{align}
  \langle W(t) \rangle_{x} &=0, \\
  \|W(t)\|_{L^{2}_{x}H^{1}_{y}(\Omega)} &< C < \infty .
\end{align}
Let further $Y_{t}$ be given by 
\begin{align}
  Y_{t}: (x,y)\mapsto (x-u(t,y), y),
\end{align}
and suppose $\p_{y}u(t,y) \in W^{2,\infty}$ satisfies 
\begin{align}
  \inf_{t,y}\frac{1}{t}\p_{y} u(t,y) > c>0.
\end{align}

Then, for any test function $\psi \in H^{1}(\Omega)$ with compact support in $y$:
 \begin{align}
   \label{eq:34}
   \begin{split}
    \iint \psi W \circ Y &= \iint \psi \frac{d}{dx} \left(\frac{d}{dx}\right)^{-1} W \circ Y \\
&= \iint \psi \frac{1}{\p_{y}u(t,y)} \frac{d}{dy} \left(\frac{d}{dx}\right)^{-1} W \circ Y + \left(\frac{d}{dx}\right)^{-1} (\p_{y}W) \circ Y) \\
&= \iint \frac{1}{\p_{y}u(t,y)} \psi \left(\frac{d}{dx}\right)^{-1} ((\p_{y}W) \circ Y ) - \frac{d}{dy}\left(\frac{1}{\p_{y}u(t,y)} \psi \right) \left(\frac{d}{dx}\right)^{-1} W \circ Y .
   \end{split}
 \end{align}

In particular, taking the supremum over all test functions $\psi$ such that $\|\psi\|_{H^{1}(\Omega)}\leq 1$, we obtain
\begin{align*}
  \|v(t)- \langle v \rangle_{x}\|_{L^{2}} \lesssim \frac{1}{ct}\|W(t)\|_{H^{-1}_{x}H^{1}_{y}} \lesssim \frac{1}{ct}\|W(t)\|_{L^{2}_{x}H^{1}_{y}}=\mathcal{O}(t^{-1}). 
\end{align*}
\end{thm}

\begin{proof}[Proof of Theorem \ref{thm:dampingsheardiffeo}]
As $W$ satisfies $\langle W \rangle_{x}=0$ and as this property is preserved under composition with $Y$,
$ \left( \frac{d}{dx} \right)^{-1} W \circ Y$ is well-defined and
\begin{align*}
  W \circ Y=  \frac{d}{dx} \left(\frac{d}{dx}\right)^{-1} W \circ Y = \frac{d}{dx} \left( \left(\frac{d}{dx}\right)^{-1} W \right) \circ Y.
\end{align*}
We further note that, by the chain rule 
\begin{align}
  \begin{split}
  \frac{d}{dx} W \circ Y &= \p_{x}Y_{1} (\p_{x}W) \circ Y  + \p_{x}Y_{2} (\p_{y} W) \circ Y,\\
  \frac{d}{dy} W \circ Y &= \p_{y}Y_{1} (\p_{x}W) \circ Y  + \p_{y}Y_{2} (\p_{y}W) \circ Y,
  \end{split}
\end{align}
and that 
\begin{align}
  \det 
  \begin{pmatrix}
   \p_{x}Y_{1} & \p_{y} Y_{1} \\ 
   \p_{x}Y_{2} & \p_{y} Y_{2}
  \end{pmatrix} \equiv 1. 
\end{align}
Thus, 
\begin{align*}
  \frac{d}{dx} W \circ Y = \frac{\p_{x}Y_{2}}{\p_{y}Y_{1}} \frac{d}{dy} W \circ Y + \frac{1}{\p_{y}Y_{1}} (\p_{y}W) \circ Y .
\end{align*}
The equation \eqref{eq:34} hence follows using integration by parts. 

In order to prove the desired damping result, we recall from the proof of Theorem \ref{thm:lin-zeng}, that 
\begin{align*}
  \|v- \langle v \rangle_{x}\|_{L^{2}} \lesssim \sup_{\psi: \|\psi\|_{H^{1}(\Omega)\leq 1}} \iint \psi (\omega - \langle \omega \rangle_{x}).
\end{align*}
Using \eqref{eq:34}, the proof hence concludes by an application of Hölder's inequality and using that 
\begin{align*}
  \frac{1}{\p_{y}u(t,y)} < \frac{1}{ct}.
\end{align*}
\end{proof}

As seen in the proof, the theorem can be formulated for flows not of the form \eqref{eq:niceformforY} and we can also allow $\det(DY)$ to be non-constant. In this case, \eqref{eq:34} is replaced by 
\begin{align*}
  \iint \psi W \circ Y = \iint \frac{\det (DY)}{\p_{y}Y_{1}} \psi \left(\frac{d}{dx}\right)^{-1} ((\p_{y}W) \circ Y ) - \frac{d}{dy}\left(\frac{\p_{x}Y_{2}}{\p_{y}Y_{1}} \psi \right) \left(\frac{d}{dx}\right)^{-1} W \circ Y.
\end{align*}
However, in order to use $\left( \frac{d}{dx} \right)^{-1} W \circ Y$, we have to require that 
\begin{align*}
  \langle W \circ Y \rangle_{x} =0,
\end{align*}
which heavily restricts the possible choices for $Y$ and $W$.
In particular, in general one can not choose $W=\omega_{0}- \langle \omega_{0} \rangle_{x}$ and $Y=X$.
\\

Thus far all damping results have been \emph{conditional} under the assumption of regularity.
In the following two sections we remove this restriction by establishing stability and thus regularity of the linearized Euler equations considered as a scattering problem around the underlying transport equation,
\begin{align*}
  \dt \omega + U(y) \p_{x}\omega=0.
\end{align*}

\section{Asymptotic stability for an infinite channel}
\label{sec:asympt-stab-y}
As discussed in Section \ref{sec:damping}, thus far all our damping results
are \emph{conditional} under the assumption that our scattered solution, $W$, of 
\begin{align}
\begin{split}
  \dt \omega + U(y)\p_{x} \omega &= U'' v_{2}, \text{ on } \T_{L} \times \R \times \R \ni (x,y,t), \\
  v_{2}&= \p_{x} \Delta^{-1}\omega, \\  
W(t,x,y):&= \omega(t,x-tU(y),y),
\end{split}
\end{align}
stays regular in the sense that the $L^{2}$, $H^{1}$ and $H^{2}$ norm of $W$ remain
uniformly bounded or at least grow very slowly. 

In the case of $L^{2}$ stability, there are classical stability results due to
Rayleigh, \cite{Rayleigh}, Fj\o rtoft, \cite[page 132]{drazin2}, and Arnold, \cite{arnold1966}. However, these results use fundamentally
different mechanisms, namely orthogonality, cancellation or convexity,
while we use mixing by shearing. In particular, our flows are
in general not covered by any of these classical stability results.
Furthermore, we show that the shearing mechanism is more robust in the sense that it can also be used to derive stability results in higher Sobolev norms. 

Before stating the main result, we introduce coordinate transformations,  notation and perform a Fourier transform in $x$ to simplify the equation.

As $U: \R \mapsto \R$ is strictly monotone, it is also bijective and invertible. We hence introduce a change of variables, $y \mapsto z=U(y)$, as well as functions
\begin{align}
  \begin{split}
  f(z):&= U''(U^{-1}(z)), \\
  g(z):&=U'(U^{-1}(z)).
  \end{split}
\end{align}
Here, it is convenient to assume that $U'$ is not only bounded from below but also from above so that the change of variables is bilipschitz.
For simplicity of notation, we often also assume that $g>0$, i.e. $U$ is strictly monotonically increasing, but all described results remain valid for strictly monotonically decreasing $U$ as well. 

In the new coordinates, the linearized Euler equations are given by
\begin{align}
  \label{eq:200}
  \begin{split}
  \dt \omega + z \p_{x} \omega &= f(z) \p_{x} \phi , \\
  (\p_{x}^{2} +(g(z)\p_{z})^{2}) \phi &= \omega.
  \end{split}
\end{align}
The underlying transport structure hence turns into Couette flow, which is particularly useful for computing derivatives and
applications of a Fourier transform.
As a trade off, the equation for the stream function is not anymore given by the Laplacian. However, the equation is still elliptic \emph{if and only if} $g$ is bounded away from zero, i.e. iff $U$ is strictly monotone.

Changing to a \emph{scattering formulation}, i.e. introducing
\begin{align}
  \begin{split}
  W(t,x,z)&:=\omega(t,x-tz,z), \\
  \Phi(t,x,z)&:=\phi(t,x-tz,z),
  \end{split}
\end{align}
the left-hand-side of \eqref{eq:200} simplifies and we obtain
\begin{align*}
  \dt W &= f(z) \p_{x} \Phi, \\
  (\p_{x}^{2}+(g(z)(\p_{z}-t\p_{x}))^{2})\Phi &=W.
\end{align*}
We further note that, like Couette flow, the $x$ average $\langle W \rangle_{x}= \langle \omega \rangle_{x}$ satisfies
\begin{align*}
  \dt \langle W \rangle_{x} = f(z)\langle \p_{x} \Phi \rangle \equiv 0
\end{align*}
and is thus conserved.
We may therefore subtract $\langle \omega_{0} \rangle_{x}$ from $\omega_{0}$ and assume that
\begin{align*}
  \langle W \rangle_{x}(t,y) \equiv 0.
\end{align*}

As $f$ and $g$ do not depend on $x$, after a Fourier transform in $x$ the system \emph{decouples} and the frequency $k$ plays the role of a parameter
\begin{align*}
  \dt \hat W &= f(z) ik \hat \Phi , \\
  (-k^{2}+(g(z)(\p_{z}-ikt))^{2}) \hat \Phi &= \hat W, \\
  (k,y,t)&\in (\Z \setminus \{0\}) \times \R \times \R .
\end{align*}
Furthermore, we adjust the definition of $\Phi$ by dividing by $k^{2}$, which is well-defined, as we assumed that 
\begin{align*}
  \langle W \rangle_{x}(t,z)= \hat W(k=0,t,\eta) \equiv 0.
\end{align*}

Relabeling $z$ as $y$, we thus obtain the following \emph{linearized Euler equations in scattering formulation}:
\begin{align}
\label{eq:LEscat}
\begin{split}
  \dt \hat W &=\frac{if}{k} \hat \Phi , \\
  (-1 + (g (\frac{\p_{y}}{k}-it))^{2}) \hat \Phi &= \hat W, \\
(k,y,t)&\in (\Z \setminus \{0\}) \times \R \times \R.
\end{split}
\end{align}

Our main result of this section is given by the following stability theorem, which is proved in Subsection \ref{sec:iter-arbitr-sobol}.
\begin{thm}[Sobolev stability for the infinite periodic channel]
\label{thm:Sobolevstab}
  Let $s \in \N_{0}$ and $f,g \in W^{s+1,\infty}(\R)$ and suppose that there exists $c>0$, such that
  \begin{align*}
    0<c<g<c^{-1}<\infty .
  \end{align*}
  Suppose further that 
  \begin{align*}
   L \|f\|_{W^{s+1,\infty}} 
  \end{align*}
  is sufficiently small.
  Then for all $m \in \N_{0}$ and $\omega_{0} \in H^{m}_{x}H^{s}_{y}(\T_{L}\times \R)$, the solution W of the linearized Euler equations in scattering formulation, \eqref{eq:LEscat}, with initial datum $\omega_{0}$ satisfies 
  \begin{align*}
    \|W(t)\|_{H^{m}_{x}H^{s}_{y}(\T_{L}\times \R)} \lesssim \|\omega_{0} \|_{ H^{m}_{x}H^{s}_{y}(\T_{L}\times\R)}.
  \end{align*}
\end{thm}

\begin{rem}
\label{rem:k_is_a_parameter}
As \eqref{eq:LEscat} decouples with respect to $k$, in our stability results we actually prove that, for any given $k$, 
\begin{align*}
  \|\hat{W}(t,k,\cdot)\|_{H^{s}_{y}(\R)} \lesssim \|\hat{\omega}_{0}\|_{H^{s}_{y}(\R)}.
\end{align*}
The results for $H^{m}_{x}H^{s}_{y}$ are then obtained by summing in $k$.
In particular, any result for $L^{2}_{x}H^{s}_{y}$ can be easily shown to also hold for $H^{m}_{x}H^{s}_{y}$.

In the following sections, we hence consider $k$ as a fixed parameter in \eqref{eq:LEscat} and study the stability of $\hat{W}(t,k,\cdot) \in H^{s}(\R)$.
\end{rem}

\begin{rem}

A main difficulty in establishing stability results such as Theorem \ref{thm:Sobolevstab} is that the operator 
\begin{align*}
  W \mapsto \Phi,
\end{align*}
interpreted as an operator from $L^{2}$ to $L^{2}$ does not improve in time, as multiplication by $e^{ikty}$ is a unitary operation.
More precisely, for any given $k$, the operator norm of the solution operator to
\begin{align}
\label{eq:ope}
 e^{-ikty} (-k^{2}+ (g\p_{y})^{2}) e^{ikty}
\end{align}
is independent of time.
As a consequence, the uniform damping results of Section \ref{sec:damping} necessarily sacrifice regularity in order to obtain uniform decay.
In the proof of Theorem \ref{thm:Sobolevstab}, we therefore have to use the more subtle mode-wise decay, where for each fixed frequency, $(k,\eta)$, the solution operator of \eqref{eq:ope}  decays with rate $\mathcal{O}(|\eta-kt|^{-2})$.
\end{rem}

In the following, we first introduce the mechanism of our proof in a simplified
setting of a constant coefficient model, for which we can also compute the solution explicitly.
Using a perturbation argument, we establish $L^{2}$ stability for the general setting in Section \ref{sec:reduct-gener-monot} and subsequently extend the result to higher Sobolev norms in Section \ref{sec:iter-arbitr-sobol}.

\subsection{A constant coefficient model}
\label{sec:const-coeff-model}

In order to obtain a better understanding of the dynamics of the linearized Euler
equations, in the following we consider a simplified model. Here, we formally replace $f(y)$ and $g(y)$ in \eqref{eq:LEscat} by constants to recover the decoupling:
\begin{align}
\label{eq:CCproblem}
  \tag{CC}
  \begin{split}
  \dt \Lambda &= c \Psi, \\
  (-1+(\frac{\p_{y}}{k}-it)^{2}) \Psi &= \Lambda, \\
  \Lambda|_{t=0} &= \hat{\omega}_{0}(k,\cdot), \\
  (k,y,t) &\in L (\Z \setminus \{0\}) \times \R \times \R . 
  \end{split}
\end{align}
Here, $c \in \C$ should be thought of as small and not necessarily imaginary.
For simplicity of notation, we choose the constant in front of $(\frac{\p_{y}}{k}-it)^{2}$ to be $1$. In general, $\min (g^{2})>0$ is the
natural choice.
\\

Like the linearized Euler equations in scattering formulation, (\ref{eq:LEscat}), the model problem, \eqref{eq:CCproblem}, decouples with respect to $k$ (c.f. Remark \ref{rem:k_is_a_parameter}).
In the following, we hence write $\Lambda(t) \in H^{s}=H^{s}(\R)$ to denote that, for given $k$,
\begin{align*}
\Lambda(t,k,\cdot) \in H^{s}(\R).  
\end{align*}
Estimates in the Sobolev spaces $H^{m}_{x}H^{s}_{y}(\T_{L}\times \R)$ can then be obtained by summing in $k$.
\\

By our choice of constant coefficients in \eqref{eq:CCproblem}, the model problem further decouples after a Fourier transform in $y$ and is explicitly solvable:
\begin{thm}
  Let $\omega_{0} \in L^{2}$, then the solution of the constant
  coefficient problem, \eqref{eq:CCproblem}, is given by
  \begin{align}
    \Lambda = \mathcal{F}^{-1} \exp\left(c \left(\arctan(\frac{\eta}{k}-t)-\arctan(\frac{\eta}{k})\right) \right) \mathcal{F}
    \omega_{0}.
  \end{align}
  In particular, for any $s \in \N$ such that $\omega_{0} \in H^{s}$, also $\Lambda(t) \in H^{s}$ and
  \begin{align*}
    \|\Lambda(t)\|_{H^{s}} \leq e^{|c|\pi} \|\hat{\omega}_{0}(k,\cdot)\|_{H^{s}}
  \end{align*}
  uniformly in time.
\end{thm}

\begin{rem}
  An estimate by $\pi |\Re (c)|$ would of course also be possible in this case.
  However, dropping the imaginary part of $c$ corresponds to using antisymmetry and orthogonality, which is more difficult to employ in the variable coefficient setting.
  As we seek to obtain a robust strategy, we therefore limit ourselves to using the shearing mechanism only.
\end{rem}

While the constant coefficient case allows for an explicit solution,
in the general case a more indirect proof is required,
which we introduce in the following.

The underlying method of our proof is to introduce a weight that
decreases at the right places at a large enough rate to counter
potential growth. This method of proof is reminiscent of integrating factors in ODE theory and is sometimes called \emph{ghost energy}, \cite{alinhac2001null}. Recent applications
of similar methods can, for example, be found in a more
sophisticated form in the work of \cite{bedrossian2013inviscid}.

For simplicity of notation, in the following we assume that $c>0$, in order to avoid
writing absolute values.
\begin{thm}
  \label{thm:CCestimate}
  Let $c>0$ and let $\Lambda$ be the solution of the constant coefficient
  problem, \eqref{eq:CCproblem}, with initial data $\omega_{0}$. Let $C>0$ and define
  \begin{align}
    E(t):= \langle \Lambda, \mathcal{F}^{-1}_{\eta} \exp \left( C \arctan(\frac{\eta}{k}-t) \right)
    \mathcal{F}_{y} \Lambda \rangle =: \langle \Lambda, A(t) \Lambda \rangle.
  \end{align}
  Then for $|c| \ll C$ sufficiently small,
  $E(t)$ is non-increasing and uniformly comparable to
  $\|W(t)\|_{L^{2}}^{2}$.
  In particular,
  \begin{align}
\label{eq:CCL2stability}
    e^{-C\pi} E(t) \leq \|\Lambda(t)\|_{L^{2}}^{2} \leq e^{C\pi} E(t) \leq
    e^{C\pi} E(0) \leq e^{2C\pi} \|\omega_{0}\|_{L^{2}}^{2}.
  \end{align}
\end{thm}

\begin{rem}
  As can be seen from the explicit solution, the assumptions of Theorem \ref{thm:CCestimate} and the factors in \eqref{eq:CCL2stability} are not optimal for our decoupling model.
  For example, even for large $c$, choosing $C\geq c$ would work. However, in the general case, we additionally have to control the commutator of $A$ and multiplication by $\frac{f}{ik}$.
  Hence, at least for finite times, we can not avoid incurring an operator
  norm, $e^{C\pi}$, and thus a condition of the form
  \begin{align*}
    c < Ce^{-C\pi},
  \end{align*}
  which does not improve for large $C$.
  This is discussed in more detail in Section \ref{sec:reduct-gener-monot}.
 Therefore, we think of $C$ as
  approximately $1$ and require $c$ to be small.
\end{rem}

\begin{proof}[Proof of Theorem \ref{thm:CCestimate}]
  We compute the time-derivative of $E(t)$:
  \begin{align}
    \dt E(t) = \langle \Lambda, \dot A \Lambda \rangle + 2 \Re \langle A(t)\Lambda, c
    \Psi \rangle.
  \end{align}
  By our choice of $A$, $\dot A$ is a negative semidefinite symmetric
  operator. For the proof of our theorem it hence suffices to show that
  \begin{align*}
    \langle \Lambda, \dot A \Lambda \rangle \leq 0
  \end{align*}
  is negative enough to absorb the possible growth of
  \begin{align*}
    |2 \Re \langle A(t)\Lambda, c
    \Psi \rangle | .
  \end{align*}
  This therefore ensures that $\dt E(t) \leq
  0$.

  Using Plancherel, it suffices to show that
  \begin{align}
    \int_{\R} \frac{-Ce^{C \arctan (\frac{\eta}{k}-t)}}{1+(\frac{\eta}{k}-t)^{2}}
    |\tilde{\Lambda} (t,k,\eta)|^{2} d\eta + 2 \int_{\R}\Re(c) \frac{e^{C \arctan (\frac{\eta}{k}-t)}}{1+(\frac{\eta}{k}-t)^{2}}
    |\tilde{\Lambda} (t,k,\eta)|^{2} d \eta \leq 0,
  \end{align}
  for arbitrary functions $|\tilde{\Lambda}(t,k,\eta)|$, which in this case holds
  if
  \begin{align*}
   2|c| \leq C.
  \end{align*}
\end{proof}

\subsection{$L^{2}$ stability for monotone shear flows}
\label{sec:reduct-gener-monot}

In the following, we adapt the $L^{2}$ stability result, Theorem \ref{thm:CCestimate} of Section \ref{sec:const-coeff-model}, to the linearized Euler equations in scattering formulation, \eqref{eq:LEscat},
\begin{align}
  \label{eq:5}
  \begin{split}
    \dt W &= \frac{if}{k}\Phi,  \\
(-1+(g(\frac{\p_{y}}{k}-it))^{2})\Phi &= W,\\
(k,y,t)&\in L (\Z \setminus \{0\}) \times \R \times \R, 
  \end{split}
\end{align}
where for simplicity we dropped the hats, $\hat{\cdot}$, from our notation.
As noted in Remark \ref{rem:k_is_a_parameter}, \eqref{eq:5} decouples with respect to $k$.
For the remainder of this article,
we thus follow the same convention as in Section \ref{sec:const-coeff-model}
and use $W \in H^{s}(\R)$ to denote that, for given $k$,
\begin{align*}
  W(t,k,\cdot)\in H^{s}(\R).
\end{align*}

In analogy to the constant coefficient model, \eqref{eq:CCproblem}, for a given solution $W$ of \eqref{eq:5}, we introduce the \emph{constant coefficient stream function} $\Psi$:
\begin{align}
\label{eq:PsiforW}
(-1+(\frac{\p_{y}}{k}-it)^{2})\Psi &= W. 
\end{align}
We stress that, starting from this section, $\Psi$ does not correspond to a solution of the constant coefficient problem, \eqref{eq:CCproblem}, but only to a given right-hand-side $W$ in \eqref{eq:PsiforW}.

More generally, we introduce the following notation:
\begin{defi}[Constant coefficient stream function]
\label{defi:CCstreamfct}
Let $k \in L (\Z \setminus \{0\})$ and let $R(t) \in L^{2}(\R)$ be a given function. Then the \emph{constant coefficient stream function}, $\Psi[R](t)$, is defined as the solution of 
\begin{align}
  (-1+(\frac{\p_{y}}{k}-it)^{2})\Psi[R](t,y) &= R(t,y).
\end{align}
Let further $W$ be a solution of \eqref{eq:5}, then for any $k,t$
\begin{align}
  \Psi(t,k,y):= \Psi[W(t,k,\cdot)](t,y).
\end{align}
\end{defi}

Since $\Phi$ and $\Psi=\Psi[W]$ satisfy very similar (shifted elliptic) equations, \eqref{eq:5} and \eqref{eq:PsiforW}, with the same right-hand-side, we can estimate Sobolev norms of $\Phi$ in terms of $\Psi$, as is shown in Lemma \ref{lem:phipsi}.
We note that, for this purpose, $W$ need not solve \eqref{eq:5}, but can be any given $L^{2}$ function.

\begin{lem}
\label{lem:phipsi}
  Let $\frac{1}{g} \in W^{1,\infty}$ and assume there exists $c>0$ such that
  \begin{align*}
 0<c<g<c^{-1}< \infty.
  \end{align*}
Then for any $W(t) \in L^{2}(\R)$, the solutions $\Phi, \Psi$ of 
\begin{align}
\label{eq:Phiw}
  (-1+(g(\frac{\p_{y}}{k}-it))^{2})\Phi &= W, \\
\label{eq:Psiw}
(-1+(\frac{\p_{y}}{k}-it)^{2})\Psi &= W,
\end{align}
  satisfy
  \begin{align*}
    \|\Phi\|_{\tilde H^{1}}^{2}:=\|\Phi\|_{L^{2}}^{2} + \|(\frac{\p_{y}}{k}-it)\Phi\|_{L^{2}}^{2} \lesssim \|\Psi\|_{L^{2}}^{2} + \|(\frac{\p_{y}}{k}-it)\Psi\|_{L^{2}}^{2}.
  \end{align*}
\end{lem}

In the following, we establish $L^{2}$ stability of \eqref{eq:5} using Lemma \ref{lem:phipsi} and subsequently give a proof of Lemma \ref{lem:phipsi}.

\begin{thm}[$L^{2}$ stability for the infinite periodic channel]
\label{thm:L2}
  Let $W$ be a solution to the linearized Euler equations, \eqref{eq:5},
  and assume that $g$ satisfies the assumptions of Lemma \ref{lem:phipsi}.
  Let further $A$ be defined as in Theorem \ref{thm:CCestimate}, i.e.
  \begin{align}
   I(t):= \langle W,A(t)W \rangle_{L^{2}(\R)}:= \int |\tilde{W}(t,k,\eta)|^{2} \exp \left( C \arctan(\frac{\eta}{k}-t) \right) d\eta,
  \end{align}
  and suppose that 
  \begin{align*}
   \|f\|_{W^{1,\infty}} L 
  \end{align*}
  is sufficiently small.
  Then, for any initial datum $\omega_{0}\in L^{2}(\R)$, $I(t)$ is non-increasing and satisfies
  \begin{align*}
    \|W(t)\|_{L^{2}}^{2} \lesssim I(t) \leq I(0) \lesssim \|\omega_{0}\|_{L^{2}}^{2}.
  \end{align*}
\end{thm}

\begin{proof}[Proof of Theorem \ref{thm:L2}]
    Let $\Psi[AW]$ be as in Definition \ref{defi:CCstreamfct}, i.e. $\Psi[Aw] \in L^{2}$ is the solution of 
    \begin{align*}      
(-1+(\frac{\p_{y}}{k}-it)^{2})\Psi[AW] = AW. 
    \end{align*}
    Then, by integration by parts, the time-derivative of $I(t)$ satisfies
  \begin{align}
    \label{eq:dtI}
    \begin{split}
    \dt I(t) &= \langle W, \dot A  W \rangle + 2 \Re \langle AW, \frac{if}{k}\Phi  \rangle \\
& \leq  \langle W, \dot A  W \rangle + 2 \|\frac{f}{k}\|_{W^{1,\infty}} \|\Psi[AW]\|_{\tilde H^{1}} \|\Phi\|_{\tilde H^{1}}.
    \end{split}
  \end{align}
  By Lemma \ref{lem:phipsi}, the last term is further controlled by 
  \begin{align}
    \label{eq:reducedestimate}
    C_{1} \|\frac{f}{k}\|_{W^{1,\infty}} \|\Psi[AW]\|_{\tilde H^{1}}\|\Psi\|_{\tilde H^{1}}.
  \end{align}
  
  As $A$ is a bounded Fourier multiplier and commutes with the Fourier multiplier $u \mapsto \Psi[u]$, we control 
  \begin{align}
    \|\Psi[AW]\|_{\tilde H^{1}} \leq \|A\| \|\Psi\|_{\tilde{H}^{1}} \leq \|A\| \sqrt{|\langle W, \Psi \rangle |},
  \end{align}
  where we used that
  \begin{align}
    \|\Psi\|_{\tilde{H}^{1}}^{2}:= \|\Psi\|_{L^{2}}^{2}+ \|(\frac{\p_{y}}{k}-it)\Psi\|_{L^{2}}^{2} = - \langle W, \Psi \rangle.
  \end{align}
  Furthermore, 
    \begin{align}
      \begin{split}
    - \langle W, A \Psi \rangle &= - \langle (-k^{2}+ (\frac{\p_{y}}{ik}-t)^{2})\Psi, A \Psi \rangle \\
& = \int (k^{2}+(\frac{\eta}{k}-t)^{2}) \exp \left( C \arctan(\frac{\eta}{k}-t) \right) |\tilde{\Psi}(t,k,\eta)|^{2} d\eta .  
      \end{split}
  \end{align}
  Therefore, 
  \begin{align*}
    \|\Psi\|_{\tilde{H}^{1}}^{2} \leq \|A\| (- \langle W, A \Psi \rangle) \leq \|A\|^{2} \|\Psi\|_{\tilde{H}^{1}}^{2},
  \end{align*}
  where we used that $A^{-1}$ has the same operator norm as $A$.
  Thus,  \eqref{eq:reducedestimate} is further controlled by 
  \begin{align}
    \label{eq:202}
    C_{1} \|\frac{f}{k}\|_{W^{1,\infty}} \|A\|^{2} |\langle W, A \Psi \rangle | .
  \end{align}

  Hence,  combining \eqref{eq:dtI} and \eqref{eq:202}, $I(t)$ satisfies
  \begin{align*}
    \p_{t}I(t) \leq \langle W, \dot{A}W \rangle + C_{2} \|A\|^{2} \|\frac{f}{k}\|_{W^{1,\infty}} |\langle W, A \Psi[W] \rangle |.
  \end{align*}
  Using the explicit characterization of $A$ and $\Psi$ in Fourier space, we conclude as in the proof of Theorem \ref{thm:CCestimate}, provided 
  \begin{align}
    \label{eq:smallnesscond}
    c:= C_{2} \|f\|_{W^{1,\infty}} \|A\|^{2} \sup_{k\neq 0} \frac{1}{|k|} \lesssim e^{2C\pi}\|f\|_{W^{1,\infty}} L
  \end{align}
  is sufficiently small.
\end{proof}

\begin{proof}[Proof of Lemma \ref{lem:phipsi}]
  Testing \eqref{eq:Phiw} with $\frac{1}{g}\Phi$ and integrating by parts, we obtain:
  \begin{align}
    \int \frac{1}{g}|\Phi|^{2} + g |(\frac{\p_{y}}{k}-it)\Phi|^{2} = \langle W, \frac{1}{g}\Phi \rangle .
  \end{align}
  As by our assumption, $c<g<c^{-1}$, the left-hand-side is bounded from below by 
  \begin{align*}
   c (\|\Phi\|_{L^{2}}^{2} + \|(\frac{\p_{y}}{k}-it)\Phi\|_{L^{2}}^{2}) \gtrsim  \|\Phi\|_{\tilde H^{1}}^{2}.
  \end{align*}
  Hence, it remains to estimate $\langle W, \frac{1}{g}\Phi \rangle$ from above.

  Using \eqref{eq:Psiw} and integrating by parts, we obtain
  \begin{align*}
    \langle W, \frac{1}{g}\Phi \rangle = & \left\langle (-1+(\frac{\p_{y}}{k}-it)^{2})\Psi ,\frac{1}{g}\Phi \right\rangle \\
 \leq & \sqrt{\|\Psi\|_{L^{2}}^{2} + \|(\frac{\p_{y}}{k}-it)\Psi\|_{L^{2}}^{2}}  \sqrt{\|\frac{1}{g}\Phi\|_{L^{2}}^{2} + \|(\frac{\p_{y}}{k}-it)\frac{1}{g}\Phi\|_{L^{2}}^{2}}  \\
\lesssim &\|\frac{1}{g}\|_{W^{1,\infty}} \sqrt{\|\Psi\|_{L^{2}}^{2} + \|(\frac{\p_{y}}{k}-it)\Psi\|_{L^{2}}^{2}} \sqrt{\|\Phi\|_{L^{2}}^{2} + \|(\frac{\p_{y}}{k}-it)\Phi\|_{L^{2}}^{2}}.
  \end{align*}
Dividing by $\|\Phi\|_{\tilde H^{1}}=\sqrt{\|\Phi\|_{L^{2}}^{2} + \|(\frac{\p_{y}}{k}-it)\Phi\|_{L^{2}}^{2}}$, we thus obtain the result.
\end{proof}

\begin{rem}
  Testing \eqref{eq:Phiw} with $\Phi$ instead of $\frac{1}{g}\Phi$ has the small drawback of introducing commutators involving $gg'$ on the left-hand-side, which one can control either by a smallness or sign condition. The right-hand-side however is simplified.

  Testing \eqref{eq:Psiw} with $\Psi$ and integrating 
  \begin{align}
    \langle W, \Psi \rangle = \langle (-1+(g(\frac{\p_{y}}{k}-it))^{2})\Phi, \Psi \rangle
  \end{align}
  by parts, we analogously obtain that 
  \begin{align*}
    \|\Psi\|_{\tilde{H}^{1}} \lesssim \|\Phi\|_{\tilde{H}^{1}}.
  \end{align*}  
  One can more generally show that, up to a factor, both $\Phi$ and $\Psi$ attain
  \begin{align*}
    \|W\|_{\tilde H^{-1}}:=\sup \{ \langle W, \mu \rangle_{L^{2}} : \|\mu\|_{L^{2}}^{2} + \|(\frac{\p_{y}}{k}-it)\mu\|_{L^{2}}^{2} \leq 1 \}.
  \end{align*} 
\end{rem}
\begin{rem}
  It is possible to reduce the requirements of Theorem \ref{thm:L2} for large $\|A\|$ slightly, by noting that
  \begin{align*}
    \Psi[AW]=A\Psi[W], 
  \end{align*}
 as Fourier multipliers commute and that, as a positive multiplier, we can 
  split $A= A^{1/2}A^{1/2}$ for the purpose of our $L^{2}$ bound.
  Hence, in \eqref{eq:dtI}, instead of estimating 
  \begin{align*}
    2 \Re \langle AW, \frac{if}{k}\Phi  \rangle \lesssim  C_{1} \|\frac{f}{k}\|_{W^{1,\infty}} \|A\|^{2} |\langle W, A \Psi \rangle |,
  \end{align*}
 it suffices to obtain an estimate of the form 
  \begin{align*}
    \|A^{1/2}\frac{if}{k}\Phi\|_{\tilde H^{1}} \lesssim \|A^{1/2}\Psi\|_{\tilde H^{1}}.
  \end{align*}
  However, we note that for non-constant $f$, even for $\Phi=\Psi$, such an estimate would have to control
  \begin{align}
    \label{eq:commAfA}
    A^{1/2} f A^{-1/2},
  \end{align}
  as an operator from $\tilde{H}^{1}$ to $\tilde{H}^{1}$.
  Asymptotically, i.e. for $t \rightarrow  \pm \infty$, $\arctan(\eta -t) \rightarrow \pm\frac{\pi}{2}$ and thus $A^{\pm 1} \rightharpoonup e^{\pm C \frac{\pi}{2}}Id$.
Therefore, for all $C$,
\begin{align*}
  A^{1/2} f A^{-1/2}  \rightharpoonup f,
\end{align*}
as  $t \rightarrow \pm \infty$.
However, for each finite time we obtain commutators involving
\begin{align*}
 C(\arctan(\eta_{1}-t)-\arctan(\eta_{2}-t)),
\end{align*}
which are not bounded uniformly in $C$.
Hence, at least for finite times, the operator norm corresponding to \eqref{eq:commAfA} is not better than
\begin{align*}
  e^{c_{1}C}\|f\|_{W^{1,\infty}}.
\end{align*}
for some $c_{1}>0$, and thus only provides a small improvement over \eqref{eq:smallnesscond}. 
\end{rem}

\subsection{Iteration to arbitrary Sobolev norms}
\label{sec:iter-arbitr-sobol}

Thus far we have only shown $L^{2}$ stability. In order to derive
damping, it remains to extend the result to ensure stability in higher Sobolev
norms.

In the constant coefficient model, this generalization is trivial as
our equation is invariant under taking derivatives. Hence, after
relabeling, we may apply the $L^{2}$ result to $\p_{y}^{s} \Lambda$.
\begin{cor}
  Let $s \in \N$, $\omega_{0} \in H^{s}(\R)$ and let $\Lambda$ be the
  solution of the constant coefficient problem, \eqref{eq:CCproblem}, with initial data
  $\omega_{0}$. Then $\p_{y}^{s} \Lambda$ solves the constant
  coefficient problem, \eqref{eq:CCproblem}, with initial data $\p_{y}^{s}\omega_{0}$ and for
  $c< C e^{-C\pi}$,
  \begin{align*}
    \|\p_{y}^{s}\Lambda\|_{L^{2}} \lesssim
    \|\p_{y}^{s}\omega_{0}\|_{L^{2}}.
  \end{align*}
\end{cor}
When taking derivatives of the linearized Euler equations, we obtain additional lower order corrections due to commutators.
More precisely, for given $j \in \N$, $\p_{y}^{j}W$ satisfies:
\begin{align}
\label{eq:pyjW}
  \begin{split}
  \dt \p_{y}^{j} W = \frac{i}{k} \p_{y}^{j} (f \Phi) &=: \frac{i}{k} \sum_{j'\leq j } c_{jj'}(\p_{y}^{j-j'}f) \p_{y}^{j'}\Phi , \\
  (-1+(g(\frac{\p_{y}}{k}-it))^{2})\p_{y}^{j'}\Phi &= \p_{y}^{j'}W + [(g(\frac{\p_{y}}{k}-it))^{2}, \p_{y}^{j'}]\Phi .
  \end{split}
\end{align}
In order to control these corrections, we introduce a family of energies
\begin{align}
  I_{j}(t)= \langle \p_{y}^{j}W, A \p_{y}^{j} W \rangle,
\end{align}
and a combined energy:
\begin{align}
  E_{j}(t)= \sum_{j' \leq j} I_{j'}(t).
\end{align}
With this notation our main theorem is:
\begin{thm}[Sobolev stability for the infinite periodic channel]
\label{thm:iterated}
 Let $j \in \N$ and assume $f,g$ satisfy the assumptions of Theorem \ref{thm:L2}, $f,g \in W^{j+1,\infty}(\R)$ and that
  \begin{align*}
\| f\|_{W^{j+1,\infty}} L 
  \end{align*}
is sufficiently small.
Then for any initial datum $\omega_{0}\in H^{j}(\R)$, $E_{j}(t)$ is non-increasing and satisfies 
   \begin{align*}
    \|W(t)\|_{H^{j}}^{2}  \lesssim E_{j}(t) \leq E_{j}(0) \lesssim \|\omega_{0}\|_{H^{j}}^{2}.
  \end{align*}
\end{thm}
As in the previous proof, we compare with constant coefficient potentials $\Psi$:
\begin{lem}
\label{lem:iterated}
  Let $j \in \N$ and let $g$ satisfy the assumptions of Theorem \ref{thm:iterated}. Then,
  \begin{align*}
     \|\p_{y}^{j}\Phi\|_{\tilde H^{1}} \lesssim  \sum_{j' \leq j } \|\p_{y}^{j'}\Psi\|_{\tilde H^{1}}.
  \end{align*}
\end{lem}

\begin{proof}[Proof of Theorem \ref{thm:iterated}] 
For any $j'\leq j$, $I_{j'}$ satisfies
  \begin{align}
    \begin{split}
    \dt I_{j'}(t)&= \langle \p_{y}^{j'} W, \dot A \p_{y}^{j'} W \rangle + \langle A \p_{y}^{j'}W , \p_{y}^{j'} \frac{if}{k} \Phi \rangle \\
    \leq & \langle \p_{y}^{j'} W, \dot A \p_{y}^{j'} W \rangle + \|\Psi[A \p_{y}^{j'} W]\|_{\tilde H^{1}} \left\| \frac{f}{k}\right\|_{W^{j'+1,\infty}} \sum_{j''\leq j'} \|\p_{y}^{j''}\Phi\|_{\tilde H^{1}}.
    \end{split}
  \end{align}
Summing over all $j'\leq j$ and using Lemma \ref{lem:iterated} and Young's inequality, we hence obtain: 
\begin{align}
\label{eq:Ejs}
    \begin{split}
  \dt E_{j}(t) &\leq \sum_{j'\leq j}\langle \p_{y}^{j'} W, \dot A \p_{y}^{j'} W \rangle + \left\| \frac{f}{k}\right\|_{W^{j+1,\infty}} \left( \sum_{j'\leq j} \|\Psi[A \p_{y}^{j'} W]\|_{\tilde H^{1}}^{2 }+ \|\p_{y}^{j'}\Phi\|_{\tilde H^{1}}^{2} \right) \\
&\lesssim \sum_{j'\leq j}\langle \p_{y}^{j'} W, \dot A \p_{y}^{j'} W \rangle + \left\| \frac{f}{k}\right\|_{W^{j+1,\infty}} \left( \sum_{j'\leq j} \|\Psi[A \p_{y}^{j'} W]\|_{\tilde H^{1}}^{2 }+ \|\p_{y}^{j'}\Psi\|_{\tilde H^{1}}^{2} \right).
    \end{split}
\end{align}
We further note that $\p_{y}^{j'}\Psi= \Psi[\p_{y}^{j'}W]$.
Hence, relabeling and applying the constant coefficient $L^{2}$ result, Theorem \ref{thm:CCestimate}, we obtain that for any $j'$ and for $c$ sufficiently small 
\begin{align}
  \label{eq:jnonincreasing}
  \langle \p_{y}^{j'} W, \dot A \p_{y}^{j'} W \rangle + c (\|\Psi[\p_{y}^{j'} W]\|_{\tilde H^{1}}^{2}+\|\Psi[A \p_{y}^{j'} W]\|_{\tilde H^{1}}^{2 }) \leq 0. 
\end{align}
Supposing that
\begin{align*}
 \sup_{k \neq 0} \|\frac{f}{k}\|_{W^{j+1,\infty}}= \|f\|_{W^{j+1,\infty}} L \ll c,
\end{align*}
summing \eqref{eq:jnonincreasing} with respect to $j$ and \eqref{eq:Ejs} hence imply
\begin{align}
  \dt E_{j}(t) \leq 0,
\end{align}
which concludes our proof.
\end{proof}
\begin{proof}[Proof of Lemma \ref{lem:iterated}]
  We prove the result by induction in $j$.
  The case $j=0$ has been proven as Lemma \ref{lem:phipsi} in Section \ref{sec:reduct-gener-monot}.
  Hence, it suffices to show the induction step $j-1 \mapsto j$:
  \begin{align}
    \label{eq:inductionstep}
    \|\p_{y}^{j}\Phi\|_{\tilde H^{1}} \lesssim  \|\p_{y}^{j}\Psi\|_{\tilde H^{1}} + \sum_{j' \leq j -1} \|\p_{y}^{j'}\Phi\|_{\tilde H^{1}},
  \end{align}
  for $j\geq 1$.
\\

  Recall that $\p_{y}^{j}\Phi$ satisfies \eqref{eq:pyjW}: 
  \begin{align*}
    (-1+(g(\frac{\p_{y}}{k}-it))^{2}) \p_{y}^{j}\Phi = \p_{y}^{j} W + [(g(\frac{\p_{y}}{k}-it))^{2}, \p_{y}^{j}] \Phi.
  \end{align*}
  Proceeding as in the proof of Lemma \ref{lem:phipsi}, we thus test \eqref{eq:pyjW} with $\frac{1}{g}\p_{y}^{j}\Phi$ to obtain an estimate by
  \begin{align}
\label{eq:110}
    \|\p_{y}^{j}\Phi\|_{\tilde H^{1}}^{2} \lesssim \|\p_{y}^{j}\Psi\|_{\tilde H^{1}}\|\p_{y}^{j}\Phi\|_{\tilde H^{1}} + \langle \p_{y}^{j}\Phi,[(g(\frac{\p_{y}}{k}-it))^{2}, \p_{y}^{j}] \Phi  \rangle,
  \end{align}
  where we used that 
  \begin{align}
    |\langle \p_{y}^{j}\Phi,\p_{y}^{j} W  \rangle | = |\langle \p_{y}^{j}\Phi, (-1+ (\frac{\p_{y}}{k}-it)^{2})\p_{y}^{j}\Psi \rangle | \leq \|\p_{y}^{j}\Psi\|_{\tilde H^{1}}\|\p_{y}^{j}\Phi\|_{\tilde H^{1}}.
  \end{align}
  In order to estimate the contribution of the commutator, 
  \begin{align}
    \label{eq:commterms}
    [(g(\frac{\p_{y}}{k}-it))^{2}, \p_{y}^{j}] \Phi,
  \end{align}
  we note that at least one of the derivatives $\p_{y}^{j}$ has to fall on the coefficient function $g$. Hence, \eqref{eq:commterms} can be expressed in terms of 
  \begin{align*}
    \p_{y}^{j'}\Phi, \ (g(\frac{\p_{y}}{k}-it))  \p_{y}^{j'}\Phi 
  \end{align*}
  and 
  \begin{align}
    \label{eq:111}
    (g(\frac{\p_{y}}{k}-it))^{2} \p_{y}^{j'}\Phi,
  \end{align}
  with $j' \leq j-1$.
  Integrating $(\frac{\p_{y}}{k}-it)$ by parts in the case \eqref{eq:111}, \eqref{eq:110} is thus further estimated by
  \begin{align}
\label{eq:300}
    \|\p_{y}^{j}\Phi\|_{\tilde H^{1}}^{2} \lesssim \|\p_{y}^{j}\Psi\|_{\tilde H^{1}}\|\p_{y}^{j}\Phi\|_{\tilde H^{1}} + C(g) \|\p_{y}^{j}\Phi\|_{\tilde H^{1}} \sum_{j'\leq j-1} \|\p_{y}^{j'}\Phi\|_{\tilde H^{1}},
  \end{align}
where $C(g)$ depends on all derivatives of $g$ up to order $j$.

  Dividing \eqref{eq:300} by $\|\p_{y}^{j}\Phi\|_{\tilde H^{1}}$ hence proves the induction step, \eqref{eq:inductionstep}, and concludes our proof.
\end{proof}

As we discuss in Section \ref{sec:consistency-outlook}, Theorem \ref{thm:iterated} in particular provides a uniform control of 
\begin{align*}
  \|W\|_{L^{2}_{x}H^{2}_{y}(\T_{L}\times \R)},
\end{align*}
and hence allows us to close our strategy and thus prove linear inviscid damping with the optimal decay rates for a large class of monotone shear flows in an infinite periodic channel. Furthermore, as discussed in Section \ref{sec:damping}, as a consequence of sufficiently fast damping,  we obtain a scattering result via Duhamel's formula.

Prior to this, however, we in the next Section \ref{sec:asympt-stab-with} prove a similar stability result in the case of a finite channel $\T_{L}\times [0,1]$ with impermeable walls. There, boundary effects are shown to have a non-negligible effect on the dynamics.

\section{Asymptotic stability for a finite channel}
\label{sec:asympt-stab-with}

Inspired by the Fourier proof in the whole space case, in the following we establish stability in the setting of a finite periodic channel $\T_{L}\times [a,b]$.
The physically natural boundary conditions in this setting are that the boundary in $y$ is impermeable:
\begin{align}
  v_{2}=0, \quad \text{for } y \in \{a,b\}.
\end{align}
As the stream function $\phi$ satisfies 
\begin{align*}
  v_{2}=\p_{x}\phi,
\end{align*}
this, in particular, implies that $\phi$ restricted to the boundary only depends on time.

\begin{figure*}[h!]
  \centering
  \includegraphics[page=5]{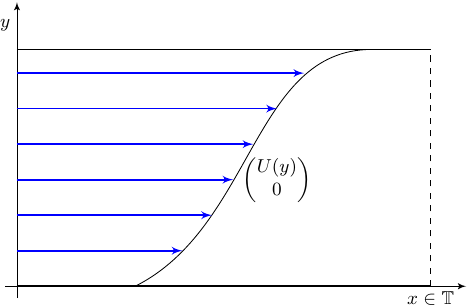}
  \caption{An example of a shear flow in a finite periodic channel.}
\end{figure*}

Following the same reduction steps as in Section \ref{sec:const-coeff-model}, in particular removing the mean $\langle W \rangle_{x}$, $\phi$ and thus $\Phi$ vanishes identically on the boundary.
The linearized Euler equations in scattering formulation are hence given by
\begin{align}
  \label{eq:2}
  \begin{split}
  \dt W &= \frac{if(y)}{k}\Phi, \\
  (-1+(g(y)(\frac{\p_{y}}{k} -it))^{2}) \Phi &= W,\\
  \Phi|_{y=U(a),U(b)}&=0, \\
  (t,k,y) &\in  \R \times L(\Z\setminus \{0\}) \times [U(a),U(b)].
  \end{split}
\end{align}
In order to simplify notation, we translate in $y$ and rescale $L$ by a factor (using Galilean symmetry and the scaling symmetry of the (linearized) Euler equations) to reduce to $[U(a),U(b)]=[0,1]$.

As in Section \ref{sec:asympt-stab-y} (c.f. Remark \ref{rem:k_is_a_parameter}), the equations \eqref{eq:2} decouple with respect to $k$.
Hence, in the following we again consider $k$ as a given parameter and write $W(t) \in H^{s}$ to denote that
\begin{align*}
  W(t,k,\cdot) \in H^{s}([0,1]).
\end{align*}

Our main result is given by the following theorem and proved in Section \ref{sec:h2-case}.
\begin{thm}
\label{thm:finsum}
  Let $W$ be a solution of \eqref{eq:2}, $f,g \in W^{3,\infty}([0,1])$ and  suppose that there exists $c>0$ such that
 \begin{align*}
   0<c<g<c^{-1}<\infty.
 \end{align*}
 Suppose further that 
 \begin{align*}
    \|f\|_{W^{1,\infty}} L 
 \end{align*}
 is sufficiently small.

 Then, for any $\omega_{0} \in H^{2}([0,1])$ with $\omega_{0}|_{y=0,1}=0$ and for any time $t$,
\begin{align*}
\|W(t)\|_{H^{2}} \lesssim \|\omega_{0}\|_{H^{2}}.
\end{align*}
\end{thm}

As we show in the following, the case of a finite channel is not only technically more involved, due to the lack of Fourier methods as well as the loss of the multiplier structure for $\Phi$ (even for Couette flow), but the qualitative behavior also changes due to boundary effects.

When differentiating the equation, $\p_{y}^{n}\Phi$ satisfies non-zero Dirichlet boundary conditions.
 Computing the boundary conditions explicitly, we, in particular, show asymptotic $H^{2}$ stability is possible if and only if $\omega_{0}$ satisfy zero Dirichlet conditions, $\omega_{0}|_{y=0,1}=0$. Higher Sobolev norms in turn would require even stronger conditions, as we discuss in Appendix \ref{sec:zero-boundary-values}.

As the damping results provide the sharp algebraic decay rates already for $H^{2}$ regularity, we restrict ourselves to considering only $L^{2}, H^{1}$ and $H^{2}$ stability.

\subsection{$L^{2}$ stability via shearing}
\label{sec:l2-stability-via-1}

As in Section \ref{sec:reduct-gener-monot}, we consider the linearized Euler equations, \eqref{eq:2}, this time in the finite periodic channel, $\T_{L}\times [0,1]$,
 \begin{align}
   \label{eq:303}
     \begin{split}
   \dt W &= \frac{if(y)}{k}\Phi, \\
   \left(-1+\left(g(y)\left(\frac{\p_{y}}{k} -it\right)\right)^{2}\right) \Phi &= W,\\
   \Phi|_{y=0,1}&=0, \\
   (t,k,y) &\in  \R \times L(\Z\setminus \{0\}) \times [0,1],
       \end{split}
 \end{align}
and additionally introduce the constant coefficient stream function $\Psi$
\begin{align}
\begin{split}
   \left(-1+\left(\frac{\p_{y}}{k} -it\right)^{2}\right) \Psi &= W,\\
  \Psi_{y=0,1}&=0.
  \end{split}
\end{align}

As in Definition \ref{defi:CCstreamfct} of Section \ref{sec:reduct-gener-monot}, we introduce constant coefficient stream functions for a given right-hand-side, where additionally prescribe boundary conditions:
\begin{defi}[Constant coefficient stream function for a finite periodic channel]
\label{defi:ccstreamfinite}
  Let $k \in L (\Z \setminus \{0\})$ and let $R(t) \in L^{2}([0,1])$ be a given function. Then the \emph{constant coefficient stream function}, $\Psi[R](t)$, is defined as the solution of 
\begin{align}
  \begin{split}
  (-1+(\frac{\p_{y}}{k}-it)^{2})\Psi[R](t,y) &= R(t,y), \\
  \Psi[R](t,y)|_{y=0,1}&=0.
  \end{split}
\end{align}
Let further $W$ be a solution of \eqref{eq:303}, then for any $k,t$, we define
\begin{align}
  \Psi(t,k,y):= \Psi[W(t,k,\cdot)](t,y).
\end{align}
\end{defi}
If we considered periodic boundary conditions, in a Fourier expansion, $\Psi[\cdot]$ would again be given by a multiplier and could be estimated explicitly in the same way as in the setting of an infinite periodic channel, $\T_{L} \times \R$.
As we however have zero Dirichlet conditions, we can not anymore solve the evolution of a constant coefficient model explicitly, but rather have to establish control of boundary effects and growth of norms, using more indirect methods.
Thus, stability results are already non-trivial even for a constant coefficient model.

Emulating the proof of the $L^{2}$ stability with a decreasing weight $A(t)$ as in Section \ref{sec:reduct-gener-monot}, a natural replacement for the Fourier transform is given by the expansion in an $L^{2}$-basis $(e_{n})$.

In view of our zero Dirichlet conditions a natural choice of such a basis is 
\begin{align*}
  \sin(ny), n \in \N.
\end{align*}
For the current purpose of $L^{2}$ stability, however, it is advantageous to instead consider an expansion in the Fourier basis 
\begin{align*}
  e^{iny}, n \in 2 \Z,
\end{align*}
for which calculations greatly simplify, at the cost of worse mapping properties in higher Sobolev spaces.
This trade-off and the role of the choice of basis is discussed in more detail in Appendix \ref{sec:test}.
\\

In the following we introduce several lemmata, which allow us to prove $L^{2}$ stability in Theorem \ref{thm:L2finite}:
\begin{itemize}
\item Lemma \ref{lem:CC2} provides a definition of a decreasing weight $A$, as in Theorem \ref{thm:CCestimate}, and proves that the constant coefficient stream function $\Psi$ can be controlled in terms of this weight.
In the case of an infinite channel as in Section \ref{sec:asympt-stab-y}, this result immediately followed from the explicit Fourier characterization.
In the setting of a finite channel, however, additional boundary effects have to be controlled, which is accomplished by the basis computations in Lemmata \ref{lem:Exp1} and \ref{lem:CC}.    
\item Lemma \ref{lem:reduction} provides an estimate of $\Phi$ in terms of $\Psi$ and hence a reduction similar to Lemma \ref{lem:phipsi} of Section \ref{sec:reduct-gener-monot}.
\end{itemize}

\begin{lem}
\label{lem:Exp1}
Let $n \in 2\pi \Z$ and let $\Psi[e^{iny}]$ be given by Definition \ref{defi:ccstreamfinite}, i.e. let $\Psi[e^{iny}]$
be the solution of 
  \begin{align*}
    (-k^{2}+(\p_{y}-ikt)^{2})\Psi[e^{iny}] &= e^{iny} ,\\
    \Psi[e^{iny}]|_{y=0,1}&=0.
  \end{align*}
Then, for any $m \in 2\pi\Z$,
  \begin{align*}
    \langle \Psi[e^{iny}], e^{imy} \rangle = \frac{\delta_{nm}}{k^{2}+(n-kt)^{2}} + \frac{k}{(k^{2}+(m-kt)^{2})(k^{2}+(n-kt)^{2})}(a-b),
  \end{align*}
  where $a,b$ solve
  \begin{align*}
     \begin{pmatrix}
    e^{k+ikt} & e^{-k+ikt} \\ 1 & 1
  \end{pmatrix}
  \begin{pmatrix}
    a \\ b
  \end{pmatrix}
  =-
  \begin{pmatrix}
    1 \\ 1
  \end{pmatrix}.
  \end{align*}
\end{lem}

\begin{lem}
\label{lem:CC}
Let $\Psi, W \in L^{2}$ solve 
  \begin{align*}
    (-k^{2}+(\p_{y}-ikt)^{2})\Psi &= W ,\\
    \Psi|_{y=0,1}&=0.
  \end{align*}
Denote the basis expansion of $W$ with respect to $e^{iny}$, $n \in 2\pi \Z$, by
\begin{align*}
  W(y)=\sum_{n} W_{n} e^{iny}.
\end{align*}
Then $W$ satisfies
\begin{align*}
  |\langle W, \Psi\rangle | \lesssim k^{-2} \sum_{n} <\frac{n}{k}-t>^{-2} |W_{n}|^{2} .
\end{align*}
\end{lem}

\begin{lem}
  \label{lem:CC2}
  Define the operator $A(t)$ by 
  \begin{align*}
    A(t): e^{iny} \mapsto \exp\left(-\int^{t} <\frac{n}{k}-\tau>^{-2} d\tau \right) e^{iny} = \exp\left(\arctan\left(\frac{n}{k}-t\right)\right) e^{iny}.
  \end{align*}
  Then $A:L^{2}\rightarrow L^{2}$ is a uniformly bounded, symmetric, positive operator and satisfies
  \begin{align*}
    \|W\|_{L^{2}}^{2} \lesssim \langle W, A W \rangle \lesssim \|W\|_{L^{2}}^{2},
  \end{align*}
  where the estimates are uniform in $t$.
  Furthermore, the time derivative $\dot A$ is symmetric and non-positive 
 and there exists a constant $C>0$ such that, for $\Psi$ as in Definition \ref{defi:ccstreamfinite},  
  \begin{align*}
    |\langle W, A \Psi \rangle| C + \langle W, \dot A W \rangle \leq 0. 
  \end{align*}
\end{lem}

\begin{lem}
  \label{lem:reduction}
  Let $W \in L^{2}$, $0<c<g<c^{-1}<\infty$, $\frac{1}{g(y)},f(y)\in W^{1,\infty}$ and $A$ as in Lemma \ref{lem:CC2}. Let $W,\Phi, \Psi$
 solve 
  \begin{align}
    \label{eq:PhiforWfin}
    \begin{split}
    (k^{2}+(g(\p_{y}-ikt))^{2}) \Phi &= W , \\
    \Phi_{y=0,1}&=0,
    \end{split} \\
    \label{eq:PsiforWfin}
    \begin{split}
    (k^{2}+(\p_{y}-ikt)^{2})\Psi &= W , \\
    \Psi_{y=0,1}&=0.
    \end{split}
  \end{align}
Then there exists a constant $C$ such that
\begin{align*}
  | \langle AW, \frac{if}{k} \Phi \rangle | \leq  \frac{C}{k}|\langle AW, \Psi \rangle|.
\end{align*}
\end{lem}

With these lemmata we can now prove $L^{2}$ stability:
\begin{thm}[$L^{2}$ stability for the finite periodic channel]
  \label{thm:L2finite}
  Let $f,g \in W^{1,\infty}([0,1])$ and suppose that there exists $c>0$, such that
  \begin{align*}
    0<c<g<c^{-1}<\infty.
  \end{align*}
  Suppose further that 
  \begin{align*}
    L \|f\|_{W^{1,\infty}}
  \end{align*}
  is sufficiently small.
  Then for all $\omega_{0}\in L^{2}$, the solution $W$ of the linearized Euler equations, \eqref{eq:2}, with initial datum $\omega_{0}$, for any time $t$, satisfies
  \begin{align*}
    \|W(t)\|_{L^{2}} \lesssim \|\omega_{0}\|_{L^{2}}.
  \end{align*}
\end{thm}

\begin{proof}[Proof of Theorem \ref{thm:L2finite}]
    The time derivative of $I(t):=\langle W, AW \rangle$ is controlled by
  \begin{align*}
    2| \langle AW, \frac{if}{k} \Phi \rangle| + \langle W, \dot A W \rangle .
  \end{align*}
  By Lemma \ref{lem:reduction} there exists a constant $C_{1}$, such that
  \begin{align*}
    \dot I (t) \leq \frac{C_{1}}{|k|}|\langle AW, \Psi \rangle| + \langle W, \dot A W \rangle .
  \end{align*}
  Requiring $|k|$ to be sufficiently large, $\frac{C_{1}}{|k|}\leq C$ . Thus, Lemma \ref{lem:CC2} yields
  \begin{align*}
    \dot I(t) \leq |\langle W, A \Psi[W] \rangle| C + \langle W, \dot A W \rangle \leq 0. 
  \end{align*}
In particular,
\begin{align*}
  \|W(t)\|_{L^{2}}^{2} \lesssim I(t) \leq I(0) \lesssim \|\omega_{0}\|_{L^{2}}^{2}.
\end{align*}
\end{proof}

It remains to prove the previously stated Lemmata \ref{lem:Exp1}-\ref{lem:reduction}.

\begin{proof}[Proof of Lemma \ref{lem:Exp1}]
The constant coefficient stream function for $e^{iny}$ is given by
  \begin{align*}
    \Psi[e^{iny}]&= \frac{1}{k^{2}+(n-kt)^{2}}(e^{iny}+ a e^{ky+ikty}+ b e^{-ky+ikty}),
  \end{align*}
  where $a,b$ solve 
  \begin{align*}
  \begin{pmatrix}
    e^{k+ikt} & e^{-k+ikt} \\ 1 & 1
  \end{pmatrix}
  \begin{pmatrix}
    a \\ b
  \end{pmatrix}
  &= -
  \begin{pmatrix}
    1 \\ 1
  \end{pmatrix}.
  \end{align*}
  Integrating against another basis function $e^{imy}$, we obtain:
\begin{align*}
  \langle \Psi[e^{iny}],e^{imy} \rangle &= \frac{1}{k^{2}+(n-kt)^{2}}\Big( \delta_{nm} + \frac{\left. e^{ky+ikty}\right|_{y=0}^{1}}{k+i(kt-m)} a +  \frac{\left. e^{ky+ikty}\right|_{y=0}^{1} }{-k+i(kt-m)} \Big)  \\
&=\frac{1}{k^{2}+(n-kt)^{2}} \Big(  \delta_{nm} + \frac{k\left. e^{ky+ikty}\right|_{y=0}^{1}}{k^{2}+(kt-m)^{2}} a - \frac{k\left. e^{ky+ikty}\right|_{y=0}^{1}}{k^{2}+(kt-m)^{2}} b\\
& \quad - \frac{i(kt-m)}{k^{2}+(kt-m)^{2}}
\begin{pmatrix}
  -1 \\ 1
\end{pmatrix}
  \begin{pmatrix}
    e^{k+ikt} & e^{-k+ikt} \\ 1 & 1
  \end{pmatrix}
  \begin{pmatrix}
    a \\ b
  \end{pmatrix} \Big)\\
&= \frac{\delta_{nm}}{k^{2}+(n-kt)^{2}}  + \frac{k}{(k^{2}+(m-kt)^{2})(k^{2}+(n-kt)^{2})} (a-b).
\end{align*}
\end{proof}

\begin{proof}[Proof of Lemma \ref{lem:CC}]
  Using Lemma \ref{lem:Exp1}, we expand $\langle W, \Psi \rangle$ in our basis and explicitly compute:
  \begin{align*}
    \langle W, \Psi \rangle &= \sum_{n,m} \overline{W}_{m} \langle e^{imy}, \Psi[e^{iny}] \rangle W_{n}
\\ &=\sum_{n,m} \overline{W}_{m} \left(\frac{\delta_{nm}}{k^{2}+(n-kt)^{2}} + \frac{k}{(k^{2}+(m-kt)^{2})(k^{2}+(n-kt)^{2})}(a-b) \right) W_{n} \\
    &= \sum_{n} \frac{1}{k^{2}+(n-kt)^{2}}|W_{n}|^{2} + k(a-b) \left( \sum_{n} \frac{1}{k^{2}+(n-kt)^{2}}W_{n} \right)\left( \sum_{m} \frac{1}{k^{2}+(m-kt)^{2}}\overline {W}_{m} \right) \\
  &\leq  \left(\sum_{n} \frac{|W_{n}|^{2}}{k^{2}+(n-kt)^{2}}\right) \left(1+ |k(a-b)| \left\|\frac{1}{\sqrt{k^{2}+(m-kt)^{2}}}\right\|_{l^{2}_{m}}^{2}\right) \\
&\lesssim \sum_{n} \frac{|W_{n}|^{2}}{k^{2}+(n-kt)^{2}}.
  \end{align*}
\end{proof}

\begin{proof}[Proof of Lemma \ref{lem:CC2}]
  Expressed in the Fourier basis, $e^{iny}$, $A(t)$ is a diagonal operator with positive, monotonically decreasing coefficients that are uniformly bounded from above and below by
$\exp(\pm \|<t>^{-2}\|_{L^{1}_{t}})= e^{\pm \pi}$. 
It remains to show
  \begin{align*}
    |\langle W, A \Psi[W] \rangle| C + \langle W, \dot A W \rangle \leq 0 .
  \end{align*}
Modifying the proof of Lemma \ref{lem:CC} slightly, we obtain that
\begin{align*}
   |\langle W, A \Psi[W] \rangle| &\lesssim \sum <\frac{n}{k}-t>^{-2} |W_{n}|^{2} \\ &\lesssim \sum <\frac{n}{k}-t>^{-2}\exp\left( \int^{t} <\frac{n}{k}-\tau>^{-2} d\tau\right) |W_{n}|^{2} \\ &= - \langle W, \dot A W \rangle .
\end{align*}
\end{proof}

\begin{proof}[Proof of Lemma \ref{lem:reduction}]
  Let $\Psi[AW]$ solve 
  \begin{align*}
    (-1+(\frac{\p_{y}}{k}-it)^{2})\Psi[AW] &= AW , \\
    \Psi[AW]_{y=0,1} &= 0 .
  \end{align*}
By integration by parts, we then obtain 
\begin{align*}
  |\langle AW, \frac{if}{k}\Phi \rangle | &\leq  \sqrt{ \|\Psi[AW]\|_{L^{2}}^{2} + \|(\frac{\p_{y}}{k}-it)\Psi[AW]\|_{L^{2}}^{2}} \|f\|_{W^{1,\infty}} \frac{1}{|k|} \sqrt{ \|\Phi\|_{L^{2}}^{2} + \|(\frac{\p_{y}}{k}-it)\Phi\|_{L^{2}}^{2}} \\
&=:  \|\Psi[AW]\|_{\tilde H^{1}}  \|f\|_{W^{1,\infty}} \frac{1}{|k|} \|\Phi\|_{\tilde H^{1}}.
\end{align*}
By our basis characterization, and as $A$ is a bounded, positive multiplier on our basis,
\begin{align*}
  \|\Psi[AW]\|_{L^{2}}^{2} + \|(\frac{\p_{y}}{k}-it)\Psi[AW]\|_{L^{2}}^{2} = |\langle AW, \Psi[AW]  \rangle| \lesssim \|A\| |\langle AW, \Psi \rangle|,
\end{align*}
so it only remains to control the factors involving $\Phi$.
Testing \eqref{eq:PhiforWfin} with $-\frac{1}{g}\Phi$ and using \eqref{eq:PsiforWfin}, we obtain:
\begin{align*}
  \|\Phi\|_{L^{2}}^{2} + \|(\frac{\p_{y}}{k}-it)\Phi\|_{L^{2}}^{2} &\lesssim -\Re \langle W,  \frac{1}{g} \Phi \rangle
=\Re \langle (1-(g(\frac{\p_{y}}{k}-it))^{2})\Phi, \frac{1}{g}\Phi \rangle \\
  &\lesssim \|\frac{1}{g}\|_{W^{1,\infty}} \|\Psi\|_{\tilde H^{1}} \|\Phi\|_{\tilde H^{1}}.
\end{align*}
Dividing by $\|\Phi\|_{\tilde H^{1}}:=\sqrt{\|\Phi\|_{L^{2}}^{2} + \|(\frac{\p_{y}}{k}-ikt)\Phi\|_{L^{2}}^{2}}$, then provides the desired estimate.
\end{proof}

This concludes our proof of Theorem \ref{thm:L2finite} and thus establishes $L^{2}$ stability for a large class of strictly monotone shear flows in a finite periodic channel.
Unlike the setting of an infinite periodic channel, where in Section \ref{sec:iter-arbitr-sobol} the $L^{2}$ stability results could be extended to arbitrarily high Sobolev norms, in the following subsections we show that boundary effects introduce additional correction terms (even in the constant coefficient model), which qualitatively change the stability behavior of the equations.

\subsection{$H^1$ stability}
\label{sec:h1-case}


In order to extend the stability results to $H^{1}$, we proceed as in Section \ref{sec:iter-arbitr-sobol} and differentiate the linearized Euler equations for a finite periodic channel, \eqref{eq:303}. We note, that $\p_{y}\Psi$ and $\p_{y}\Phi$ do not anymore satisfy zero Dirichlet boundary conditions, and thus split $\p_{y}\Phi=\Phi^{(1)}+H^{(1)}$:
\begin{align}
\label{eq:H1}
  \begin{split}
  \dt \p_{y} W &= \frac{if}{k}(\Phi^{(1)}+H^{(1)}) + \frac{if'}{k} \Phi , \\
  (-1+(g(\frac{\p_{y}}{k}-it))^{2})\Phi^{(1)} &= \p_{y}W + [(g(\p_{y}-it))^{2}, \p_{y}] \Phi, \\
  \Phi^{(1)}_{y=0,1} &= 0 , \\
  H^{(1)}&= \p_{y}\Phi -\Phi^{(1)}.
  \end{split}
\end{align}
The \emph{homogeneous correction}, $H^{(1)}$, hence satisfies
\begin{align}
  \label{eq:H1correction}
  \begin{split}
  (-1+(g(\frac{\p_{y}}{k}-it))^{2})H^{(1)} &=0, \\
  H^{(1)}|_{y=0,1}&= \p_{y}\Phi_{y=0,1}.
  \end{split}
\end{align}

The control of the contributions by $\Phi$ and $\Phi^{(1)}$ is obtained as in Section \ref{sec:l2-stability-via-1}, while the control of the boundary corrections due to $H^{(1)}$ is given by the following lemmata.

\begin{lem}[$H^{1}$ boundary contributions]
\label{lem:H11}
  Let $A(t)$ be a diagonal operator comparable to the identity, i.e. 
  \begin{align*}
    A(t)&: e^{iny} \mapsto A_{n}(t) e^{iny}, \\
    1 &\lesssim A_{n}(t)\lesssim 1,
  \end{align*}
  $\omega_{0} \in H^{1}, f,g \in W^{2,\infty}$ and suppose that $0<c<g<c^{-1}<\infty$.
  Let further $W$ be the solution of \eqref{eq:H1}.

  Then, for any $0<\gamma,\beta<\frac{1}{2}$ there exists a constant $C=C(\gamma,\beta,\|f\|_{W^{2,\infty}},c, \|g\|_{W^{2,\infty}})$, such that
  \begin{align*}
    | \langle A \p_{y}W, f H^{(1)} \rangle| \leq C<t>^{-2(1-\gamma)}\|\omega_{0}\|_{H^{1}}^{2} + C\sum_{n} <t>^{-2\gamma}<\frac{n}{k}-t>^{-2\beta} |(\p_{y}W)_{n}|^{2} .
  \end{align*}

If additionally $\omega_{0}|_{y=0,1} \equiv 0$, then for any $0<\beta<\frac{1}{2}$ there exists a constant $C=C(\gamma,\beta,\|f\|_{W^{2,\infty}},c, \|g\|_{W^{2,\infty}})$, such that
\begin{align*}
   | \langle A \p_{y}W, f H^{(1)} \rangle| \leq C \sum_{n} <t>^{-1}<\frac{n}{k}-t>^{-2\beta} |(\p_{y}W)_{n}|^{2} .
\end{align*}
\end{lem}

\begin{lem} [$H^{1}$ stream function estimate]
\label{lem:H12}
Let $A,W,f,g$ satisfy the assumptions of Lemma \ref{lem:H11}. Then 
\begin{align*}
  |\langle A\p_{y}W, \frac{if}{k} \Phi^{(1)} + \frac{if'}{k} \Phi\rangle| \lesssim  |k|^{-1}\|f\|_{W^{2,\infty}} (\|\Psi[A\p_{y}W]\|_{\tilde H^{1}}^{2} +\|\Psi[\p_{y}W]\|_{\tilde H^{1}}^{2} +\|\Psi[W]\|_{\tilde H^{1}}^{2} ).
\end{align*}
\end{lem}

Using Lemmata \ref{lem:H11}, \ref{lem:H12} and the lemmata of Section \ref{sec:l2-stability-via-1}, we prove $H^{1}$ stability.

\begin{thm}[$H^{1}$ stability for the finite periodic channel]
\label{thm:H1short}
  Let $W$ be a solution of the linearized Euler equations, \eqref{eq:H1}, and suppose that $f,g \in W^{2,\infty}$ and that there exists $c>0$, such that
  \begin{align*}
 0<c<g<c^{-1}<\infty.
  \end{align*}
  Further define a diagonal weight $A(t)$:
  \begin{align}
    \begin{split}
    A(t)&:e^{iny} \mapsto A_{n}(t)e^{iny}, \\
    A_{n}(t)&= \exp \left( -\int^{t}_{0} <\frac{n}{k}-\tau>^{-2} + <\tau>^{-2\gamma} <\frac{n}{k}-\tau>^{-2\beta} d\tau\right),
    \end{split}
  \end{align}
  where $0<\beta,\gamma<\frac{1}{2}$ and $2\gamma+2\beta>1$.
  Also suppose that
  \begin{align*}
    \|f\|_{W^{2,\infty}} L 
  \end{align*}
  is sufficiently small.
  Then, for any $\omega_{0}\in H^{1}([0,1])$, the solution $W$ of \eqref{eq:303} (and hence \eqref{eq:H1}) with initial datum $\omega_{0}$, satisfies 
\begin{align*}
\|W(t)\|_{H^{1}}^{2} \lesssim I(t):=\langle A(t)W, W\rangle+\langle A(t)\p_{y}W,\p_{y}W \lesssim I(0) \lesssim \|\omega_{0}\|_{H^{1}}.
\end{align*}
  If additionally $\omega_{0}|_{y=0,1}\equiv 0$, then $I(t)$ is non-increasing. 
\end{thm}

\begin{proof}[Proof of Theorem \ref{thm:H1short}]
  Let $W$ be a solution of \eqref{eq:H1}, then we compute
\begin{align*}
  \frac{d}{dt} \langle \p_{y}W, A \p_{y}W \rangle = \langle \dot A \p_{y}W, \p_{y}W \rangle 
+ 2\Re \langle A\p_{y}W, \frac{if}{k}\Phi^{(1)} + \frac{if'}{k} \Phi  \rangle
+ 2\Re \langle A\p_{y}W, \frac{if}{k}H^{(1)}\rangle .
\end{align*}
Using Lemma \ref{lem:H12} in combination with Lemma \ref{lem:CC}, we estimate the second term by:
\begin{align*}
& \quad 2\Re \langle A\p_{y}W, \frac{if}{k}\Phi^{(1)} + \frac{if'}{k} \Phi  \rangle \\
& \lesssim \frac{\|f\|_{W^{2,\infty}}}{|k|} (\|\Psi[A\p_{y}W]\|_{\tilde H^{1}}^{2} +\|\Psi[\p_{y}W]\|_{\tilde H^{1}}^{2} +\|\Psi[W]\|_{\tilde H^{1}}^{2} ) \\ 
&\lesssim \frac{\|f\|_{W^{2,\infty}}}{|k|}  (| \langle W, \dot A W \rangle| + | \langle \p_{y}W, \dot A \p_{y}W \rangle|).
\end{align*}
Using Lemma \ref{lem:H11}, the last term is controlled by:
\begin{align*}
& \quad 2\Re \langle A\p_{y}W, \frac{if}{k}H^{(1)}\rangle \\
&\lesssim 
   \frac{\|f\|_{W^{2,\infty}}}{|k|} <t>^{-2(1-\gamma)}\|\omega_{0}\|_{H^{1}}^{2} +  \frac{\|f\|_{W^{2,\infty}}}{|k|}\sum_{n} <t>^{-2\gamma}<\frac{n}{k}-t>^{-2\beta} |(\p_{y}W)_{n}|^{2},
\end{align*}
or by 
\begin{align*}
 C_{1} \frac{\|f\|_{W^{2,\infty}}}{|k|} \sum_{n} <t>^{-1}<\frac{n}{k}-t>^{-2\beta} |(\p_{y}W)_{n}|^{2}, 
\end{align*}
if $\omega_{0}|_{y=0,1}\equiv 0$. 
\\

Hence, for
\begin{align*}
  \sup_{k \neq 0} \frac{\|f\|_{W^{2,\infty}}}{|k|}
\end{align*}
sufficiently small,
\begin{align*}
   2\Re \langle A\p_{y}W, \frac{if}{k}\Phi^{(1)} + \frac{if'}{k} \Phi  \rangle
+ 2\Re \langle A\p_{y}W, \frac{if}{k}H^{(1)}\rangle
\end{align*}
can be absorbed by
\begin{align*}
\langle  \dot A \p_{y}W, \p_{y}W \rangle = - \sum_{n} A_{n}(t) \left( <\frac{n}{k}-t>^{-2}+ <t>^{-2\gamma} <\frac{n}{k}-t>^{-2\beta} \right) |(\p_{y}W)_{n}|^{2} \leq 0. 
\end{align*}
Thus, $I(t)$ satisfies
\begin{align*}
  \frac{d}{dt} I(t) \lesssim <t>^{-2(1-\gamma)}\|\omega_{0}\|_{H^{1}}^{2},
\end{align*}
or, in the case of vanishing Dirichlet data, $\omega_{0}|_{y=0,1}=0$,
\begin{align*}
  \frac{d}{dt} I(t) \leq 0.
\end{align*}
Integrating these inequalities in time concludes the proof.
\end{proof}

\begin{proof}[Proof of Lemma \ref{lem:H11}]
Similar to the construction of Lemma \ref{lem:Exp1}, let $u_{j}$, $j=1,2$, be solutions of
\begin{align*}
  (-1+(g(\frac{\p_{y}}{k}-it))^{2})u_{j} &= 0
\end{align*}
with boundary values
\begin{align}
\label{eq:boundaryvaluesuj}
  \begin{split}
  u_{1}(0)=u_{2}(1)&=1, \\
  u_{1}(1)= u_{2}(0)&=0.
  \end{split}
\end{align}
Recalling the sequence of transformations turning $\phi$ into $\Phi$, the functions $u_{j}$ are given by linear combinations of the homogeneous solutions
\begin{align*}
  e^{\pm k G(y)+ ikty},
\end{align*}
where $G(y)= U^{-1}(y)$ satisfies $G(y)' = g(y)$.

Further recalling the boundary conditions in \eqref{eq:H1correction}, $H^{(1)}$ is hence given by
\begin{align*}
  H^{(1)}= \p_{y}\Phi (0) u_{1} + \p_{y}\Phi(1) u_{2}.
\end{align*}

In order to compute $\p_{y}\Phi|_{y=0,1}$, we test the equation for $\Phi$ in \eqref{eq:303}, i.e.
\begin{align}
  \begin{split}
    (-1+(g(y)(\frac{\p_{y}}{k} -it))^{2}) \Phi &= W,\\
   \Phi|_{y=0,1}&=0,
  \end{split}
\end{align}
 with $u_{j}$:
\begin{align*}
  \langle W, u_{j} \rangle &= \langle (-1+(g(\frac{\p_{y}}{k}-it))^{2}) \Phi, u_{j} \rangle \\
&= \left. u_{j} g (\frac{\p_{y}}{k}-it)(g \Phi) \right|_{y=0}^{1} - \left. g \Phi (\frac{\p_{y}}{k}-it) (g u_{j})  \right|_{y=0}^{1} \\
&\quad + \langle \Phi, (-1+(g(\frac{\p_{y}}{k}-it))^{2}) u_{j} \rangle \\
&=\left. u_{j} \frac{g^{2}}{k} \p_{y} \Phi  \right|_{y=0}^{1},
\end{align*}
where we used that $\Phi|_{y=0,1}=0$.
Using the boundary values of $u_{j}$, \eqref{eq:boundaryvaluesuj}, 
\begin{align*}
  \left. u_{1} \frac{g^{2}}{k} \p_{y} \Phi  \right|_{y=0}^{1} &= - \frac{g^{2}(0)}{k}\p_{y}\Phi|_{y=0}, \\
  \left. u_{2} \frac{g^{2}}{k} \p_{y} \Phi  \right|_{y=0}^{1} &= \frac{g^{2}(1)}{k}\p_{y}\Phi|_{y=1}.
\end{align*}
As $k \neq 0$  and $g^{2} >c> 0$, we may solve for $\p_{y}\Phi|_{y=0,1}$:
\begin{align}
\label{eq:400}
  H^{(1)}= \frac{k}{g^{2}(0)} \langle W, u_{1} \rangle u_{1} - \frac{k}{g^{2}(1)} \langle W, u_{2} \rangle u_{2}.
\end{align}

The boundary contribution can thus be explicitly computed in terms of $u_{1},u_{2}$:
\begin{align*}
  \langle A \p_{y}W, fH^{(1)} \rangle =  \frac{k}{g^{2}(0)} \langle W, u_{1} \rangle \langle A \p_{y}W, f u_{1} \rangle 
- \frac{k}{g^{2}(1)} \langle W, u_{2} \rangle \langle A \p_{y}W, f u_{2} \rangle.
\end{align*}

As the homogeneous solutions $e^{\pm kG(y)+ikty}$ and thus $u_{1},u_{2}$ are highly oscillatory, we integrate 
$k \langle W, u_{j}  \rangle$ by parts and use that the evolution of \eqref{eq:2} preserves boundary values, i.e. $W|_{y=0,1}= \omega_{0}|_{y=0,1}$. Denoting primitive functions of $u_{j}$ by $U_{j}$ and using that
\begin{align*}
  e^{\pm kG(y)+ikty} = \frac{1}{\pm kg +ikt} \p_{y}e^{\pm kG(y)+ikty},
\end{align*}
we therefore obtain
\begin{align}
\label{eq:401}
  \begin{split}
  k \langle W, u_{j}  \rangle &= kU_{j}\omega_{0}|_{y=0}^{1} - \langle \p_{y}W, kU_{j}\rangle\\
\leq & \mathcal{O}(t^{-1})(\|\omega_{0}\|_{H^{1}}+ |\langle \p_{y}W , u_{1}\rangle|+|\langle \p_{y}W , u_{2}\rangle| ).
  \end{split}
\end{align}
Using Young's inequality, this yields a bound by
\begin{align}
\label{eq:H1boundary}
\begin{split}
 \left|\langle A \p_{y}W, fH^{(1)} \rangle \right| \lesssim & <t>^{-1}( |\langle \p_{y}W , u_{1}\rangle|^{2} +|\langle A \p_{y}W, f u_{j} \rangle|^{2})  \\ &+<t>^{-1}|\langle A \p_{y}W, f u_{j} \rangle| \|\omega_{0}\|_{H^{1}} \\
\lesssim & <t>^{-2\gamma} ( |\langle \p_{y}W , u_{1}\rangle|^{2} +|\langle A \p_{y}W, f u_{j} \rangle|^{2}) \\ &+ <t>^{-2(1-\gamma)}\|\omega_{0}\|_{H^{1}}^{2},
\end{split}
\end{align}
where $0<\gamma<\frac{1}{2}$ is chosen close to $\frac{1}{2}$.

Expanding $\p_{y}W$ in our basis and choosing $0<\beta< \frac{1}{2}$ close to $\frac{1}{2}$, we further estimate
\begin{align*}
  |\langle \p_{y} W, u_{j}\rangle | &\lesssim  \sum_{n}|(\p_{y}W)_{n}| |\langle e^{iny}, u_{j} \rangle |
\lesssim \sum_{n} |(\p_{y}W)_{n}| \frac{1}{|k+i(n-kt)|} \\
 &\leq  \frac{1}{k} \|(\p_{y}W)_{n} <\frac{n}{k}-t>^{-\beta}\|_{l^{2}_{n}} \| <\frac{n}{k}-t>^{-1+\beta}\|_{l^{2}_{n}} \\
&\lesssim_{\beta}  \|(\p_{y}W)_{n} <\frac{n}{k}-t>^{-\beta}\|_{l^{2}}. 
\end{align*}
A similar bound also holds for $\langle A \p_{y}W, f u_{j} \rangle$, where the constant further includes a factor $\|f\|_{W^{1,\infty}}$.

Thus, \eqref{eq:H1boundary} can further be controlled by
\begin{align*}
  \left|\langle A \p_{y}W, fH^{(1)} \rangle \right| \lesssim <t>^{-2(1-\gamma)}\|\omega_{0}\|_{H^{1}}^{2}+ \sum_{n} <t>^{-2\gamma}<\frac{n}{k}-t>^{-2\beta}|(\p_{y}W)_{n}|^{2}.
\end{align*}

The improved result for $\omega_{0}|_{y=0,1}\equiv 0$ similarly follows from \eqref{eq:H1boundary}, as in that case
the term $<t>^{-1}|\langle A \p_{y}W, f u_{j} \rangle|$ is not present.
\end{proof}

\begin{proof}[Proof of Lemma \ref{lem:H12}]
 Using the vanishing boundary values of $\Phi$ and $\Phi^{(1)}$ and introducing  
 \begin{align*}
   (-1+(\frac{\p_{y}}{k}-it)^{2}) \Psi[A\p_{y}W]&= A\p_{y}W , \\
   \Psi[A\p_{y}W]|_{y=0,1} &= 0,
 \end{align*}
we integrate by parts to bound by
\begin{align*}
 &\left| \left\langle (-1+(\frac{\p_{y}}{k}-it)^{2}) \Psi[A\p_{y}W], \frac{if}{k}\Phi^{(1)} + \frac{if'}{k} \Phi  \right\rangle \right| \\
\leq & \left(\|\Psi\|_{L^{2}} +\|(\frac{\p_{y}}{k}-it)\Psi\|_{L^{2}}\right) 
\frac{\|f\|_{W^{2,\infty}}}{k} \left(\|\Phi\|_{L^{2}} +\|(\frac{\p_{y}}{k}-it)\Phi\|_{L^{2}}+\|\Phi^{(1)}\|_{L^{2}} +\|(\frac{\p_{y}}{k}-it)\Phi^{(1)}\|_{L^{2}}\right) \\
\leq & \frac{\|f\|_{W^{2,\infty}}}{k} (\|\Psi\|_{\tilde H^{1}}^{2} + \|\Phi\|_{\tilde H^{1}}^{2} + \|\Phi^{(1)}\|_{\tilde H^{1}}^{2}).
\end{align*}

In order to further estimate $\|\Phi^{(1)}\|_{\tilde H^{1}}$, we again use the vanishing boundary values of $\Phi^{(1)}$ and test
\begin{align*}
  (-1+(g(\frac{\p_{y}}{k}-it))^{2})\Phi^{(1)} &= \p_{y}W + [(g(\p_{y}-it))^{2}, \p_{y}] \Phi, \\
  \Phi^{(1)}_{y=0,1}&=0,
\end{align*}
with $-\frac{1}{g}\Phi^{(1)}$, to obtain that
\begin{align*}
  \|\Phi^{(1)}\|_{\tilde H^{1}}^{2} \lesssim & -\langle (-1+(g(\frac{\p_{y}}{k}-it))^{2})\Phi^{(1)}, \frac{1}{g}\Phi^{(1)} \rangle\\
\leq & -\langle (-1+(\frac{\p_{y}}{k}-it)^{2}) \Psi[\p_{y}W],\frac{1}{g}\Phi^{(1)}   \rangle + \langle [(g(\frac{\p_{y}}{k}-it))^{2}, \p_{y}] \Phi, \Phi^{(1)} \rangle  \\
\lesssim & \|\Psi[\p_{y}W]\|_{\tilde H^{1}} \|\Phi^{(1)}\|_{\tilde H^{1}} + \|\Phi\|_{\tilde H^{1}} \|\Phi^{(1)}\|_{\tilde H^{1}}.
\end{align*}
Using this inequality and Lemma \ref{lem:reduction} to estimate $\|\Phi\|_{\tilde H^{1}} \lesssim \|\Psi\|_{\tilde H^{1}} $, then concludes the proof.
\end{proof}
As a consequence of the $H^{1}$ stability result, Theorem \ref{thm:H1short}, Theorem \ref{thm:lin-zeng} of Section \ref{sec:damping} yields damping with rate $t^{-1}$, i.e. 
\begin{align*}
  \|v- \langle v \rangle_{x}\|_{L^{2}(\T_{L} \times [0,1])} \leq \mathcal{O}(t^{-1})\|W(t)\|_{L^{2}(\T_{L} \times [0,1])} &\leq \mathcal{O}(t^{-1})\|\omega_{0}\|_{L^{2}(\T_{L} \times [0,1])}, \\
   \|v_{2}\|_{L^{2}(\T_{L} \times [0,1])} &\leq \mathcal{O}(t^{-1})\|\omega_{0}\|_{L^{2}(\T_{L} \times [0,1])}.
\end{align*}
As discussed in Section \ref{sec:damping}, the first estimate thus already attains the optimal damping rate and regularity requirements.
The estimate for $v_{2}$, however, does not yet provide an integrable decay rate, $\mathcal O(t^{-1-\epsilon})$, and thus, in particular, is not sufficient to prove scattering. 

In the following section, we thus prove $H^{2}$ stability and hence linear inviscid damping with the optimal rates as well as scattering.
There, we additionally require our perturbations to satisfy zero Dirichlet boundary conditions, $\omega_{0}|_{y=0,1}=0$.

As we discuss in Appendix \ref{sec:zero-boundary-values}, this is not only a technical restriction: We show that otherwise $\p_{y}W$ asymptotically develops a logarithmic singularity at the boundary, which by the trace theorem in particular forbids stability in any Sobolev space more regular than $H^{\frac{3}{2}}_{y}$.

\subsection{$H^2$ stability}

\label{sec:h2-case}

Following a similar approach as in the previous Subsection \ref{sec:h1-case}, we obtain $H^{2}$ stability and hence linear inviscid damping with the optimal rates and scattering for a large class of monotone shear flows in a finite periodic channel.
As we discuss in Appendix \ref{sec:zero-boundary-values},
for this stability result it is necessary to restrict to perturbations with zero Dirichlet data, $\omega_{0}|_{y=0,1}=0$.

We again differentiate our equation and introduce homogeneous correction terms $H^{(1)}, H^{(2)}$.
Let thus $W$ be a solution of (\ref{eq:303}), then $\p_{y}^{2}W$ satisfies
\begin{align}
\label{eq:H2}
\begin{split}
  \dt \p_{y}^{2}W &= \frac{if}{k}(\Phi^{(2)}+H^{(2)}) + \frac{2f'}{ik} (\Phi^{(1)}+H^{(1)}) + \frac{f''}{ik}\Phi ,\\
(-1+(g(\frac{\p_{y}}{k}-it))^{2})\Phi^{(2)} &= \p_{y}^{2}W + [(g(\frac{\p_{y}}{k}-it))^{2}, \p_{y}^{2}] \Phi ,\\
\Phi^{(2)}_{y=0,1}&=0 .
\end{split}
\end{align}
Here the \emph{homogeneous correction} $H^{(2)}$ satisfies 
\begin{align*}
  (-1+(g(\frac{\p_{y}}{k}-it))^{2})H^{(2)} &= 0, \\
  H^{(2)}|_{y=0,1}&= \p_{y}^{2}\Phi|_{y=0,1}.
\end{align*}
We recall that the equations satisfied by $\p_{y}W, \Phi^{(1)}, H^{(1)}$ are given by \eqref{eq:H1} and \eqref{eq:H1correction}, respectively.

As in Section \ref{sec:h1-case}, we introduce several lemmata to control boundary corrections.
Using these lemmata, we then prove the main stability result, Theorem \ref{thm:H2stability}. 

\begin{lem}[$H^{2}$ boundary contribution I]
  \label{lem:H2boundary1}
  Let $A(t)$ be a diagonal operator comparable to the identity, i.e. 
  \begin{align*}
    A: e^{iny} &\mapsto A_{n} e^{iny}, \\
    1 &\lesssim A_{n} \lesssim 1,
  \end{align*}
 and let $W$ be a solution of \eqref{eq:H2} with initial datum $\omega_{0} \in H^{2}([0,1])$ with $\omega_{0}|_{y=0,1}=0$.
  Suppose further that $f,g \in W^{3,\infty}$ and $k$ (or $L$ respectively) satisfy the assumptions of the $H^{1}$ stability result, Theorem \ref{thm:H1short}. 
  Then $H^{(1)}$ satisfies
  \begin{align*}
    \|H^{(1)}\|_{\tilde H^{1}} ^{2} \lesssim <t>^{-2} \|W\|_{H^{1}}^{2} \lesssim <t>^{-2}\|\omega_{0}\|_{H^{1}}^{2}
  \end{align*}
  and for any $0<\beta,\gamma<\frac{1}{2}$,
  \begin{align*}
    |\langle A\p_{y}^{2}W, \frac{if'}{k} H^{(1)}\rangle | &\lesssim_{\beta,\gamma} \|f\|_{W^{2,\infty}}k^{-1}\Big{(}\log^{2}(t)<t>^{-2(1-\gamma)} \|\omega_{0}\|_{H^{2}}^{2}\\
&\quad + \sum_{n} <t>^{-2\gamma} <\frac{n}{k}-t>^{-2 \beta}|(A \p_{y}^{2}W)_{n}|^{2} \Big{)}.
  \end{align*}
\end{lem}

\begin{lem}[$H^{2}$ boundary contribution II]
\label{lem:H2boundary2}
  Let $A,f,g,W,k$ as in Lemma \ref{lem:H2boundary1}.
Then for $0<\gamma,\beta < \frac{1}{2}$ there exits a constant $C=C(f,g,k,\beta,\gamma)$, such that 
\begin{align*}
  |\langle A\p_{y}^{2}W, \frac{if}{k} H^{(2)}\rangle | &\leq C \log^{2}(t)<t>^{-2(1-\gamma)}\|\omega_{0}\|_{H^{2}}^{2} \\ &\quad + C\sum_{n} <t>^{-2\gamma}<\frac{n}{k}-t>^{-2\beta} |(\p_{y}^{2}W)_{n}|^{2} .
\end{align*}
\end{lem}

\begin{lem} [$H^{2}$ stream function estimate I]
\label{lem:H2bulk1}
Let $A,f,g,W,k$ as in Lemma \ref{lem:H2boundary1}.
Then, 
\begin{align*}
  |\langle A\p_{y}W, \frac{if}{k} \Phi^{(2)} \rangle |  \lesssim  k^{-1}\|f\|_{W^{1,\infty}} (\|\Psi[A\p_{y}^{2}W]\|_{\tilde H^{1}}^{2} +\|\Psi[\p_{y}W]\|_{\tilde H^{1}}^{2} +\|\Psi[W]\|_{\tilde H^{1}}^{2} ).
\end{align*}
\end{lem}

\begin{lem} [$H^{2}$ stream function estimate II]
\label{lem:H2bulk2}
Let $A,f,g,k,W$ as in Lemma \ref{lem:H2boundary1}. Then, 
\begin{align*}
  |<A\p_{y}W, \frac{if}{k} \Phi^{(1)} + \frac{if'}{k} \Phi>| \lesssim  \frac{1}{|k|}\|f\|_{W^{2,\infty}} (\|\Psi[A\p_{y}W]\|_{\tilde H^{1}}^{2} +\|\Psi[\p_{y}W]\|_{\tilde H^{1}}^{2} +\|\Psi[W]\|_{\tilde H^{1}}^{2} ).
\end{align*}
\end{lem}

\begin{thm}[$H^{2}$ stability for the finite periodic channel]
\label{thm:H2stability}
Let $f,g,W,k$ as in Lemma \ref{lem:H2boundary1} and let $A(t)$ be defined as in Theorem \ref{thm:H1short}, i.e.
  let $A(t)$  be a diagonal weight:
  \begin{align}
    \begin{split}
    A(t)&:e^{iny} \mapsto A_{n}(t)e^{iny}, \\
    A_{n}(t)&= \exp \left( -\int^{t}_{0} <\frac{n}{k}-\tau>^{-2} + <\tau>^{-2\gamma} <\frac{n}{k}-\tau>^{-2\beta} d\tau\right),
    \end{split}
  \end{align}
  where $\beta,\gamma<\frac{1}{2}$ and $2\gamma+2\beta>1$.
  Further suppose that
  \begin{align*}
    \|f\|_{W^{3,\infty}} L 
  \end{align*}
  is sufficiently small.
Then, for any $\omega_{0} \in H^{2}([0,1])$ with vanishing Dirichlet data, $\omega_{0}|_{y=0,1}=0$, 
\begin{align*}
E_{2}(t):= \langle A(t)W, W\rangle+\langle A(t)\p_{y}W,\p_{y}W \rangle+\langle A(t)\p_{y}^{2}W , \p_{y}^{2} W\rangle
\end{align*}
satisfies
\begin{align*}
  \|W(t)\|_{H^{2}} \lesssim E_{2}(t) \lesssim E_{2}(0)\lesssim \|\omega_{0}\|_{H^{2}}.
\end{align*}
\end{thm}

\begin{proof}[Proof of Theorem \ref{thm:H2stability}]
The control of
\begin{align*}
  \langle A(t)W, W\rangle+\langle A(t)\p_{y}W,\p_{y}W \rangle
\end{align*}
has been established in Theorem \ref{thm:H1short}.
\\

Differentiating $\langle A(t)\p_{y}^{2}W , \p_{y}^{2} W\rangle$ in time, we have to control
\begin{align*}
  \langle A\p_{y}^{2}W, \frac{if}{k}\Phi^{(2)}+ \frac{2f'}{ik}\Phi^{(1)}+  \frac{f''}{ik} \Phi \rangle + \langle A\p_{y}^{2}W, \frac{if}{k}H^{(2)}+ \frac{2f'}{ik}H^{(1)} \rangle .
\end{align*}
As $\Phi^{(2)}, \Phi^{(1)}$ and $\Phi$ have zero boundary values, we integrate
\begin{align*}
 (-1+(\frac{\p_{y}}{k}-it)^{2})\Psi[A \p_{y}^{2}W] &=  A\p_{y}^{2}W , \\
  \Psi[A \p_{y}^{2}W]_{y=0,1}&=0,
\end{align*}
by parts and bound by the $\tilde H^{1}$ norm:
\begin{align*}
  \| \Psi[A\p_{y}^{2}W]\|_{\tilde H^{1}} \frac{\|f\|_{W^{3,\infty}}}{k} \left(\|\Phi^{(2)}\|_{\tilde H^{1}}+\|\Phi^{(1)}\|_{\tilde H^{1}}+\|\Phi\|_{\tilde H^{1}}\right).
\end{align*}
Lemmata \ref{lem:H2bulk1} and \ref{lem:H2bulk2} provide control by 
\begin{align}
\label{eq:bulkH2}
  \frac{1}{|k|}\|f\|_{W^{3,\infty}} (\|\Psi[A\p_{y}^{2}W]\|_{\tilde H^{1}}^{2} + \|\Psi[A\p_{y}W]\|_{\tilde H^{1}}^{2} +\|\Psi[\p_{y}W]\|_{\tilde H^{1}}^{2} +\|\Psi[W]\|_{\tilde H^{1}}^{2}).
\end{align}
Supposing that 
\begin{align}
\label{eq:smallnessH2}
  \sup \frac{1}{|k|}\|f\|_{W^{3,\infty}}
\end{align}
is sufficiently small, and using Lemma \ref{lem:CC2}, \eqref{eq:bulkH2} can be absorbed by 
\begin{align}
\label{eq:H2absorb}
  \langle W, \dot A W\rangle+\langle \p_{y}W, \dot A \p_{y}W \rangle+\langle \p_{y}^{2}W , \dot A \p_{y}^{2} W\rangle .
\end{align}

Using Lemmata \ref{lem:H2boundary1} and \ref{lem:H2boundary2} and supposing again that \eqref{eq:smallnessH2} is sufficiently small, the boundary contributions
\begin{align*}
  \langle A\p_{y}^{2}W, \frac{if}{k}H^{(2)}+ \frac{2f'}{ik}H^{(1)} \rangle , 
\end{align*}
can be partially absorbed in \eqref{eq:H2absorb}, with the remaining terms estimated by 
\begin{align}
  <t>^{-2}\|\omega_{0}\|_{H^{1}}^{2} + \|f\|_{W^{2,\infty}}|\frac{1}{k}| <t>^{-2(1-\gamma)} \|\omega_{0}\|_{H^{2}}^{2} +  \log^{2}(t)<t>^{-2(1-\gamma)}\|\omega_{0}\|_{H^{2}}^{2}.
\end{align}

We thus obtain that $E_{2}(t)$ satisfies
\begin{align*}
  \dt E_{2}(t) & \leq \frac{d}{dt} \left( \langle A(t)W, W\rangle+\langle A(t)\p_{y}W,\p_{y}W \rangle+\langle A(t)\p_{y}^{2}W , \p_{y}^{2} W\rangle \right) \\
 & \lesssim (<t>^{-2} +\log^{2}(t)<t>^{-2(1-\gamma)})\|\omega_{0}\|_{H^{2}}^{2}.
\end{align*}
As $0<\gamma<\frac{1}{2}$, this is integrable and thus yields the result.
\end{proof}

It remains to prove the Lemmata \ref{lem:H2boundary1}-\ref{lem:H2bulk2}.

\begin{proof}[Proof of Lemma \ref{lem:H2boundary1}]
We recall from the proof of Lemma \ref{lem:H11} that $H^{(1)}$ is explicitly given by
\begin{align*}
  H^{(1)} = \p_{y}\Phi(0) u_{1} + \p_{y}\Phi(1) u_{2}. 
\end{align*}
By the triangle inequality, we thus estimate by 
\begin{align*}
  \|H^{(1)}\|_{\tilde H^{1}} \lesssim |\p_{y}\Phi(t,0)| \|u_{1}(t)\|_{\tilde H^{1}}+ |\p_{y}\Phi(t,1)| \|u_{1}(t)\|_{\tilde H^{1}}. 
\end{align*}
We further recall that the homogeneous solutions $u_{1}, u_{2}$ are of the form
\begin{align*}
  a(t) e^{kG(y)+ikty} + b(t)e^{kG(y)+ikty},
\end{align*}
where $a(t),b(t)$ are chosen to satisfy the boundary conditions, (\ref{eq:boundaryvaluesuj}).
Hence, for any time $t$
\begin{align*}
  \|u_{j}(t)\|_{\tilde H^{1}} \leq |a(t)| \|e^{kG(y)+ikty}\|_{\tilde{H}^{1}} + |b(t)| \|e^{-kG(y)+ikty}\|_{\tilde{H}^{1}}
= |a(t)|\|e^{kG(y)}\|_{H^{1}} + |b(t)| \|e^{-kG(y)}\|_{H^{1}}.
\end{align*}
Therefore, by direct computation of the coefficients $a,b$ (analogously to Lemma \ref{lem:Exp1}),
\begin{align*}
  \|u_{j}(t)\|_{\tilde H^{1}}<C<\infty,
\end{align*}
uniformly in time.

Thus, $H^{(1)}$ satisfies
\begin{align*}
  \|H^{(1)}(t)\|_{\tilde H^{1}}  \lesssim |\p_{y}\Phi(t,0)| + |\p_{y}\Phi(t,1)|.
\end{align*}
Recalling the explicit characterization of $\p_{y}\Phi|_{y=0,1}$, \eqref{eq:400},
\begin{align*}
  \p_{y}\Phi|_{y=0}&= \frac{k}{g^{2}(0)} \langle W, u_{1} \rangle, \\
\p_{y}\Phi|_{y=1}&=- \frac{k}{g^{2}(1)} \langle W, u_{2} \rangle,
\end{align*}
and the subsequent estimate, \eqref{eq:401},
\begin{align*}
 \left| k \langle W, u_{j}  \rangle  \right| \lesssim <t>^{-1}(\|\omega_{0}\|_{H^{1}}+ |\langle \p_{y}W , u_{1}\rangle|+|\langle \p_{y}W , u_{2}\rangle| ),
\end{align*}
we further control 
\begin{align}
  \|H^{(1)}\|_{\tilde H^{1}} \lesssim <t>^{-1} \|W\|_{H^{1}},
\end{align}
which yields the first result.
\\

In order to estimate 
\begin{align*}
    \langle A\p_{y}^{2}W, \frac{if'}{k} H^{(1)}\rangle  &= \p_{y}\Psi(0) \langle A\p_{y}^{2}W , u_{1}\rangle + \p_{y}\Psi(1) \langle A\p_{y}^{2}W , u_{2}\rangle \\
&\lesssim <t>^{-1} | \langle A\p_{y}^{2}W , u_{j}\rangle |,
\end{align*}
we proceed as in Lemma \ref{lem:H11} and expand $\langle A\p_{y}^{2}W, u_{j} \rangle$ in our basis.
Thus, we obtain:
\begin{align*}
  | \langle A\p_{y}^{2}W , u_{j}\rangle | \lesssim_{\beta} \|(A \p_{y}^{2}W)_{n} <\frac{n}{k}-t>^{-\beta}\|_{l^{2}} ,
\end{align*}
for $0<\beta<\frac{1}{2}$.
Hence, by Young's inequality:
\begin{align*}
 \langle A\p_{y}^{2}W, \frac{if'}{k} H^{(1)}\rangle  & \lesssim  <t>^{-1}\|W\|_{H^{1}} \|(A \p_{y}^{2}W)_{n} <\frac{n}{k}-t>^{-\beta}\|_{l^{2}}  \\
& \lesssim <t>^{-2(1-\gamma)}\|W\|_{H^{1}}^{2} + \sum_{n} <t>^{-2\gamma} <\frac{n}{k}-t>^{-2 \beta}|(A \p_{y}^{2}W)_{n}|^{2}.
\end{align*}

\end{proof}

\begin{proof}[Proof of Lemma \ref{lem:H2boundary2}]
  We follow the same strategy as in the proof of Lemma \ref{lem:H11} and Lemma \ref{lem:H2boundary1} and explicitly
  compute
  \begin{align}
    \label{eq:3}
    \begin{split}
    H^{(2)}&= \p_{y}^{2}\Phi(0) u_{1} + \p_{y}^{2}\Phi(1) u_{2}, \\
\langle A\p_{y}^{2}W, \frac{if}{k} H^{(2)}\rangle &=\p_{y}^{2}\Phi(0) \langle A\p_{y}^{2}W, \frac{if}{k} u_{1} \rangle 
+ \p_{y}^{2}\Phi(1) \langle A\p_{y}^{2}W, \frac{if}{k} u_{2} \rangle \\
&\lesssim  |\p_{y}^{2}\Phi|_{y=0,1}| \left\|<\frac{n}{k}-t>^{-\beta}(A \p_{y}^{2}W)_{n} \right\|_{l^{2}_{n}} .
    \end{split}
  \end{align}
It hence remains to estimate $\p_{y}^{2}\Phi|_{y=0,1}$.
We thus expand the equation for the stream function $\Phi$ in the linearized Euler equations, \eqref{eq:303},
\begin{align*}
  (-1+(g(\frac{\p_{y}}{k}-it))^{2})\Phi = W,
\end{align*}
and obtain
\begin{align*}
  -\Phi + g^{2} k^{-2} \p_{y}^{2}\Phi + k^{-1} gg'(\frac{\p_{y}}{k}-it)\Phi -g^{2}it\frac{\p_{y}}{k} \Phi + g^{2}t^{2} \Phi = W .
\end{align*}
Thus, using that $\Phi$ and $W$ vanish at the boundary, $\p_{y}^{2}\Phi|_{y=0,1}$ satisfies
\begin{align*}
  g^{2} \p_{y}^{2}\Phi|_{y=0,1} = (-gg'+iktg^{2})\p_{y}\Phi|_{y=0,1}.
\end{align*}
Dividing by $g^{2}$ (which we required to be bounded away from zero), we may thus solve for $\p_{y}^{2} \Phi_{y=0,1}$:
\begin{align*}
  \p_{y}^{2}\Phi|_{y=0,1} =\frac{-gg'+iktg^{2}}{g^{2}}\p_{y}\Phi|_{y=0,\psi}= \mathcal{O}(kt)\p_{y}\Phi|_{y=0,1}.
\end{align*}
Again recalling the explicit characterization of $\p_{y}\Phi|_{y=0,1}$, \eqref{eq:400},
\begin{align*}
  \p_{y}\Phi|_{y=0}&= \frac{k}{g^{2}(0)} \langle W, u_{1} \rangle, \\
\p_{y}\Phi|_{y=1}&=- \frac{k}{g^{2}(1)} \langle W, u_{2} \rangle,
\end{align*}
from the proof of Lemma \ref{lem:H11}, we further compute
\begin{align*}
  \mathcal{O}(kt)\p_{y}\Phi|_{y=0,1} &\lesssim k^{2}t \langle W, u_{j} \rangle \lesssim k \langle \p_{y}W, u_{j} \rangle + ku_{j}W|_{y=0}^{1} \\
&\lesssim  <t>^{-1}\langle \p_{y}^{2}W, u_{j}\rangle + <t>^{-1} \p_{y}W u_{j}|_{y=0}^{1}+ ku_{j}W|_{y=0}^{1} \\
&=  <t>^{-1}\langle \p_{y}^{2}W, u_{j}\rangle + <t>^{-1} \p_{y}W u_{j}|_{y=0}^{1},
\end{align*}
where we again used that $W|_{y=0,1}= \omega_{0}|_{y=0,1} \equiv 0$.
The first term can again be estimated by
\begin{align*}
  <t>^{-1}\|<\frac{n}{k}-t>^{-\beta}(A \p_{y}^{2}W)_{n} \|_{l^{2}_{n}}
\end{align*}
and thus yields a contribution of the desired form.

To estimate the second term,
\begin{align}
\label{eq:420}
  <t>^{-1} \p_{y}W u_{j}|_{y=0}^{1},
\end{align}
 we restrict the evolution equation for $\p_{y}W$, (\ref{eq:H1}), to the boundary and obtain
  \begin{align*}
  \dt \p_{y}W|_{y=0,1}= \frac{f'}{ik}\Phi|_{y=0,1} + \frac{f}{ik}\p_{y}\Phi_{y=0,1} = \frac{f}{ik}\p_{y}\Phi \lesssim \left|\langle W, u_{j} \rangle \right|.
\end{align*}
Controlling the right-hand-side by $\mathcal{O}(t^{-1})\|W\|_{H^{1}}$ and using the $H^{1}$ stability result, Theorem \ref{thm:H1short}, we thus obtain a logarithmic control 
\begin{align*}
\left| \p_{y}W|_{y=0,1} \right|  = \mathcal{O} (\log (t)) \|\omega_{0}\|_{H^{1}}. 
\end{align*}
Hence, \eqref{eq:420} can be bounded by
\begin{align*}
  <t>^{-1} \p_{y}W u_{j}|_{y=0}^{1} \lesssim \log(t) <t>^{-1}\|\omega_{0}\|_{H^{1}}.
\end{align*}
Using these estimates, we may further estimate equation \eqref{eq:3} by
\begin{align*}
 \langle A\p_{y}^{2}W, \frac{if}{k} H^{(2)}\rangle \lesssim & (<t>^{-1}\|<\frac{n}{k}-t>^{-\beta}(A \p_{y}^{2}W)_{n} \|_{l^{2}_{n}} + <t>^{-1}\log(t)\|\omega_{0}\|_{H^{1}})  \\
& \cdot\|<\frac{n}{k}-t>^{-\beta}(A \p_{y}^{2}W)_{n} \|_{l^{2}_{n}} \\
\lesssim & \log^{2}(t)<t>^{-2(1-\gamma)} \|\omega_{0}\|_{H^{1}}^{2} + \sum_{n}<t>^{-2\gamma}<\frac{n}{k}-t>^{-2\beta}|(A \p_{y}^{2}W)_{n}|^{2} \\
& + \sum_{n}<t>^{-1}<\frac{n}{k}-t>^{-2\beta}|(A \p_{y}^{2}W)_{n}|^{2}. 
\end{align*}
\end{proof}

\begin{proof}[Proof of Lemma \ref{lem:H2bulk1}]
  We introduce 
\begin{align*}
  (-1+(\frac{\p_{y}}{k}-it)^{2})\Psi[A \p_{y}^{2}W] &= A \p_{y}^{2}W ,\\
  \Psi[A \p_{y}^{2}W]_{y=0,1}&=0,
\end{align*}
and use the vanishing boundary values of $ \Phi^{(2)}$ to integrate by parts and obtain
\begin{align*}
|\langle A\p_{y}W, \frac{if}{k} \Phi^{(2)} \rangle |  \lesssim  k^{-1}\|f\|_{W^{1,\infty}} (\|\Psi[A\p_{y}^{2}W]\|_{\tilde H^{1}}^{2} +\|\Phi^{(2)}\|_{\tilde H^{1}}^{2} ).
\end{align*}
It thus remains to control $\|\Phi^{(2)}\|_{\tilde H^{1}}^{2}$.
Testing
\begin{align*}
 (-1+(g(\frac{\p_{y}}{k}-it))^{2})\Phi^{(2)} &= \p_{y}^{2}W + [(g(\frac{\p_{y}}{k}-it))^{2}, \p_{y}^{2}] \Phi ,\\
\Phi^{(2)}_{y=0,1}&=0,
\end{align*} 
with $-\frac{1}{g}\Phi^{(2)}$, we estimate 
\begin{align*}
  \|\Phi^{(2)}\|_{\tilde H^{1}}^{2} \lesssim &- \langle(-1+(g(\frac{\p_{y}}{k}-it))^{2})\Phi^{(2)}, \frac{1}{g}\Phi^{(2)}  \rangle \\
\lesssim & \|\Psi[\p_{y}^{2} W]\|_{\tilde H^{1}}\|\Phi^{(2)}\|_{\tilde H^{1}} + \langle [(g(\frac{\p_{y}}{k}-it))^{2}, \p_{y}^{2}] \Phi ,  \frac{1}{g}\Phi^{(2)}  \rangle \\
\lesssim &\|\Phi^{(2)}\|_{\tilde H^{1}}  ( \|\Psi[\p_{y}^{2} W]\|_{\tilde H^{1}} + \|\p_{y}\Phi\|_{\tilde H^{1}} + \|\Phi\|_{\tilde H^{1}}). 
\end{align*}
Using the triangle inequality 
\begin{align*}
  \|\p_{y}\Phi\|_{\tilde H^{1}} \lesssim \|\Phi^{(1)}\|_{\tilde H^{1}}+ \|H^{(1)}\|_{\tilde H^{1}},
\end{align*}
Lemma \ref{lem:H2boundary1} and Lemma \ref{lem:H12} then provide the desired control.
\end{proof}

\begin{proof}[Proof of Lemma \ref{lem:H2bulk2}]
  We introduce 
\begin{align*}
  (-1+(\frac{\p_{y}}{k}-it)^{2})\Psi[A \p_{y}^{2}W] &= \p_{y}^{2}W ,\\
  \Psi[A \p_{y}^{2}W]_{y=0,1}&=0,
\end{align*}
and use the vanishing boundary values of $ \Phi^{(1)}$ and $\Phi$ to integrate by parts, to obtain
\begin{align*}
|<A\p_{y}W, \frac{if}{k} \Phi^{(1)} + \frac{if'}{k} \Phi>| \lesssim  |k^{-1}|\|f\|_{W^{2,\infty}} (\|\Psi[A\p_{y}W]\|_{\tilde H^{1}}^{2} +\|\Phi^{(1)}\|_{\tilde H^{1}}^{2} +\|\Phi[W]\|_{\tilde H^{1}}^{2} ).
\end{align*}
Lemma \ref{lem:H12} then provides the desired control.
\end{proof}

With these stability results, we now have the desired control on $\|W\|_{H^{2}}$ and hence, as we discuss in the following section, can prove linear inviscid damping with the optimal algebraic decay rates for a large class of strictly monotone shear flows 
in a finite periodic channel.

\section{Linear inviscid damping, scattering and consistency}
\label{sec:consistency-outlook}

In this section, we combine the results of the previous sections and thus close our strategy to prove linear inviscid damping for monotone shear flows:
\begin{figure*}[h!]
  \centering
  \includegraphics[page=1]{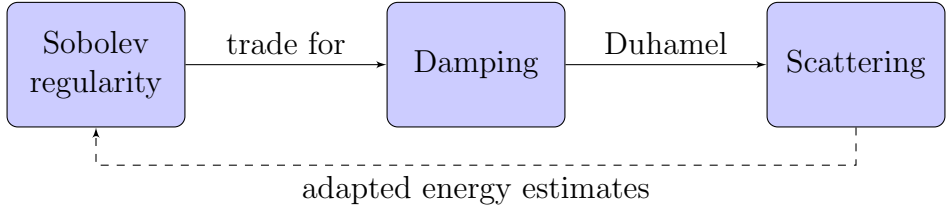}
  \caption{The stability results of Sections \ref{sec:asympt-stab-y} and \ref{sec:asympt-stab-with} allow us prove linear inviscid damping and scattering.}
\end{figure*}

\begin{thm}[Linear inviscid damping for the infinite periodic channel and finite periodic channel]
\label{thm:summary}
  Let $\omega_{0} \in L^{2}_{x}H^{2}_{y}$ with  $\langle \omega_{0} \rangle_{x} \equiv 0$ and let $W$ solve 
  \begin{align}
    \label{eq:genLE}
    \begin{split}
    \dt W &= f \p_{x} \Phi , \\
    (\p_{x}^{2}+(g(\p_{y}-t\p_{x}))^{2}) \Phi &= W,
    \end{split}
  \end{align}
  either on the infinite periodic channel, $\T_{L}\times \R$, or finite periodic channel, $\T_{L}\times [0,1]$. 
  Suppose that there exists $c>0$ such that
  \begin{align*}
    c < g <c^{-1}, \\
    \frac{1}{g},f \in W^{3,\infty},
  \end{align*}
  and that 
  \begin{align*}
     L \left\| \frac{f}{k} \right\|_{W^{3,\infty}} 
  \end{align*}
  is sufficiently small.
  In the case of a finite periodic channel, additionally assume that 
  \begin{align*}
\omega_{0}(x,0)\equiv 0 \equiv \omega_{0}(x,1).
  \end{align*}

  Then there exists a function $W_{\infty} \in L^{2}_{x}H^{2}_{y}$ such that 
  \begin{align*}
    \tag{Stability}
    \|W\|_{L^{2}_{x}H^{2}_{y}} &\lesssim \|\omega_{0}\|_{L^{2}_{x}H^{2}_{y}} ,\\
    \tag{Damping}
    \|v- \langle v \rangle_{x}\|_{L^{2}} &= \mathcal{O}(t^{-1}), \\
    \|v_{2}\|_{L^{2}} &= \mathcal{O}(t^{-2}), \\
    \tag{Scattering}
    W(t) &\rightarrow_{L^{2}} W_{\infty},
  \end{align*}
  as $t \rightarrow \infty$.
\end{thm}

\begin{proof}
Let $\omega_{0} \in L^{2}_{x}H^{2}_{y}$ and $f,g$ be given. Then by the stability results for the infinite channel, Theorem \ref{thm:iterated}, and for the finite channel, Theorem \ref{thm:H2stability}, $W$ satisfies
\begin{align*}
  \|W\|_{L^{2}_{x}H^{2}_{y}} \lesssim \|\omega_{0}\|_{L^{2}_{x}H^{2}_{y}}.
\end{align*}
As the mean in $x$ is conserved, i.e.
\begin{align*}
\langle W(t) \rangle_{x} \equiv \langle \omega_{0} \rangle_{x} \equiv 0,  
\end{align*}
we may apply Poincar\'e's  theorem to deduce that
\begin{align*}
  \|W\|_{H^{-1}_{x}H^{2}_{y}} \lesssim \|W\|_{L^{2}_{x}H^{2}_{y}} \lesssim \|\omega_{0}\|_{L^{2}_{x}H^{2}_{y}}.
\end{align*}
The damping result, Theorem \ref{thm:lin-zeng}, of Section \ref{sec:damping} then implies decay of the velocity field with the optimal algebraic rates.
\\

Duhamel's formula in our scattering formulation is just integrating \eqref{eq:genLE} in time and leads to:
\begin{align}
  \label{eq:scatter}
  W(t,x,y)= \omega_{0}(x,U^{-1}(y)) + \int_{0}^{t} f(y) V_{2}(\tau,x,y)d\tau,
\end{align}
where 
\begin{align*}
  V_{2}(t,x,y)=\p_{x}\Phi(t,x,z)= v_{2}(t,x-ty,U^{-1}(y)).
\end{align*}
Hence, as the change of variables $(x,y) \mapsto (x-tU(y),y)$ is an isometry and $y \mapsto U(y)$ is bilipschitz,
\begin{align*}
\|f V_{2}\|_{L^{2}} \leq \|f\|_{L^{\infty}} \|V_{2}\|_{L^{2}} \lesssim \|f\|_{L^{\infty}} \|v_{2}\|_{L^{2}}=\mathcal{O}(t^{-2}).  
\end{align*}
Thus, the integral in \eqref{eq:scatter} is uniformly bounded in $L^{2}$ for all $t$ and the improper integral for $t \rightarrow \pm \infty$ exists as a limit in $L^{2}$.
Therefore, 
\begin{align*}
  W  \xrightarrow {L^{2}} W_{\pm \infty} := \omega_{0} + \lim_{t\rightarrow \pm \infty }\int^{t}_{0} f V_{2}(\tau) d\tau .
\end{align*}
As $\|W\|_{L^{2}_{x}H^{2}_{y}} \lesssim \|\omega_{0}\|_{L^{2}_{x}H^{2}_{y}} $ uniformly in time, weak compactness and  lower semi-continuity imply $W_{\pm \infty} \in L^{2}_{x}H^{2}_{y}$ and
\begin{align*}
  \|W_{\pm \infty}\|_{H^{2}_{y} L^{2}_{x}} \lesssim \|w_{0}\|_{L^{2}_{x}H^{2}_{y}}. 
\end{align*}
\end{proof}

\begin{cor}[$L^{2}$ scattering]
\label{cor:L2scatterin}
  Let $f,g,L$ be as in Theorem \ref{thm:summary} and let $\omega_{0} \in L^{2}$, then there exists $W_{\infty} \in L^{2}$ such that 
  \begin{align*}
    W(t) \xrightarrow {L^{2}} W_{\infty}, \text { as } t \rightarrow \infty.
  \end{align*}
\end{cor}

\begin{proof}
  Let $H^{2} \ni \omega_{0}^{j} \xrightarrow{L^{2}} \omega_{0}$ as $j \rightarrow \infty$. 
  Then by the previous theorem there exist $W_{\infty}^{j}$ such that 
  \begin{align*}
    W^{j}(t) \xrightarrow {L^{2}} W_{\infty}^{j}, \text{ as } t \rightarrow \infty.
  \end{align*}
  By the $L^{2}$ stability results, Theorem \ref{thm:L2} of Section \ref{sec:asympt-stab-y} and Theorem \ref{thm:L2finite} of Section \ref{sec:asympt-stab-with}, and letting $t$ tend to infinity, $W_{\infty}^{j}$ is a Cauchy sequence in $L^{2}$. Denoting the limit by $W_{\infty}$, a diagonal sequence argument yields $W(t) \xrightarrow {L^{2}} W_{\infty}$ as $t \rightarrow \infty$. 
\end{proof}

A natural question following these linear inviscid damping and scattering results is, of course, whether such behavior also persists under the non-linear evolution.
Bedrossian and Masmoudi, \cite{bedrossian2013inviscid}, answer this question positively in the case of Couette flow in an infinite periodic channel, where the perturbations are required to be small in Gevrey regularity to control nonlinear effects.

As a small step in the direction of similar results for monotone shear flows, we follow Bouchet and Morita, \cite{Euler_stability}, and answer the simpler question of consistency.
In the derivation of the linearized Euler equations
\begin{align}
  \begin{split}
  \dt \omega + U(y)\p_{x}\omega &= U'' \p_{x}\phi, \\
  \Delta \phi &= \omega,
  \end{split}
\end{align}
one neglects the nonlinearity:
\begin{align}
\label{eq:nonlinearity}
  v \cdot \nabla \omega= v_{1}\p_{x}\omega +v_{2}\p_{y}\omega .
\end{align}
For a consistency result, we show that the nonlinearity, \emph{when evolved with the linear dynamics}, is an integrable perturbation in the sense that 
\begin{align*}
  \sup_{T>0} \left\| \int_{0}^{T}v \cdot \nabla \omega dt \right\|_{L^{2}} < C<\infty .
\end{align*}
 
In view of Theorem \ref{thm:summary}, at first sight we would expect decay of \eqref{eq:nonlinearity} with a rate of only $\mathcal{O}(t^{-1})$, as 
\begin{align*}
\|v_{1}\|_{L^{2}}&=\mathcal{O}(t^{-1}),\\ \|v_{2}\|_{L^{2}}&=\mathcal{O}(t^{-2}), \\
\|\p_{y}\omega\|_{L^{2}}&= \|(\p_{y}-tU'\p_{x})W\|_{L^{2}} = \mathcal{O}(t). 
\end{align*}

However, there is some additional cancellation, which can be used.
In scattering coordinates $v \cdot \nabla \omega$ is given by 
\begin{align*}
    -(\p_{y} -tU' \p_{x}) \Phi \p_{x} W + \p_{x} \Phi (\p_{y}-tU' \p_{x})W 
= \nabla^{\bot}\Phi \cdot \nabla W.
\end{align*}

Combining the stability results on $\nabla W$  and the damping results on $\nabla^{\bot} \Phi$ of Sections \ref{sec:iter-arbitr-sobol} and \ref{sec:h1-case},
we obtain quadratic decay
\begin{align*}
  \|\nabla^{\bot} \Phi\|_{L^{2}} = \mathcal{O}(t^{-2}).
\end{align*}
and thus consistency.
\begin{lem}[Consistency]
  Let $W$ be a solution to the linearized 2D Euler equation, (\ref{eq:genLE}), on $\T \times \R$ with initial datum $\omega_{0}\in H^{3}_{x,y}(\T_{L} \times \R)$ .
  Suppose further that the assumptions of the Sobolev regularity result, Theorem \ref{thm:iterated}, for $j=3$, as well as of the damping result, Theorem \ref{thm:lin-zeng}, are satisfied.
Then
\begin{align*}
  \|\nabla^{\bot} \Phi \cdot \nabla W \|_{L^{2}} = \mathcal{O}(t^{-2}).
\end{align*}
In particular, 
\begin{align*}
  W(t)+ \int^{t} \nabla^{\bot} \Phi(\tau) \nabla W (\tau) d\tau
\end{align*}
is close to $W(t)$ in $L^{2}$  uniformly in time and there exist asymptotic profiles $W^{\pm \infty}_{\text{con}}$ such that
\begin{align*}
  W(t) + \int^{t} \nabla^{\bot} \Phi(\tau) \nabla W (\tau) d\tau \xrightarrow {L^{2}} W^{\pm \infty}_{\text{con}}, \text{ as } t\rightarrow \pm \infty .
\end{align*}
\end{lem}

\begin{proof}
  By Theorem \ref{thm:iterated}, $W$ satisfies
  \begin{align*}
    \|W(t)\|_{H^{3}_{x,y}} \lesssim \|\omega_{0}\|_{H^{3}_{x,y}}.
  \end{align*}
  Hence, by Theorem \ref{thm:lin-zeng},
  \begin{align*}
      \|\nabla^{\bot} \Phi\|_{L^{2}} = \mathcal{O}(t^{-2})\|W(t)\|_{H^{3}_{x,y}} =\mathcal{O}(t^{-2})\|\omega_{0}\|_{H^{3}_{x,y}}.
  \end{align*}
  Using the Sobolev embedding, 
  \begin{align*}
    \|\nabla W\|_{L^{\infty}_{x,y}} \lesssim  \|W(t)\|_{H^{3}_{x,y}} \lesssim \|\omega_{0}\|_{H^{3}_{x,y}}.
  \end{align*}
  An application of Hölder's inequality then gives the desired bound: 
  \begin{align*}
    \|\nabla^{\bot} \Phi \nabla W \|_{L^{2}} \leq \|\nabla^{\bot}\Phi\|_{L^{2}} \|\nabla W \|_{L^{\infty}}= \mathcal{O}(t^{-2})\|\omega_{0}\|_{H^{3}_{x,y}}^{2} .
  \end{align*}
\end{proof}

We remark that the regularity assumptions on $\omega_{0}$ are not sharp. As we are, however, only interested in the qualitative property of consistency, we assume sufficiently much regularity to use a Sobolev embedding.

  In the setting of a finite periodic channel $\T \times [0,1]$, we thus far have only established stability in $H^{2}$, which is not sufficient for integrable decay of $\|\nabla^{\bot}\Phi\|_{L^{2}}$.
Furthermore, in two dimensions $H^{2}_{x,y}$ regularity is critical for the Sobolev embedding.
Hence, control of $\|W\|_{H^{2}(\T_{L}\times [0,1])}$ only yields $\nabla_{x,y} W \in \text{BMO}(\T_{L}\times[0,1])$ instead of $L^{\infty}(\T_{L}\times[0,1])$.

A natural question is thus whether the stability result in a finite periodic channel can be improved to higher Sobolev spaces. As we sketch in Appendix \ref{sec:zero-boundary-values}, stability in $H^{3}$ is in general not possible.

In a follow-up article, we show that here the fractional Sobolev spaces $H^{\frac{3}{2}}_{y}$ and $H^{\frac{5}{2}}_{y}$ are critical, in the sense that stability holds in all sub-critical spaces and blow-up occurs in all super-critical spaces. Furthermore, we discuss the exact singularity formation at the boundary and the associated blow-up and its implications for the nonlinear dynamics, where high regularity is essential to control nonlinear effects.

\appendix

\label{sec:appendix}

\section{Bases and mapping properties}
\label{sec:test}
In this section, we elaborate on the role of boundary conditions, the choice of basis and the mapping properties of 
\begin{align*}
  W \mapsto& \langle W, \Psi \rangle , \\
  (k^{2}-(\p_{y}-ikt)^{2})\Psi &= W, \\
  \Psi|_{y=0,1}&=0 .
\end{align*}

In analogy to the whole space setting, a first natural approach is via a Fourier basis, which we used in Section \ref{sec:l2-stability-via-1}.
There, the coefficients of $\Psi$  have been computed in Lemma \ref{lem:Exp1}:
\begin{lem}
  Let $\Psi$ be as above, $n,m \in 2\pi \Z$, then
  \begin{align*}
    \langle \Psi[e^{iny}], e^{imy} \rangle = \frac{\delta_{nm}}{k^{2}+(n-kt)^{2}} + \frac{k}{(k^{2}+(m-kt)^{2})(k^{2}+(n-kt)^{2})}(a-b),
  \end{align*}
  where $a,b$ solve
  \begin{align*}
     \begin{pmatrix}
    e^{k+ikt} & e^{-k+ikt} \\ 1 & 1
  \end{pmatrix}
  \begin{pmatrix}
    a \\ b
  \end{pmatrix}
  =
  \begin{pmatrix}
    1 \\ 1
  \end{pmatrix}.
  \end{align*}
\end{lem}
\vspace{0.5cm}
This choice of basis has a distinct advantage in its simplicity and good decoupling multiplier structure.
In particular, we may easily prove Lemma \ref{lem:CC} using Cauchy-Schwarz.
We, however, see that we can not obtain a bounded map in $H^{s}, s\geq \frac{1}{2}$, in this way, as the decay is not fast enough and thus
\begin{align*}
  \frac{n^{s}}{k^{2}+(n-kt)^{2}} \not \in l^{2}.
\end{align*}
When trying to use Schur's test instead, one encounters the problem of slow decay as $n, m \rightarrow \infty$ at an even earlier stage of our proof:
\begin{align}
  \label{eq:Schurstuff}
  \begin{split}
  \sup_{m} \sum_{n}<\frac{n}{k}-t>^{\alpha}<\frac{m}{k}-t>^{\alpha}|\langle \Psi[e^{iny}], e^{imy} \rangle | \\
\leq 1 + k^{-3} \sum_{n} <\frac{n}{k}-t>^{\alpha-2} \lesssim_{\alpha} 1 + k^{-2}.
  \end{split}
\end{align}
Therefore, this approach does not even provide an $l^{2}$ estimate with optimal decay, but only a weaker variant of Lemma \ref{lem:CC} with $\alpha <1$.

Furthermore, testing against homogeneous solutions, we only obtain slow decay:
\begin{align*}
\langle e^{iny}, e^{\pm y+ity} \rangle &= \frac{1}{\pm 1 +i (t-n)} e^{\pm y +i(t-n)y}|_{y=0}^{1}= \mathcal{O}(<n-t>^{-1}), \\
  \langle \sin(ny), e^{\pm y+ity} \rangle &= -n \frac{1}{\pm 1+it} \langle \cos(ny), e^{\pm y+ity} \rangle = \mathcal{O}(n <t>^{-1}<n-t>^{-1}).
\end{align*}

Considering a $\sin$ basis instead, we may make use of vanishing boundary terms to obtain additional cancellations and better coefficients:
\begin{lem}
\label{lem:Sin}
  Let $n\in \pi \N$ and let $\Psi[\sin(ny)]$ be the solution of
  \begin{align*}
    (k^{2}-(\p_{y}-ikt)^{2})\Psi[\sin(ny)] &= \sin(ny), \\
  \Psi[\sin(ny)]|_{y=0,1}&=0.
  \end{align*}
 Then, for any $m \in \pi \N$,
  \begin{align*}
   &\langle \Psi[\sin(ny)], \sin(my) \rangle = \delta_{nm} (\frac{1}{k^{2}+(n-kt)^{2}} + \frac{1}{k^{2}+(n+kt)^{2}}) \\
  & \quad + d k\left( \frac{1}{k^{2}+(kt+n)^{2}}-  \frac{1}{k^{2}+(kt-n)^{2}} \right) \left( \frac{1}{k^{2}+(kt+m)^{2}}-  \frac{1}{k^{2}+(kt-m)^{2}} \right) \\
&\quad + i((-1)^{n+m}-1) nmkt \\
& \quad \cdot \frac{k^{4}t^{4}+2k^{4}t^{2}+2k^{4}-2k^{2}t^{2}(m^{2}+n^{2}) + 2k^{2}(m^{2}+n^{2}) + 2 m^{2}n^{2}}{(k^{2}+(kt+m)^{2})(k^{2}+(kt-m)^{2})(k^{2}+(kt+n)^{2})(k^{2}+(kt-n)^{2})(n^{2}-m^{2})},
  \end{align*}
where 
\begin{align*}
  d = -((-1)^{n+m}-1) +2 \frac{(-1)^{n+m}e^{-k}+e^{k}}{e^{k}-e^{-k}}+  2\frac{(-1)^{m}e^{ikt}-(-1)^{n}e^{-ikt}}{e^{k}-e^{-k}}. 
\end{align*}
\end{lem}

Before proving this result, let us comment on some of the implications and the relation to the results of Section \ref{sec:asympt-stab-with}.
\begin{itemize}
\item While these coefficients are much less simple than for a Fourier basis, they asymptotically decay with rates $n^{-3}m^{-3}$. Hence, an argument via Schur's test as in \eqref{eq:Schurstuff} does not have to require $\alpha<1$.
Furthermore, the rapid decay suggests that the mapping
\begin{align}
\label{eq:basisest}
  \begin{split}
   W &\mapsto  \Psi, \\
   L^{2} &\rightarrow L^{2}
  \end{split}
\end{align}
can be extended to a bounded mapping on the fractional Sobolev spaces: 
\begin{align*}
  \sum_{n} n^{2s} |W_{n}|^{2},
\end{align*}
for $s>0$ not too large.

\item Using that $n,m,kt \geq 0$, one may roughly bound 
\begin{align*}
  \frac{k^{2}t^{2}}{\sqrt{k^{2}+(n+kt)^{2}}\sqrt{k^{2}+(m+kt)^{2}}} \leq 1,
\end{align*}
and thus trade the additional decay for the convenience of a uniform bound.
While this is far from optimal, it reduces estimates to the ones for the Fourier basis.
\item In Section \ref{sec:asympt-stab-with} we use a different approach and consider boundary terms separately.
That is, we decompose $\p_{y}\Phi$ into a function, $\Phi^{(1)}$, with zero Dirichlet conditions
\begin{align*}
  (k^{2}-(g(\p_{y}-ikt))^{2})\Phi^{(1)} &= \p_{y}W + [(g(\p_{y}-ikt))^{2},\p_{y}]\Phi, \\
  \Psi^{(1)}|_{y=0,1}&=0,
\end{align*}
and a \emph{homogeneous correction}
\begin{align*}
  (k^{2}-(\p_{y}-ikt)^{2})H^{(1)} &= 0, \\
  H^{(1)}|_{y=0,1}&=\p_{y}\Phi|_{y=0,1}.
\end{align*}
The estimate of 
\begin{align*}
  \p_{y}W + [(g(\p_{y}-ikt))^{2},\p_{y}]\Phi \mapsto \Phi^{(1)}
\end{align*}
is then similar to the estimate of $\Phi$ in terms of $W$.
In order to control $H^{(1)}$, we make additional use of the dynamics and study the evolution of 
\begin{align*}
 \p_{y}\Phi|_{y=0,1}.
\end{align*}
\item We further note, that by our choice of basis, for $\frac{1}{2}<s<1$,  $\p_{y}W \in H^{s}$ would also imply that $\p_{y}W|_{y=0,1}$ vanishes for all times. However, $\p_{y}W|_{y=0,1}$ is not conserved by the linearized Euler equations.     
\end{itemize}

\begin{proof}[Proof of Lemma \ref{lem:Sin}]
The streamfunction $\Psi[\sin(ny)]$ is given by
\begin{align*}
 \Psi[\sin(ny)]&= \left(\frac{1}{k^{2}+(n-kt)^{2}}+ \frac{1}{k^{2}+(n+kt)^{2}}\right) \sin(ny) \\
& \quad + i\left(\frac{1}{k^{2}+(n-kt)^{2}}- \frac{1}{k^{2}+(n+kt)^{2}}\right)\left( \cos(ny) + a e^{ky+ikty}+ b e^{-ky+ikty}\right),
\end{align*}
 where $a,b$ solve
\begin{align*}
  \begin{pmatrix}
    e^{k+ikt} & e^{-k+ikt} \\ 1 & 1
  \end{pmatrix}
  \begin{pmatrix}
    a \\ b
  \end{pmatrix}
  &= -
  \begin{pmatrix}
    (-1)^{n} \\ 1
  \end{pmatrix}.
\end{align*}
Integrating against another basis function, $\sin(my)$, we obtain:
\begin{align*}
  &\langle \Psi[\sin(ny)],\sin(my) \rangle \\
&= \delta_{nm} \left(\frac{1}{k^{2}+(n-kt)^{2}} + \frac{1}{k^{2}+(n+kt)^{2}}\right) 
+ i\left(\frac{1}{k^{2}+(n-kt)^{2}} - \frac{1}{k^{2}+(n+kt)^{2}}\right)\\ 
 &\quad \cdot\left(\frac{m((-1)^{n+m}-1)}{n^{2}-m^{2}} + \frac{1}{2i}\left(\frac{1}{k+i(kt+m)}- \frac{1}{k+i(kt-m)}\right) ((-1)^{m} e^{k + ikt}-1) a  \right. \\
&\quad +\left.  \frac{1}{2i}\left(\frac{1}{-k+i(kt+m)}- \frac{1}{-k+i(kt-m)}\right) ((-1)^{m} e^{-k + ikt}-1) b \right).
\end{align*}

As the $\delta_{nm}$ term is already of the desired form, in the following we consider only the remaining terms.
Using the equation for $a,b$, we obtain
\begin{align}
  \label{eq:longcoefficients}
  \begin{split}
&\left(\frac{1}{k^{2}+(n-kt)^{2}} + \frac{1}{k^{2}+(n+kt)^{2}}\right) \Big( \frac{im((-1)^{n+m}-1)}{n^{2}-m^{2}} \\ 
 &- \frac{1}{2}\left(\frac{1}{k+i(kt+m)}- \frac{1}{k+i(kt-m)}+\frac{1}{-k+i(kt+m)}-\frac{1}{-k+i(kt-m)} \right)((-1)^{n+m}-1)\\
& + \frac{1}{2}\left(\frac{1}{k+i(kt+m)}- \frac{1}{k+i(kt-m)}-\frac{1}{-k+i(kt+m)}+ \frac{1}{-k+i(kt-m)} \right) d \Big),
  \end{split}
\end{align}
where 
\begin{align*}
  d &=((-1)^{m} e^{k + ikt}-1) a - ((-1)^{m} e^{-k + ikt}-1) b  \\
&= ((-1)^{n+m} -1) -2 ((-1)^{m} e^{-k + ikt}-1) b \\
&=-((-1)^{n+m}-1) +2 \frac{((-1)^{m}e^{-k +ikt} -1) ((-1)^{n}-e^{k+ikt})}{e^{k+ikt}-e^{-k+ikt}} \\
&= -((-1)^{n+m}-1) +2 \frac{(-1)^{n+m}e^{-k}+e^{k}}{e^{k}-e^{-k}}+  2\frac{(-1)^{m}e^{ikt}-(-1)^{n}e^{-ikt}}{e^{k}-e^{-k}} .
\end{align*}
We, in particular, note that
\begin{align*}
  d(t,k,n,m)= \overline{d(t,k,m,n)}= d(-t,k,m,n),
\end{align*}
and that $d$ is uniformly bounded if $k$ is bounded away from zero.
Furthermore, consider $k$ large and $n+m$ even, then in \eqref{eq:longcoefficients} only the contribution involving $d$ is present and 
\begin{align*}
  (-1)^{n+m}-1 &=0, \\
  d = 2 \frac{e^{k}}{e^{k}-e^{-k}} + \mathcal{O}(e^{-k})&= 2+\mathcal{O}(e^{-k})>1.
\end{align*}
The factor in front of $d$ and $((-1)^{n+m}-1)$, in \eqref{eq:longcoefficients}, are given by the real and imaginary part of 
\begin{align*}
  &\frac{1}{k+i(kt+m)}- \frac{1}{k+i(kt-m)} 
=\frac{k-i(kt+m)}{k^{2}+(kt+m)^{2}} - \frac{k-i(kt-m)}{k^{2}+(kt-m)^{2}} \\
&= (k-ikt)(\frac{1}{k^{2}+(kt+m)^{2}}-\frac{1}{k^{2}+(kt-m)^{2}}) - im (\frac{1}{k^{2}+(kt+m)^{2}}+\frac{1}{k^{2}+(kt-m)^{2}}). \\
&= k(\frac{1}{k^{2}+(kt+m)^{2}}-\frac{1}{k^{2}+(kt-m)^{2}}) + i \frac{kt (4ktm) -m (2k^{2}+2k^{2}t^{2} +2 m^{2})}{(k^{2}+(kt-m)^{2})(k^{2}+(kt-m)^{2})} \\
&= k(\frac{1}{k^{2}+(kt+m)^{2}}-\frac{1}{k^{2}+(kt-m)^{2}}) - i \frac{m (2k^{2}-2k^{2}t^{2} +2 m^{2})}{(k^{2}+(kt-m)^{2})(k^{2}+(kt-m)^{2})} .
\end{align*}

The coefficients $c_{nm}(t,k)$ are hence explicitly given by:
\begin{align*}
  & \delta_{nm} (\frac{1}{k^{2}+(n-kt)^{2}} + \frac{1}{k^{2}+(n+kt)^{2}}) \\
  &\quad + d k\left( \frac{1}{k^{2}+(kt+n)^{2}}-  \frac{1}{k^{2}+(kt-n)^{2}} \right) \left( \frac{1}{k^{2}+(kt+m)^{2}}-  \frac{1}{k^{2}+(kt-m)^{2}} \right) \\
&\quad + i((-1)^{n+m}-1) \\
&\quad \cdot \left( \frac{1}{k^{2}+(kt+n)^{2}}-  \frac{1}{k^{2}+(kt-n)^{2}} \right) 
\left( -  \frac{m (2k^{2}-2k^{2}t^{2} +2 m^{2})}{(k^{2}+(kt+m)^{2})(k^{2}+(kt-m)^{2})} + \frac{m}{n^{2}-m^{2}}  \right) \\
  &=\delta_{nm} (\frac{1}{k^{2}+(n-kt)^{2}} + \frac{1}{k^{2}+(n+kt)^{2}}) \\
  &\quad + d k\left( \frac{1}{k^{2}+(kt+n)^{2}}-  \frac{1}{k^{2}+(kt-n)^{2}} \right) \left( \frac{1}{k^{2}+(kt+m)^{2}}-  \frac{1}{k^{2}+(kt-m)^{2}} \right) \\
&\quad + i((-1)^{n+m}-1) \\
&\quad \cdot nmkt \frac{k^{4}t^{4}+2k^{4}t^{2}+2k^{4}-2k^{2}t^{2}(m^{2}+n^{2}) + 2k^{2}(m^{2}+n^{2}) + 2 m^{2}n^{2}}{(k^{2}+(kt+m)^{2})(k^{2}+(kt-m)^{2})(k^{2}+(kt+n)^{2})(k^{2}+(kt-n)^{2})(n^{2}-m^{2})}.
\end{align*}

\end{proof}

\section{Stability and boundary conditions}
\label{sec:zero-boundary-values}

In Section \ref{sec:h2-case}, we required $\omega_{0}$ to satisfy zero Dirichlet conditions to establish decay of $\p_{y}^{2}\Phi$ and $\p_{y}^{2} \Psi$.
In this section, we show that this conditions is necessary, both for the explicit example
\begin{align*}
 \omega_{0}(x,y)= 2i\sin(x), \quad (x,y) \in \T_{\pi} \times [0,1],
\end{align*}
as well as for general functions with $\hat \omega_{0}(k) \in H^{2}([0,1])$.
For simplicity, we first consider linearized Couette flow. 

\begin{lem}
  Consider the linearized Couette flow in scattering formulation
  \begin{align*}
    \dt W &=0 , \\
    (-k^{2} +(\p_{y}-ikt)^{2}) \Psi &=W, 
  \end{align*}
  with initial datum $\hat \omega_{0}(k,y)=\delta_{1}(k) - \delta_{-1}(k)$.
  Then there exists a sequence $t_{n} \rightarrow \infty$, such that $e^{-kt_{n}y}\p_{y}^{2}\Psi(t_{n},k,y)$ converges to a non-trivial limit in $L^{2}_{y}$.
\end{lem}

\begin{proof}
By symmetry it suffices to consider $k=1$. The stream function $\Psi$ is then given by
\begin{align*}
  \frac{1}{1+t^{2}} (-1 + a(t) e^{y+ity} + b(t) e^{-y+ity}), 
\end{align*}
where $a,b$ solve
\begin{align*}
  \begin{pmatrix}
    e^{1 + it} & e^{-1 +it} \\ 1 & 1
  \end{pmatrix}
  \begin{pmatrix}
    a \\ b
  \end{pmatrix}
= 
\begin{pmatrix}
   1\\ 1
\end{pmatrix}.
\end{align*}
Differentiating twice, we obtain
\begin{align*}
\Xi:= e^{-ity} \p_{y}^{2} \Psi = \frac{(1+it)^{2}}{1+t^{2}}a(t) e^{y}  +\frac{(-1+it)^{2}}{1+t^{2}}b(t) e^{-y}.  
\end{align*}
As $a(t),b(t)$ depend on $t$ only via $e^{it}$, for any $c \in \R, m \in 2\pi \Z$
\begin{align*}
  a(c)=a(c+m), \\
  b(c)=b(c+m).
\end{align*}
We may thus, for example, consider sequences $t_{1,n} \in 2\pi \Z$ and $t_{2,n} \in 2\pi \Z +\pi$ tending to $\pm \infty$.
Along these sequences $a,b$ are constant and non-trivial, while 
\begin{align*}
  \frac{(\pm 1+it)^{2}}{1+t^{2}} \rightarrow -1.
\end{align*}
Therefore, 
\begin{align*}
  \Xi(t_{n}) \rightarrow -ae^{y} -be^{-y} \neq 0,
\end{align*}
which yields the desired result.
  
\end{proof}

A similar result also holds for generic $\omega_{0}$:
\begin{lem}
  Consider the linearized Couette flow in scattering formulation
  \begin{align*}
    \dt W &=0 , \\
    (-k^{2} +(\p_{y}-ikt)^{2}) \Psi &=W. 
  \end{align*}
  Let further $\hat \omega_{0}(k,\cdot) \in H^{2}([0,1])$ and suppose that for some $k\neq 0$, $\hat{\omega}_{0}(k,\cdot)|_{y=0,1}$ is non-trivial.
  Then $e^{-ity} \p_{y}^{2}\Psi$ does not converge to zero in $L^{2}$ as $t \rightarrow \pm \infty$.
\end{lem}

\begin{proof}
Splitting $\p_{y}^{2}\Psi = \Psi^{(2)}+ H^{(2)}$ as in Section \ref{sec:h2-case}, we obtain
\begin{align*}
  \|\Psi^{(2)}\|_{L^{2}}^{2} \leq \|\Psi^{(2)}\|_{\tilde H^{1}}^{2} =\langle \Psi^{(2)}, \p_{y}^{2}W \rangle \leq \sum_{n} <\frac{n}{k}-t>^{-2} |(\p_{y}^{2}W)_{n}|^{2}.
\end{align*}
Using a similar argument as in the proof of Theorem \ref{thm:lin-zeng}, one can show that $\|\Psi^{(2)}\|_{L^{2}}\rightarrow 0$.
Here we use that, for Couette flow, $W$ is preserved in time and hence an $L^{2}$ estimate suffices.
In the more general case, for this argument one would either need some additional control of the $L^{2}$ integrability, e.g. 
\begin{align*}
  \lim_{N\rightarrow \infty}\sup_{t>0} \sum_{|n|\geq N}|(\p_{y}^{2}W)_{n}|^{2} =0,
\end{align*}
or control in a fractional Sobolev space.
\\

It thus remains to consider
\begin{align*}
e^{-ikty} H^{(2)} = \p_{y}^{2}\Psi(0) e^{-ikty} u_{1} + \p_{y}^{2}\Psi(1)e^{-ikty}u_{2}.  
\end{align*}
For convenience of notation, we again set $k=1$.

Restricting to sequences $t_{n} \in 2\pi\N$, $e^{-ity}u_{1}$ and $e^{-ity}u_{2}$ do not depend on $t$ and are linearly independent.
It thus suffices to show that $\p_{y}^{2}\Psi(0)$ and $\p_{y}^{2}\Psi(1)$ cannot both converge to zero unless $\omega_{0}$ satisfies zero Dirichlet conditions.

Solving
\begin{align*}
  (-1+(\p_{y}-it)^{2})\Psi = \hat \omega_{0},
\end{align*}
for $\p_{y}^{2}\Psi$, we obtain
\begin{align*}
  \p_{y}^{2}\Psi|_{y=0,1}= \hat \omega_{0}|_{y=0,1} +2it\p_{y}\Psi |_{y=0,1}.
\end{align*}
Testing the above equation with $u_{j}$, yields
\begin{align*}
  \p_{y}\Psi|_{y=0,1} &= \langle \hat \omega_{0}, u_{j} \rangle = \langle \hat \omega_{0}, e^{ity} (ae^{y}+be^{-y}) \rangle\\
&= \frac{1}{it} \hat \omega_{0}|_{y=0,1} - \frac{1}{it} \int e^{ity} \p_{y} \left(\hat \omega_{0} (a e^{y}+b e^{-y}) \right) \rangle \\
&= \frac{1}{it} \hat \omega_{0}|_{y=0,1} + \| \hat \omega_{0}\|_{H^{2}} \mathcal{O}(t^{-2}).
\end{align*}
Here we used that $e^{it_{n}y}|_{y=0,1}=1$ for our sequence of $t_{n}$.
Therefore, 
\begin{align*}
  \p_{y}^{2}\Psi|_{y=0,1} = 3 \hat \omega_{0}|_{y=0,1} + \mathcal{O}(t_{n}^{-1}) \not \rightarrow 0,
\end{align*}
which concludes the proof.  
\end{proof}

The above method of proof also allows to derive an instability result for flows other than Couette flow:
\begin{thm}
\label{thm:boundaryblowup}
Let $f,g \in W^{2,\infty}([0,1])$ such that 
  \begin{align*}
    0<c<g<c^{-1}<\infty,
  \end{align*}
  and suppose that $f|_{y=0,1}\neq 0$.
  Then for any $\omega_{0} \in H^{2}$ with $\omega_{0}|_{y=0,1}\neq 0$, the solution to the linearized Euler equations, \eqref{eq:303},
  \begin{align*}
    \begin{split}
      \dt W &= \frac{if(y)}{k}\Phi, \\
      \left(-1+\left(g(y)\left(\frac{\p_{y}}{k} -it\right)\right)^{2}\right) \Phi &= W,\\
      \Phi|_{y=0,1}&=0, \\
      (t,k,y) &\in  \R \times L(\Z\setminus \{0\}) \times [0,1],
    \end{split}
  \end{align*}
  satisfies
  \begin{align*}
    \sup_{t\geq 0} \|W(t)\|_{H^{2}}= \infty.
  \end{align*}
\end{thm}

\begin{proof}[Proof of Theorem \ref{thm:boundaryblowup}]
Assume for the sake of contradiction that 
\begin{align}
\label{eq:assumptionboundary}
  \sup_{t\geq 0} \|W(t)\|_{H^{2}}=C<\infty.
\end{align}

We claim that then $\p_{y}\Phi|_{y=0,1}$ satisfies 
\begin{align}
\label{eq:claimforboundary}
  \p_{y}\Phi|_{y=0,1} =\frac{1}{ikt}\frac{k}{g^{2}}\omega_{0}|_{y=0,1} + C\mathcal{O}(t^{-2}). 
\end{align}
Considering the evolution of $\p_{y}W$ restricted to the boundary,
\begin{align*}
  \dt \p_{y}W|_{y=0,1}= \frac{if}{k}\p_{y}\Phi|_{y=0,1},
\end{align*}
and integrating in time, we thus obtain that, as $T \rightarrow \infty$ 
\begin{align*}
 |\p_{y}W(T)|_{y=0,1}| \gtrsim \left| \int_{1}^{T}\frac{if}{k}\frac{1}{ikt}\frac{k}{g^{2}}\omega_{0}|_{y=0,1} dt \right| + \mathcal{O}(1) \gtrsim \log(T).
\end{align*}
Hence, 
\begin{align*}
  \sup_{t\geq 0}\|\p_{y}W\|_{C^{0}} = \infty,
\end{align*}
which by the Sobolev embedding contradicts our assumption \eqref{eq:assumptionboundary} and concludes the proof.

It remains to show the claim, \eqref{eq:claimforboundary}.
In equation \eqref{eq:400} of Section \ref{sec:h1-case}, we have shown that $\p_{y}\Phi|_{y=0,1}$ can be computed as
\begin{align*}
  \p_{y}\Phi|_{y=0}&= \frac{k}{g^{2}(0)} \langle W,u_{1} \rangle, \\
  \p_{y}\Phi|_{y=1}&= -\frac{k}{g^{2}(1)} \langle W,u_{2} \rangle,
\end{align*}
where $u_{j}$ are solutions of 
\begin{align*}
  (-1+(g(y)(\frac{\p_{y}}{k} -it))^{2}) u_{j}=0,
\end{align*}
with boundary conditions, \eqref{eq:boundaryvaluesuj},
\begin{align*}
    \begin{split}
  u_{1}(0)=u_{2}(1)&=1, \\
  u_{1}(1)= u_{2}(0)&=0.
  \end{split}
\end{align*}
It hence suffices to show that
\begin{align*}
  \langle W,u_{1} \rangle &= -\frac{1}{ikt} \omega_{0}|_{y=0}+C\mathcal{O}(t^{-2}), \\
  \langle W,u_{2} \rangle &= \frac{1}{ikt} \omega_{0}|_{y=1}+C\mathcal{O}(t^{-2}).
\end{align*}
For this purpose we note that $u_{j}(t,y)$ satisfy 
\begin{align*}
  u_{1}(t,y)&=e^{ikty}u_{1}(0,y), \\
  u_{2}(t,y)&=e^{ikt(y-1)}u_{2}(0,y).
\end{align*}
Integrating $e^{ikty}$ by parts twice, we hence obtain
\begin{align*}
  \langle W,u_{1}(t,y) \rangle &= \frac{1}{ikt}W u_{1}(t,y)|_{y=0}^{1}+ \frac{1}{k^{2}t^{2}}e^{ikty}\p_{y}(Wu_{1}(0,\cdot))|_{y=0}^{1}- \frac{1}{k^{2}t^{2}} \langle e^{ikty}, \p_{y}^{2} (Wu_{1}(0,\cdot)) \rangle\\
&= \frac{1}{ikt}W u_{1}(t,y)|_{y=0}^{1} + C\mathcal{O}(t^{-2}),
\end{align*}
where we used a trace estimate to control the second boundary term.
\end{proof}

Using the same approach, one can obtain similar results for higher Sobolev norms involving boundary values of higher derivatives.
However, for non-Couette flow the boundary values of higher derivatives are not conserved by the evolution and therefore conditions of the form
\begin{align*}
  \p_{y}^{n}W|_{y=0,1}\equiv 0
\end{align*}
are in general never satisfied for $n\geq 1$. Instead, one would have to derive necessary and sufficient conditions under which $\p_{y}^{n}W|_{y=0,1} \rightarrow 0$ as $t \rightarrow \pm \infty$.

In a follow-up article, we further study the singularity formation and show that for general monotone shear flows (i.e. $U''$ not vanishing at the boundary) the critical fractional Sobolev space is given by $H_{y}^{\frac{3}{2}}$.
That is, we prove stability in all sub-critical fractional Sobolev spaces $H^{s}_{y}(\T), s<\frac{3}{2}$, and blow-up in all super-critical Sobolev space.
Restricting to perturbations with zero Dirichlet data, $\omega_{0}|_{y=0,1}=0$, we show that the critical space is given by $H_{y}^{\frac{5}{2}}$.

\bibliographystyle{alpha} \bibliography{citations}
\end{document}